\documentclass[final]{imanum}
\usepackage{amsfonts,amsmath,amscd,amsthm,amssymb}
\usepackage{mathtools}
\usepackage[mathscr]{euscript} 
\usepackage{graphics}
\usepackage{wrapfig}
\usepackage{lscape}
\usepackage{rotating}
\usepackage{epsfig}
\usepackage{color}
\usepackage[notcite,notref]{showkeys}
\usepackage{multirow}
\usepackage{framed}
\usepackage{algorithm}
\usepackage{algpseudocode}
\usepackage{thmtools}

\jno{drnxxx}
\received{03 March 2016}
\revised{14 April 2017}

\newcommand{\bld}[1]{\boldsymbol{#1}}
\newcommand{\und}[1]{\underline{\bld #1}}

\newcommand{\dO}{{\partial \Omega}}

\newcommand{\Oh}{{\mathcal{T}_h}}
\newcommand{\dOh}{{\partial \Oh}}
\newcommand{\dK}{{\partial K}}

\newcommand{\Eh}{{\mathcal{F}_h}}

\newcommand{\divs}{\mathop{\nabla\cdot}}
\newcommand{\divv}{\mathop{\bld{\nabla\cdot}}}

\newcommand{\gradv}{\mathop{\bld{\nabla}}}

\newcommand{\grads}{\mathop{\nabla}}
\newcommand{\pol}[2]{{P}_{#1}(#2)} 
\newcommand{\bpol}[2]{\bld P_{#1}(#2)}

\newcommand{\upol}[2]{{{{P}}}_{#1}(#2;\mathbb{S})}

\newcommand{\es}{\bld{\varepsilon}_{{\ensuremath{\scriptscriptstyle{\und{\sigma}}}}}}
\newcommand{\eos}{\bld{e}_{{\ensuremath{\scriptscriptstyle{\und{\sigma}}}}}}
\newcommand{\ep}{\bld{\varepsilon}_{{\ensuremath{\scriptscriptstyle{p}}}}}
\newcommand{\esd}{\bld{\varepsilon}^D_{{\ensuremath{\scriptscriptstyle{\und{\sigma}}}}}}
\newcommand{\eshat}{\widehat{\bld{\varepsilon}}_{{\ensuremath{\scriptscriptstyle{\und{\sigma}}}}}}
\newcommand{\eu}{\bld \varepsilon_{{\ensuremath{\scriptscriptstyle{\bld u}}}}}
\newcommand{\euhat}{\widehat{\bld \varepsilon}_{{\ensuremath{\scriptscriptstyle{{{\bld u}}}}}}}

\newcommand{\ds}{\bld{\delta}_{{\ensuremath{\scriptscriptstyle{\und{\sigma}}}}}}
\newcommand{\du}{\bld \delta_{{\ensuremath{\scriptscriptstyle{\bld u}}}}}

\newcommand{\Vh}{\boldsymbol{V}_h}
\newcommand{\Sh}{\und{\Sigma}_h}
\newcommand{\Mh}{\bld{M}_h}

\newcommand{\uhat}{\widehat{\bld u}_h}
\newcommand{\shatn}{\widehat{\und \sigma}_h\n}

\newcommand{\n}{\boldsymbol{n}}
\newcommand{\Pis}{\und{\Pi}_\Sigma}

\newcommand{\Piv}{\boldsymbol{\varPi}_V}

\newcommand{\Pim}{\bld{P}_M}

\newcommand{\bint}[2]{\langle #1\,,\,#2 \rangle_{\partial{\Oh}}}

\newcommand{\bintK}[2]{\langle #1\,,\,#2 \rangle_{\partial{K}}}

\newcommand{\inp}[2]{(#1\,, \, #2)_{\Oh}}
\usepackage{color}
\definecolor{black}{rgb}{0,0,0}

\definecolor{red}{rgb}{0,0,0}
\definecolor{redd}{rgb}{0,0,0}
\newcommand\red[1]{\textcolor{redd}{#1}}
\definecolor{blue}{rgb}{0,0,0}

\newcommand{\mm}{$\bld M$}
\newcommand{\eg}{\bld{\mathrm{e}}}
\newcommand{\vt}{\mathrm{v}}
\newcommand{\ppol}{\bld P}
\newcommand{\pcol}{{{P}}}
\newcommand{\bppol}{\underline{\bld P}^s}
\newcommand{\qqol}{\bld Q}
\newcommand{\bqqol}{\underline{\bld Q}^s}
\newcommand{\tr}{\mathrm{tr}}
\newcommand{\vv}{\und{\tau}}
\newcommand{\W}{{\bld V}}
\newcommand{\WWW}{{\bld V_{\!g}}}
\newcommand{\w}{\boldsymbol{v}}
\newcommand{\nn}{\boldsymbol n}
\newcommand{\VV}{{\und{\Sigma}}}
\newcommand{\VVV}{{\und{{{\Sigma}}}_g}}
\newcommand{\Vtilde}{\widetilde\VV}
\newcommand{\Wtilde}{\widetilde{\boldsymbol{V}}}

\newcommand{\Vperp}{\widetilde\VV{}^\perp}
\newcommand{\Wperp}{\widetilde{\boldsymbol{V}}^\perp}
\newcommand{\vvtilde}{\widetilde{\und{\sigma}}}
\newcommand{\wtilde}{\widetilde{\boldsymbol{v}}}
\newcommand{\dvm}{\delta\und{{\Sigma}}}

\definecolor{gfu}{rgb}{0,0,0}
\newcommand\gfu[1]{\textcolor{gfu}{#1}}
\definecolor{ggfu}{rgb}{1,0,0.1}
\newcommand\ggfu[1]{\textcolor{ggfu}{#1}}
\begin{document}

\title{Devising superconvergent HDG methods  with symmetric approximate stresses for linear elasticity by M-decompositions.}

\shorttitle{Linear elasticity with strongly symmetric stresses}

\author{%
{\sc
Bernardo Cockburn} \\[2pt]
School of Mathematics, University of Minnesota,
                Minneapolis, MN 55455, USA
\\
{\tt cockburn@math.umn.edu}\\
\vspace{0.2cm}
{\sc and}\\
\vspace{0.2cm}
{\sc
Guosheng Fu$^*$} \\[2pt]
Division of Applied Mathematics, Brown University, Providence, %
  RI 02912, USA\\
$^*$Corresponding author: {\tt guosheng\_fu@brown.edu}
}
\shortauthorlist{B. Cockburn and G. Fu}

\maketitle

\begin{center}
\ggfu{This is the extended version of a paper submitted to IMA Journal of Numerical Analysis on March 16, 2016; accepted on April 14, 2017.}
\end{center}

\begin{abstract}
{
We propose a new tool, which we call \mm-decompositions, for devising superconvergent
hybridizable discontinuous Galerkin (HDG) methods and hybridized-mixed methods for linear elasticity with strongly symmetric approximate stresses on unstructured polygonal/polyhedral meshes. 

We show that for an HDG method, when its local approximation space
admits an \mm-decomposition, optimal convergence of the approximate
stress and superconvergence of an element-by-element postprocessing of the
displacement field are obtained.  The resulting methods are
locking-free. 

Moreover, we explicitly construct approximation spaces that admit \mm-decompositions on general
polygonal elements. We display numerical results on triangular meshes validating our theoretical findings.
}

{
hybridizable; discontinuous Galerkin; superconvergence; linear elasticity; strong symmetry.
}
\end{abstract}

\section{Introduction}
\label{sec:introduction}
We present a technique to systematically construct superconvergent hybridizable discontinuous Galerkin (HDG) and  
mixed methods with {\em strongly} symmetric approximate stresses on unstructured polygonal/polyhedral meshes for linear  
elasticity. By a superconvergent method,  we mean, roughly speaking, a method that provides an approximate 
displacement converging to certain projection of the exact displacement faster than it converges to the exact 
displacement itself. It is  then possible to obtain, by means of an elementwise and parallelizable computation,  a new 
approximate displacement converging faster than the original one. 
This property
was uncovered  back in 1985 in \cite{ArnoldBrezzi85} in the framework of mixed methods for diffusion problems and has 
been extended to various mixed methods for several elliptic problems; see \cite{BoffiBrezziFortin13}.

This paper is part of a series in which we devise superconvergent HDG and mixed methods for steady-state problems.
Indeed, superconvergent HDG (and mixed) methods for second-order diffusion were considered in
\cite{CockburnQiuShi12,CockburnQiuShiCurved12}, superconvergent HDG methods based on the velocity
gradient-velocity-pressure formulation of the Stokes equations of 
incompressible flow in \cite{CockburnShiHDGStokes13}, 
and superconvergent HDG methods with {\em weakly} symmetric approximate stresses for the equations of 
linear elasticity  in \cite{CockburnShiHDGElas13}. 

In \cite{CockburnFuSayas16}, we  refined the
 work on second-order diffusion carried out in  \cite{CockburnQiuShi12,CockburnQiuShiCurved12}
 and showed that, by using the concept of an $M$-decomposition (for the divergence and gradient operators),
 it is possible to systematically  find HDG and mixed methods \red{which} superconverge on unstructured meshes made of
 polygonal/polyhedral elements of arbitrary shapes. The actual construction of such 
 $M$-decompositions for arbitrary polygonal elements was carried out in \cite{CockburnFu17a}, 
 and for tetrahedra, prisms, pyramids, and hexahedra in three-space dimensions in \cite{CockburnFu17b}.
 The extension of these results to the heat equation \cite{ChabaudCockburn12} 
 and wave equation \cite{CockburnQuennevilleHDGWave14} are straightforward. The extension to the 
 \red{velocity gradient-velocity-pressure} formulation of HDG and mixed methods for the Stokes equations is more delicate and was carried out
 in \cite{CockburnFuQiu16}.
 
  Here, we continue  this effort and consider the more challenging task of devising superconvergent 
 HDG and mixed methods with strongly symmetric approximate stresses by developing 
 a general theory of \mm-decompositions for the vector divergence and symmetric gradient operators. We also provide a 
practical construction for polygonal elements. We do this in the framework of the following model problem:
\begin{subequations}
\begin{alignat}{2} 
\label{original equation-1}
 \mathcal{A}\und \sigma - \und \epsilon(\bld u )&= \und 0 \qquad && \text{in $\Omega$,}\\
\label{original equation-2}
-\divv \und \sigma &= \bld f && \text{in $\Omega$,}\\
\label{boundary condition}
\bld u &= \bld g && \text{on $\partial \Omega$,}
\end{alignat}
\end{subequations}
where $\Omega \subset \mathbb{R}^n$ ($n=2,3$) is a bounded polyhedral domain, 
$\dO$ is the Dirichlet boundary. Here $\bld u=\{u_i\}_{i=1}^n$ and $\und\sigma = \{\sigma_{ij}\}_{i,j=1}^n$ 
represent the displacement vector and the Cauchy stress tensor,
respectively. The functions $\bld f=\{f_i\}_{i=1}^n$ and $\bld g=\{g_i\}_{i=1}^n$
represent the body force vector and the prescribed displacement on $\dO$,
respectively. As usual, 
$\und\epsilon (\cdot) := \frac{1}{2}\left(\gradv(\cdot)+\gradv^t(\cdot)\right)$ is the 
symmetric gradient (or strain) operator,
and $\mathcal{A} = \{\mathcal{A}_{ijkl}(x)\}_{i,j,k,l=1}^{n}$is the so-called
compliance tensor, which is bounded and positive definite. In the homogeneous,
isotropic case, it is given by 
\begin{equation}
\label{isotropic}
  \mathcal{A}{\und\sigma} = \frac{1}{2\mu}(\und\sigma-\frac{\lambda}{2\mu+n\lambda}\mathrm{tr}(\und\sigma)\und I),
\end{equation}
where $\und I$ is the second-order identity tensor and $\lambda,\mu$ are the  Lam\'e constants.

To better describe our results, let us begin by introducing the general form
of the methods we are going to consider. 
We denote by  $\Oh$ a conforming triangulation of $\Omega$ made of polygonal/polyhedral
elements $K$. We denote by $\Eh$ the set of faces $F$ of all the elements $K$ in
the triangulation $\Oh$, and by $\dOh$ the set of boundaries $\partial K$ of all
elements $K$ in $\Oh$. 
We set $h_K := \mathrm{diam}({K})$ and $h : =\min_{K\subset\Oh} h_K$.
To each element $K\in\Oh$, 
we associate (finite dimensional) spaces of symmetric-matrix-valued functions $\VV(K)$, and vector-valued functions 
$\W(K)$. We also consider a (finite dimensional) 
space of vector-valued functions $\bld M(F)$ associated to
each $F\in\Eh$. We assume that elements of the above spaces are regular enough
so that all traces belong to $\bld L^2(\dK)$.
\gfu{
To simplify \red{the} notation, we denote the normal trace of $\VV(K)$ on $\partial K$ \red{by}
\begin{subequations}
 \label{traces}
\begin{align}
\label{trace-1}
 \gamma\left( \VV(K)\right) := \{\bld\tau\bld n|_{\partial K}:\quad \bld\tau \in \VV(K)\},
\end{align}
and the trace of $\W(K)$ on $\partial K$ \red{by}
\begin{align}
\label{trace-2}
 \gamma\left(\W(K)\right) := \{\bld v|_{\partial K}:\quad \bld v \in \W(K)\}.
\end{align}
\end{subequations}
}

The methods we are interested in seek an approximation to $(\und \sigma, \bld u, \bld u|_{\Eh})$,
$(\und \sigma_h, \bld u_h, \uhat)$, in the finite element space $\Sh \times \Vh
 \times \Mh$, where
\begin{subequations}
\label{spaces}
\begin{alignat}{3}
\Sh:=&\;\{\und{\tau}\in {L}^2(\Oh; \mathbb{S}):&&\;\und{\tau}|_K\in\und\Sigma (K),&&\; \forall K\in\Oh\},
\\
\Vh:=&\;\{\w\in \bld{L}^2(\Oh): &&\;\w|_K\in \bld V(K),&&\; \forall K\in\Oh\},
\\
\Mh:=&\;\{\bld \mu\in \bld {L}^2(\Eh):&&\;\bld \mu|_F\in \bld M(F),&&\;\forall F\in\Eh\},
\end{alignat}
\end{subequations}
and determine it as the only solution of the following weak formulation:
\begin{subequations}
\label{weak formulation}
 \begin{alignat}{3}
  \label{weak formulation-1}
(\mathcal{A}\und\sigma_h,{\und\tau})_\Oh + (\bld u_h, \divv\und\tau)_\Oh-\bint{\uhat}{\und\tau\n} &\;=0, \\
\label{weak formulation-2}
(\und\sigma_h,\gradv \w)_\Oh - \bint{\shatn}{\w} &\;=(\bld f,\w)_\Oh, \\
\label{weak formulation-3}
\langle{\shatn},{\bld \mu}\rangle_{\partial\Oh\setminus\partial\Omega} &\;= 0,\\
\label{weak formulation-4}
 \langle \uhat,{\bld \mu}\rangle_{\partial\Omega} &\;=  \langle {\bld g},{\bld \mu}\rangle_{\partial\Omega},
 \end{alignat}
for all $(\und\tau, \boldsymbol{v}, \bld\mu) \in \Sh \times \Vh \times \Mh$.
Here we write 
$\inp{\eta}{\zeta} := \sum_{K \in \Oh} (\eta, \zeta)_K,$ 
where $(\eta,\zeta)_D$ denotes the integral of $\eta\zeta$ over the domain $D \subset \mathbb{R}^n$. We also write
$\bint{\eta}{\zeta}:= \sum_{K \in \Oh} \langle \eta \,,\,\zeta \rangle_{\partial K},$
where $\langle \eta \,,\,\zeta \rangle_{D}$ denotes the integral of $\eta \zeta$ over the domain $D \subset \mathbb{R}^{n-1}$ 
and $\partial \Oh := \{ \partial K: K \subset \Oh \}$. The definition of the method
is completed with the definition of the normal component of the numerical trace:
\begin{equation}
\label{trace-q}
 \shatn= \und\sigma_h \n -{\alpha(\bld u_h - \uhat)} \quad\text{ on }\quad \partial \Oh,
\end{equation}
\end{subequations}
where $\alpha: \bld L^2(\dK)\rightarrow \bld L^2(\dK)$ is a  suitably chosen {\it linear local stabilization} operator. 
By taking particular choices of the local spaces $\und\Sigma(K)$, $\bld V(K)$ and 
\[
\bld M(\partial K):=\{\bld\mu\in \bld L^2(\partial K): \;\bld\mu|_F\in \bld M(F)\mbox{ for all }F\in \mathcal{F}(K)\},
\]
and of the { linear local stabilization} operator ${\alpha}$, \red{all the} different HDG methods
are obtained; when we can set $\alpha= 0$, we obtain the hybridized version of a mixed
method. 

Our contributions are two.
The first is that we show that, if for every element $K\in \Oh$, 
the local spaces $\VV(K)$,
$\W(K)$ and $\bld M(\dK)$
satisfy certain inclusion conditions, it is
possible to define, in an element-by-element fashion,
an auxiliary projection 
\[\Pi_h = (\Pis,\Piv): H^1(K,\mathbb{S})\times \bld H^1(K)\rightarrow \und\Sigma(K)\times\bld V(K)\]
such that
\begin{alignat*}{1}
\|\und{\sigma} - \und{\sigma}_h\|_{\mathcal{A},\Oh} \leq&\; 2\,\|\und{\sigma} - \Pis\, \und{\sigma}\|_{\mathcal{A},\Oh},
\\
\|\Piv \bld u - \bld u_h\|_{\Oh} \leq&\; C\, h \|\und{\sigma} - \Pis \,\und{\sigma}\|_{\Oh},\\
\|\Pim \bld u - \widehat{\bld u}_h\|_{\dOh} \leq&\; C\, h^{1/2} \|\und{\sigma} - \Pis \,\und{\sigma}\|_{\Oh},
\end{alignat*}
where 
$\|\cdot\|_{\mathcal{A},\Oh}$ denotes the $\mathcal{A}$-weighted $L^2(\Oh)$-norm, i.e.,
$\|\und\tau\|_{\mathcal{A},\Oh}^2 = (\mathcal{A}\und\tau,\und\tau)_\Oh$, 
$\|\cdot\|_{D}$ denotes the $L^2(D)$-norm on a domain $D$, and $\bld P_M$ is the $L^2$-projection onto $\bld M_h$. 
Moreover, if the error
$\Piv \bld u-\bld u_h$ converges to zero {\em faster} than the error $\bld u-\bld u_h$, this 
{\em superconvergence} property can be advantageously exploited to locally 
obtain a more accurate approximation of the displacement $\bld u$;
see
\cite{ArnoldBrezzi85,BrezziDouglasMarini85,Stenberg88,GastaldiNochetto89,Stenberg91}
for applications to mixed methods 
and
\cite{CockburnQiuShi12,CockburnQiuShiCurved12,CockburnShiHDGStokes13,CockburnShiHDGElas13}
for applications to HDG methods
and the references therein. 

Our second contribution is a refinement of our previous work
\cite{CockburnQiuShi12,CockburnQiuShiCurved12,CockburnShiHDGStokes13,CockburnShiHDGElas13}
along the lines provided in \cite{CockburnFuSayas16}: 
We propose an algorithm that tells us how to systematically 
obtain the local spaces rendering any given HDG method superconvergent. 
Indeed, given any local spaces $\VVV(K)$, $\WWW(K)$ and $\bld M(\partial K)$, 
satisfying very simple inclusion properties, we can give a characterization of the space 
$\delta\und{\Sigma}(K)$ of the {\em smallest} dimension for which the space
\red{$(\VVV(K)\oplus\delta\und{\Sigma}(K))\,\times\WWW(K)$ admits an $\bld M(\partial K)$-decomposition and hence,
defines} a superconvergent HDG method. We also show how to obtain (the hybridized version of) two {\em sandwiching}
 mixed methods from those spaces. In particular, we show 
that the HDG method with 
local spaces $\VVV(K)\oplus\delta\und{\Sigma}(K)$,
$\divv\VVV(K)$ and $\bld M(\partial K)$ is actually a 
well-defined hybridizable mixed method. 
Applying this approach, as done in \cite{CockburnFu17a,CockburnFu17b}, 
we find new superconvergent HDG  and mixed methods on 
general polygonal elements in two-space dimensions.
These new spaces are closely related to those of the mixed methods proposed in  
\cite{ArnoldDouglasGupta84} for triangles and quadrilaterals, and in \cite{GuzmanNeilan14} for 
triangles. 

Let us make a couple of comments on the relevance of these findings. 
The first concerns the first HDG method for linear elasticity proposed in
\cite{SoonThesis08}; see also \cite{SoonCockburnStolarski09}. 
When using polynomial approximations of degree $k$ on triangular meshes, this method was experimentally 
shown to provide approximations of order $k+1$ for the stress and displacement for $k\ge0$, 
and superconvergence of order $k+2$ for the displacement. 
Recently, in \cite{FuCockburnStolarski14}, it was proven that 
the stress convergences with order $k+1/2$ and displacement with $k+1$ without
superconvergence; numerical experiments showed that the orders were \red{actually} sharp.
It turns out the spaces for that method do not admit \mm-decompositions.
By using the machinery  propose here, we show that, on triangular meshes, 
it is enough to add a small space $\delta\und{\Sigma}(K)$ (of dimension $2$ if $k=1$,
and of dimension $3$ if $k\ge 2$) to the space of approximate stresses $\und{\Sigma}(K)$ 
to obtain convergence of order $k+1$ for the stress and displacement, 
as well as superconvergence of order $k+2$ for the displacement for $k\ge1$.

The second comment concerns mixed methods (with symmetric approximate stresses) for linear elasticity.
A symmetric and conforming mixed method 
use finite element spaces (for the stress and displacement)
$\VV_h\times \W_h\subset H(\mathrm{div},\Omega; \mathbb{S})\times \bld L^2(\Omega)$.
Here $H(\mathrm{div},\Omega; \mathbb{S})$ denotes the space of $H(\mathrm{div})$-conforming, symmetric-tensorial fields.
It turns out to be quite difficult to design finite elements in $H(\mathrm{div},\Omega; \mathbb{S})$ 
which preserve both the symmetry and the $H(\mathrm{div})$-conformity.
The first successful discretization use the so-called composite elements, see 
the low-order composite elements in \cite{JohnsonMercier78} and the
high-order composite elements in \cite{ArnoldDouglasGupta84}, both defined on triangular and quadrilateral meshes; 
see also a low-order composite element on tetrahedral meshes in \cite{AinsworthRankin11}.
While these methods can be efficiently implemented via hybridization, leading to an symmetric-positive-definite
linear system, \red{their} main drawback is that their basis functions are usually
quite hard to construct, especially for the high-order methods in \cite{ArnoldDouglasGupta84}.

In the pioneering \red{work} \cite{ArnoldWinther02}, \red{the} first
discretization of $H(\mathrm{div})$-conforming symmetric-tensorial fields \red{was presented, for triangular meshes, which only uses
piecewise polynomial functions}. 
\red{Other piecewise-polynomial discretizations of $H(\mathrm{div};\mathrm{S})$ on
tetrahedral meshes were later obtained in \cite{AdamsCockburn05,ArnoldAwanouWinther08,HuZhang15}
and on rectangular meshes in \cite{ArnoldAwanou05,ChenYang10}}.
Since all of these polynomial discretizations use vertex degrees of freedom to define the basis functions, 
\red{the resulting methods cannot} be efficiently implemented via hybridization, 
and a saddle-point linear system needs to be solved. 
\red{Moreover, as argued in the last paragraph of Section 3 of
\cite{ArnoldWinther02}, vertex degrees of freedom cannot be avoided in the construction of  piecewise-polynomial 
 $H(\mathrm{div})$-conforming symmetric-tensorial fields.}

\red{For this reason, nonconforming} mixed methods 
\cite{ArnoldWinther03, Yi05, Yi06, HuShi07, ManHuShi09, 
Awanou09,GopalakrishnanGuzman11, ArnoldAwanouWinther14} 
that violate $H(\mathrm{div})$-conformity (but preserve symmetry) 
of the stress field offer an attractive alternative to the conforming methods.
These methods also use polynomial basis functions but 
do not use vertex degrees of freedom. As a consequence,
they can be efficiently implemented via hybridization. 
See also an interesting nonconforming mixed method 
\cite{PechsteinSchoberl11} that uses tangential-continuous displacement field, and normal-normal continuous 
symmetric stress field.

Recently, another high-order conforming discretization of $H(\mathrm{div},\Omega;\mathbb{S})$ 
without using vertex degrees of freedom
was introduced in \cite{GuzmanNeilan14} for triangular meshes.
The space enriches the symmetric-tensorial polynomial fields of degree $k\ge 2$  
with {\em three} rational basis functions on each triangle. The resulting mixed methods can be efficiently 
implemented via hybridization.

In this paper, we obtain 
high-order conforming discretizations of $H(\mathrm{div},\Omega;\mathbb{S})$ 
without vertex degrees of freedom on general polygonal meshes. 
On a general polygon, we enrich the 
symmetric-tensorial polynomial fields with certain number of composite/rational 
functions given by explicit formulas. Our spaces 
on triangular meshes is similar as those in \cite{GuzmanNeilan14}, 
while our spaces on rectangular meshes enrich the symmetric-tensorial polynomial fields 
with {\em four} rational functions and {\em two} exponential functions
along with a minimal number of polynomial functions.

To end, let us point out that there are other HDG methods that superconverge on meshes of arbitrary polyhedral elements have been recently introduced. A modification of the method \cite{SoonThesis08} which can achieve optimal convergence 
was introduced in \cite{QiuShenShi16}. The spaces, which do {\em not } admit \mm-decompositions,  are
$
 \VV(K)\times \W(K)\times \bld M(\dK) :=
 \pcol_{k}(K;\mathbb{S})\times \ppol_{k+1}(K)
 \times \ppol_k(\dK)
$
and the stabilization function is $\alpha(\bld u_h-\widehat{\bld u}_h):= \frac{1}{h}\left(
\bld P_M\,\bld u_h-\widehat{\bld u}_h\right)$\red{, where $\bld P_M$ denotes the $L^2$-projection into the space of traces $\bld{M}_h$}.
These methods were proven to achieve optimal convergence order of $k+1$ 
for stress and $k+2$ for displacement on general polygonal/polyhedral meshes
for $k\ge 1$.  Another method that can achieve this is the hybrid high-order (HHO) method introduced (in primal form) in
\cite{DiPietroErn15b}. Typically, our spaces $\VV(K)\times \W(K)$ are bigger, but 
the globally-coupled system for these and our methods has the same size and sparsity structure, as it only depends on the
space of traces $\bld M(\dK)$. 
On the other hand, our mixed methods provide $H(\mathrm{div},\Omega;\mathbb{S})$-conforming approximate 
stresses with optimally convergent divergence. 

The rest of the paper is organized as follows. 
In Section \ref{sec:error}, we present new a priori
error estimates for HDG methods with spaces admitting \mm-decompositions. 
In Section \ref{sec:m1}, 
we present a characterization of \mm-decompositions and present three ways to construct them.
In Section \ref{sec:main}, we give 
{\em ready-for-implementation} spaces admitting \mm-decompositions on general polygonal elements.
In Section \ref{sec:proof}, we prove the main result of Section \ref{sec:main}, 
Theorem \ref{thm:polygon}. In Section \ref{sec:n},
we present numerical results on triangular meshes to validate our theoretical results.
Finally, in Section \ref{sec:c} we end with some concluding remarks.

\section{The error estimates}
\label{sec:error}
In this section, we introduce the notion of an \mm-decomposition and then present 
our main results on the 
a priori error estimates of 
the HDG methods defined with spaces admitting \mm-decompositions. 
The proofs of the error estimates are very similar to those 
for the diffusion case \cite{CockburnFuSayas16}.

\subsection{Definition of an \mm-decomposition}
To simplify the notation, when there is no possible confusion,
we do not indicate the domain on which the functions of a given space are
defined. For example, instead of  $\VV(K)$, we simply write $\VV$.

The notion of an \mm-decomposition relates the trace of the normal component of the space of approximate 
stress $\VV\subset \{\und\tau\in H(\mathrm{div},K;\mathbb{S}):\;\und\tau\,\bld n|_\dK\in\bld L^2(\dK)\}$ and 
the trace of the space of approximate displacement
$\W\subset \bld H^1(K)$ with the space of approximate traces  $\bld M\subset \bld L^2(\partial K)$.  To define it, 
we need to consider the combined trace operator
\begin{alignat*}{6}
\tr : & \VV\times \W &\quad\longrightarrow\quad 
& \bld L^2(\partial K)\\
& (\vv,\w) &\quad\longmapsto\quad & 
(\vv\nn +\w)|_{\partial K}
\end{alignat*}
where $\nn:\partial K\to \mathbb R^n$ is the unit outward pointing normal field on $\partial K$.

\begin{definition} [The \mm-decomposition] 
\label{def:m}
We say that $\VV\times \W$ admits an \mm-decomposition when
\begin{itemize}
\item[{\rm(a)}] $\tr (\VV\times \W)\subset \bld M$,
\end{itemize}
and there exist subspaces $\Vtilde\subset\VV$ and $\Wtilde\subset \W$ satisfying
\begin{itemize}
\item[{\rm(b)}] $\und\epsilon (\W)\subset \Vtilde$,  $\divv\VV\subset \Wtilde,$
\item[{\rm(c)}] $\tr: \Vtilde^\perp\times \Wtilde^\perp\rightarrow \bld M$ is an isomorphism.
\end{itemize}
Here $\Vtilde^\perp$ and $\widetilde{\W}^\perp$ are the $L^2$-orthogonal complements of 
$\Vtilde$ in $\VV$, and of $\widetilde{\W}$ in $\W$, respectively.
\end{definition}

\subsection{The HDG projection} 
Next, we show an immediate consequence of the fact that
the space $\VV\times \W$ admits an \mm-decomposition, namely,
the  existence of an auxiliary HDG-projection 
which is the key to our \red{error} analysis.

To state the result, we need to introduce the quantities related to the stabilization operator $\alpha$:
\[
a_{\Wtilde^\perp}:=\begin{cases}
\inf_{\bld\mu\in \gamma\Wtilde^\perp\setminus{\{\bld 0\}}} \langle \alpha(\bld \mu),\bld \mu\rangle_{\partial K}/
\|\bld \mu\|^2_{\partial K}&\quad\mbox{ if } \Wtilde^\perp\neq\{\bld 0\},\\
\infty&\quad\mbox{ if } \Wtilde^\perp=\{\bld 0\},
\end{cases}
\]
and
\[
\|\alpha\|:=\sup_{\bld\lambda, \bld\mu\in \bld M\setminus\{\bld 0\}}
\langle\alpha(\bld\lambda), \bld\mu\rangle_{\partial K}/(\|\bld\lambda\|_{\partial K}\|\bld\mu\|_{\partial K}).
\]

Throughout this section, we assume that, on each element $K$, 
the space $\VV\times \W$ admits an \mm-decomposition, and that
the stabilization operator $\alpha: \bld L^2(\dK)\rightarrow
\bld L^2(\dK)$ 
satisfies the following three properties:
\begin{subequations}
\begin{itemize}
\item[{\rm(S1) }]
$\langle \alpha(\bld \lambda),\bld \mu\rangle_{\partial K}=\langle \alpha(\bld\mu),\bld\lambda\rangle_{\partial K}$
for all $\bld\lambda, \bld\mu\in \bld L^2(\dK)$.
\item [{\rm(S2) }]
$\langle \alpha(\bld\mu),\bld\mu\rangle_{\partial K}\ge 0$ for all $\bld\mu\in \bld L^2(\dK)$. 
\item [{\rm(S3) }]
$a_{\Wtilde^\perp}>0$.
\end{itemize}
\end{subequations}
Properties (S1) and (S2) mean that $\alpha$ is self-adjoint and non-negative, 
while  property (S3) means that $\alpha$ is positive definite on the trace space $\gamma \Wtilde^\perp$.
{When $\Wtilde^\perp=\{\bld 0\}$, we take $\alpha = 0$ so that the resulting 
method is a (hybridized) mixed method. In this case, 
$a_{\Wtilde^\perp} = \infty$ and $\|\alpha\| = 0$. 
On the other hand, when 
$\Wtilde^\perp\not=\{\bld 0\}$, we can take 
$\alpha =Id$ to be the identity operator. In this case, we have 
$a_{\Wtilde^\perp} = 1$ and $\|\alpha\| = 1$. 
See in \cite[Section 2.1]{CockburnGopalakrishnanNguyenPeraireSayasHDGStokes11}
for a presentation of different choices of the stabilization operator 
$\alpha$.
}

The HDG-projection ${\Pi}_h(\und{\sigma}, \bld u )=
(\Pis\, \und{\sigma},\Piv \bld u)\in\VV\times \W$
is defined as follows:
\begin{subequations}
\label{proj}
\begin{alignat}{2}
\label{proj-1}
(\Pis\, \und{\sigma}, \und{\tau})_K = &\;
 (\und{\sigma}, \und{\tau})_K  &&\;\quad
\forall \und{\tau} \in \widetilde{\und\Sigma},\\
 \label{proj-2}
 (\Piv \bld u, \w)_K  =&\; (\bld u, \w)_K&&
\;\quad \forall \w \in \widetilde{\bld V},\\
\label{proj-3}
\langle \Pis\, \und{\sigma} \n - {\alpha(\Piv\bld u)} 
, \bld \mu \rangle_{\dK}  =&\; 
\langle \und{\sigma}  \n -\alpha(\bld P_{M} \bld u), \bld \mu 
\rangle_{\dK}&&\; \quad\forall \bld \mu
\in \bld M.
\end{alignat}
\end{subequations}

We have the following result on the approximation properties of the projection.
\gfu{The proof, given in Appendix B for completeness,
is 
very similar to the diffusion case presented in \cite[Appendix]{CockburnFuSayas16}.}
 
\begin{theorem}[Approximation properties of the HDG-projection]
 \label{actual-proj}
The auxiliary HDG-projection in \eqref{proj} is well-defined.
Moreover, 
we have
\begin{alignat*}{1}
    \| \und\sigma-\Pis {\und{\sigma}} \, \|_K\le&\;
    \| (Id-\und P_{\Sigma})\ {\und{\sigma}}\, \|_K
+\mathsf{C}_1\,h^{1/2}_K\,\|\left((Id-\und P_{{\Sigma}}){\und{\sigma}}\right) \n\|_{ \partial K}
\\&+  \mathsf{C}_2\,h_K\,
\|(Id-\bld {P}_{\widetilde{V}})\divv {\und{\sigma}}\|_{K}
+
\mathsf{C}_3\,h^{1/2}_K\,\|(Id-\bld P_{V})\bld u\|_{\partial K},
    \\
    \|\bld u-\Piv  \bld u \|_K \le&\;
    \| (Id-\bld P_V) \bld u\|_K +\mathsf{C}_4\,h^{1/2}_K\,
\|(Id- \bld P_V) \bld u\|_{\partial K}
  \\                         &+
\mathsf{C}_5\,h_K\,
\|(Id-\bld {P}_{\widetilde{V}})\divv {\und{\sigma}}\|_{K},
  \end{alignat*}
  where $\mathsf{C}_1:=C_{\Vperp}$ and
\begin{alignat*}{6}
\mathsf{C}_2&:=\frac{C_{\Wtilde^\perp}}{a_{\Wtilde^\perp}}\,C_{\Vperp}\,\|\alpha\|,
&&
\;\;
\mathsf{C}_3&&:={\left(1+\,\frac{\|\alpha\|}{a_{\Wtilde^\perp}}\right)\,C_{\Vperp}\,\|\alpha\|},
\;\mathsf{C}_4&:={\frac{C_{\Wtilde^\perp}}{a_{\Wtilde^\perp}}\,\|\alpha\|},
&&
\;\;\mathsf{C}_5&&:=\frac{C_{\Wtilde^\perp}^2}{a_{\Wtilde^\perp}}.
\end{alignat*}
Here
\begin{alignat*}{2}
C_{\Vperp}
&:=\sup_{\vv\in\Vperp\setminus\{\boldsymbol{0}\}} h^{-1/2}_K\,\|\vv\|_K/\|\vv\,\boldsymbol{n}\|_{\partial K},
\;\;\;\;
C_{\Wtilde^\perp}\!
&:=\sup_{\w\in\Wtilde^\perp\setminus\{{0}\}} h^{-1/2}_K\,\|\w\|_K/\|\w\|_{\partial K}.
\end{alignat*}
\end{theorem}
Note that, if $\W=\Wtilde=\divv \VV$, then
$\mathsf{C}_i=0$ for $i=2,3,4,5$ since in this case we have
$a_{\Wtilde^\perp}=\infty$ and $\|\alpha\|= 0$.
{Note also} that the above error estimates depend on the choice of the space 
$\Vtilde$ only through the stability 
constant $C_{\Vperp}$. 
{The constants $C_{\Vperp}$ and $C_{\Wtilde^\perp}$ are  
optimal bounds for inverse inequalities bounding the $\bld L^2(K)$-norm of $\vv\in \Vperp$ and 
$\w\in \Wtilde^\perp$ by the $\bld L^2(\partial K)$-norm of their respective traces.
They are independent of the mesh size $h_K$.
}

Now, if $\upol{k}{K}\times \bpol{k}{K}\subset \VV\times \W$, where 
$\upol{k}{K}$ is the symmetric-matrix-valued polynomial space of degree no greater than $k$
and $\bpol{k}{K}$ is the vector-valued polynomial space of degree no greater than $k$,
with the choice of stabilization
operator $\alpha = Id$, 
we get the following estimates from Theorem \ref{actual-proj}:
\begin{align*}
  \| \und\sigma-\Pis\und\sigma\|_K \le C\,h_K^{k+1} \left(
 \|\und\sigma\|_{k+1,K}+\|\bld u\|_{k+1,K}
 \right),\\
  \| \bld u-\Piv\bld u\|_K \le C\,h_K^{k+1} \left(
   \|\und\sigma\|_{k+1,K}+\|\bld u\|_{k+1,K}
 \right),
\end{align*}
where $\|\cdot\|_{k+1,K}$ denotes the $H^{k+1}(K)$-norm.
Hence, the HDG-projection gives quasi-optimal converge in both variables.

\subsection{The error estimates}
Now, we are ready to present our main results on the a priori error estimates. 
We display their proofs in Appendix A.

\subsubsection{Estimate of the stress approximation}
We start with the estimation on the projection error $\Pis\und{\sigma}-\und\sigma_h$.
\begin{theorem}\label{L2error-q}
For the solution of of the HDG method given by \eqref{weak formulation},
we have
\begin{subequations}
 \label{sigma-e}
\begin{alignat}{1}
 \label{sigma-e1}
 \|\Pis\, \und{\sigma} - \und{\sigma}_h\|_{\mathcal{A},\Oh} 
 &\leq \|\und\sigma-\Pis\, \und{\sigma}\|_{\mathcal{A},\Oh}. 
 \end{alignat} 
 Moreover, if we have 
 a homogeneous and isotropic material with the compliance tensor  
 given by \eqref{isotropic},
then 
\begin{alignat}{1}
\label{sigma-e2}
 \|\Pis\, \und{\sigma} - \und{\sigma}_h\|_{\Oh} 
 &\leq C\,\left(1+h^{1/2}\|\alpha\|\right)\|\und\sigma-\Pis\, \und{\sigma}\|_{\Oh}. 
 \end{alignat}
 \end{subequations}
 Here the constant $C$ is independent of the mesh size $h$, the exact solution, and the compliance 
 tensor $\mathcal{A}$.
 \end{theorem}

Note that, since the estimate \eqref{sigma-e2} implies
$
\| \und{\sigma}- \und{\sigma}_h\|_{\Oh} \le C\, \| \und{\sigma}- \Pis\, \und{\sigma}\|_{\Oh},
$
the error $\| \und{\sigma}- \und{\sigma}_h\|_{\Oh}$ 
only depends on the approximation properties of the first component of the projection
$\Pi_h$.
Note also that  the estimate \eqref{sigma-e2} implies that the method is {\em free from volumetric locking}
in the sense that the error $\|\und\sigma-\und\sigma_h\|_\Oh$ does not grow as 
the Lam\'e constant $\lambda\rightarrow \infty$ in the incompressible limit.

\subsubsection{Local estimates of the piecewise derivatives and the jump term}
Now, we present local 
stability and error estimates on the piecewise divergence of $\und\sigma_h$,
the piecewise symmetric gradient of $\bld u_h$, 
and the jump term $\bld u_h-\widehat{\bld u}_h$, similar to the results in  
\cite[Section 4]{CockburnFuSayas16}.
\begin{theorem} 
\label{thm:stability}
For the solution of \eqref{weak formulation},
we have the following local stability  and error estimates:
\begin{alignat*}{1}
\|\divv \und \sigma_h\|_K
&\le C_1\,\|\mathcal{A}\|_{L^\infty(K)}^{1/2}\,\|\und{\sigma}_h\|_{\mathcal{A},K}+ {C_2\, \|P_{\W} \bld f\|_K},\\
\|\und\epsilon( \bld u_h)\|_K
&\le C_3\,\|\mathcal{A}\|_{L^\infty(K)}^{1/2}\,\|\und{\sigma}_h\|_{\mathcal{A},K}+ {C_4\, \|P_{\Wtilde^\perp}\bld f\|_K},
\\
\|\bld u_h-\widehat{\bld u}_h\|_{\partial K}
&\le C_5\,h_K^{1/2}\|\mathcal{A}\|_{L^\infty(K)}^{1/2}\,\|\und{\sigma}_h\|_{\mathcal{A},K}+ {C_6\,h_K^{1/2} \|P_{\Wtilde^\perp}\bld f\|_K},
\end{alignat*}
\begin{alignat*}{1}
\|\divv (\Pis\und\sigma-\und \sigma_h)\|_K
&\le C_1\,\|\mathcal{A}\|_{L^\infty(K)}^{1/2}\,\|\und\sigma-\und \sigma_h\|_{\mathcal{A},K}\\
\|\und\epsilon( \Piv\bld u-\bld u_h)\|_K
&\le C_3\,\|\mathcal{A}\|_{L^\infty(K)}^{1/2}\,\|\und\sigma-\und \sigma_h\|_{\mathcal{A},K},
\\
\|\Piv \bld u-\bld u_h-(\Pim\bld u-\widehat{\bld u}_h)\|_{\partial K}
&\le C_5\,h_K^{1/2}\|\mathcal{A}\|_{L^\infty(K)}^{1/2}\,\|\und\sigma-\und \sigma_h\|_{\mathcal{A},K},
\end{alignat*}
where
\begin{alignat*}{5}
C_1&:=C_{\Vtilde^\perp}\,C_{\divv \VV}\,\|\alpha\|\,\left(1+\frac{\|\alpha\|}{a_{\Wtilde^\perp}}\right),
&&\;\; 
C_2 &&:=1+ C_{\divv \VV}\,\|\alpha\|\,
\frac{C_{\Wtilde^\perp}}{a_{\Wtilde^\perp}},
 \\
C_3 &:=1+
C_{\Vtilde^\perp}\,C_{\und\epsilon( \W)}\left(1+\frac{\|\alpha\|}{a_{\Wtilde^\perp}}\right),
&&\;\; 
C_4 &&:=C_{\und\epsilon (\W)}\,
\frac{C_{\Wtilde^\perp}}{a_{\Wtilde^\perp}},\\
C_5&:=C_{\Vtilde^\perp}\left(1+\frac{\|\alpha\|}{a_{\Wtilde^\perp}}\right),
&&\;\;
C_6 &&:= \frac{C_{\Wtilde^\perp}}{a_{\Wtilde^\perp}}.
\end{alignat*}
Here
\begin{alignat*}{2}
C_{\und\epsilon (\W)}&:=
\sup_{\und{0}\not = \und{\tau}\in \und\epsilon( \W)} 
h^{1/2}_K\,\|\und{\tau}\boldsymbol{n}\|_{\partial K}/\|\und{\tau}\|_K,
&\;\;\;\;C_{\divv\VV}
&:=\sup_{\bld 0\not= \boldsymbol{v}\in \divv\VV}
h_K^{1/2}\,\|\w \|_{\partial K}/\|\boldsymbol{v}\|_{K}.
\end{alignat*}
\end{theorem}
Note that, summing over all the elements $K\in \Oh$, we easily get that 
\[\|\divv(\Pis\und\sigma-\und\sigma_h)\|_\Oh,\quad 
\|\und\epsilon(\Piv\bld u-\bld u_h)\|_\Oh,\quad
\|\Piv \bld u-\bld u_h-(\Pim\bld u-\widehat{\bld u}_h)\|_{\dOh}
\]
only depend on $\|\und\sigma-\und\sigma_h\|_{\mathcal{A},\Oh}$, which, in turn, 
depends on the first component of the projection $\Pi_h$.
\subsubsection{Estimates of the approximation of the displacement}

Our next result shows that $\Piv\bld u -\bld  u_h$
 can {\em also} be controlled  
solely in terms of the approximation error of the auxiliary projection 
$ \und{\sigma}- \Pis\, \und{\sigma}$.
In addition, an improvement can be achieved under a typical elliptic regularity
property we state next. We assume that, for any given $\bld \theta \in \bld L^2(\Omega)$, we have 
\begin{equation}\label{regularity}
\|\bld \phi\|_{2, \Omega} + \|\und{\psi}\|_{1, \Omega} \le C \|\bld \theta\|_{\Omega},
\end{equation}
where $C$ only depends on the domain $\Omega$, 
and $(\und{\psi},\bld \phi)$ is the solution of the \emph{dual} problem:
\begin{subequations}\label{adjoint}
\begin{alignat}{2}
\label{adjoint-1}
\mathcal{A}{\und \psi}-\und\epsilon(\bld \phi)  & = 0 \quad &&\text{in $\Omega$,}\\
\label{adjoint-2}
\divv \und{\psi} & = \bld \theta &&\text{in $\Omega$,}\\
\label{adjoint-3}
\bld\phi &= \bld 0 && \text{on $\partial \Omega$.}
\end{alignat}
\end{subequations}

We are now ready to state our \red{result}.

\begin{theorem}\label{L2error-u}
If $\bpol{1}{K}\subset \W(K)$ for every element $K\in \mathcal{T}_h$, and the 
elliptic regularity property \eqref{regularity} holds, then,
for the solution of \eqref{weak formulation}, we have
\begin{align*}
 \|\Piv\bld u - \bld u_h\|_{\Oh} \le &\;C\, h \,\| \und{\sigma} - \Pis\, \und{\sigma}\|_{\Oh}. 
\end{align*}
The constant $C$ depends on $\mathcal{A}$ but is independent of $h$ and the exact solution. 

\end{theorem}

Combining this result with the last estimate in Theorem \ref{thm:stability} and applying simple triangle, trace and inverse 
inequalities, we immediately get
\begin{align*}
 \|\Pim\bld u - \widehat{\bld u}_h\|_{\dOh}  \le &\; 
 C\, h^{1/2}\, \| \und{\sigma} - \Pis\, \und{\sigma}\|_{\Oh},
\end{align*}
and we see that the quality of the approximation $\bld u_h$ and $\widehat{\bld u}_h$ only depends on the approximation error
of the auxiliary projection, as claimed.

\subsubsection{Estimates of a postprocessing of the displacement}
Note that if $h\, \| \und{\sigma} - \Pis\, \und{\sigma}\|_{\Oh}$ converges faster 
than $\|\bld u-\Piv\bld u\|_\Oh$, the convergence of $\bld u_h$ to $\Piv\bld u$ is {\em faster} than
that of $\bld u_h$ to $\bld u$. As mentioned before, we can take advantage of this 
{\em superconvergence} result to
show the existence of a displacement postprocessing $\bld u^*_h$ 
%
converge to $\bld u$
as fast as $\bld u_h$ superconverges to $\Piv\bld u$.  
To this end, we associate to each element $K$ a vector space $\W^*(K)$ that contains $\W(K)$.
Then, the function $\bld u_h^{*}$ is the 
element of ${\bld V}^*(K)$ such that
\begin{subequations}
\label{u2*}
\begin{alignat}{2}
\label{u2*-1}
\left(\gradv \bld u_h^{*},\gradv \bld  w\right)_K=& 
-\left(\bld u_h,\triangle \,\bld  w \right)_K
+\bintK{\widehat{\bld u}_h}{\gradv\bld  w\,\n}
&&\quad\forall\;\bld  w 
\in\widetilde{{\bld V}}^{*}(K)^\perp,
\\
\label{u2*-2}
(\bld u_h^{*},\bld r)_K=&\;(\bld u_h,\bld r)_K && \quad \forall \; 
\bld r\in 
\widetilde{{\bld V}}^{*}(K),
\end{alignat}
\end{subequations}
where $\widetilde{{\bld V}}^{*}(K)^\perp\subset {{\bld V}}^{*}(K)$ is the 
$L^2$-orthogonal complement of $\widetilde{{\bld V}}^{*}(K)$,  and 
$\widetilde{{\bld V}}^{*}(K)$ is any non-trivial subspace of $\divv \und\Sigma(K)$ containing 
the  constant vectors
$
 \bpol{0}{K}.
$

We have the following estimate. 
\begin{theorem}\label{L2error-u*}
Suppose that
\[
 \bpol{0}{K}\subset \divv\VV(K),
 \quad \triangle \W^*(K)\subset \divv\VV(K),\quad
 (\gradv\W^*)\n\left|_{\dK}\right.\subset \bld M(\dK).
\]
Let $\bld u_h^*$ be the solution to \eqref{u2*} with $(\bld u_h, \widehat {\bld u}_h)$
being the solution to \eqref{weak formulation},
then 
\[
\|\bld u-\bld u^{*}_h\|_{K} \leq C \big(\|\Piv\bld u -\bld u_h\|_{K}+h_K^{1/2}\|\Pim\bld u -\widehat{\bld u}_h\|_{\dK}
+\gfu{h_K\|\gradv(\bld u-\bld P_{V^*}\bld u)\|_K}\big).
\]
Here the constant $C$ depends on $\mathcal{A}$ but is independent of $h$ and the exact solution, \gfu{and 
$\bld P_{V^*}$ is the $L^2$-projection onto ${\bld V}^*(K)$}. 
\end{theorem}

Note again that after summing the estimates over all the elements $K\in\Oh$, we get 
\begin{align*}
\|\bld u-\bld u^{*}_h\|_{\Oh} \leq&C  \big(\|\Piv\bld u -\bld u_h\|_{\Oh}+h^{1/2}\,\|\Pim\bld u -\widehat{\bld u}_h\|_{\dK}
+h
\|\gradv(\bld u-\bld P_{V^*}\bld u)\|_\Oh\big).
\end{align*}

\subsubsection{A practical example}
\red{To end this section}, 
we apply the error estimates in Theorem \ref{actual-proj}, 
and the error estimates in Theorems \ref{L2error-q}--\ref{L2error-u*} to 
obtain convergence rates for $L^2$-error of $\und\sigma_h$, 
$\bld u_h$, and $\bld u_h^*$ for a special case with 
the following conditions on the spaces (on each element $K$) and the 
stabilization operator:
\begin{itemize}
\item [(C.1)] $\bld M = \ppol_k(\dK)$, and $\VV\times \W$ admits an \mm-decomposition with 
 $\ppol_k(K,\mathbb{S})\times \ppol_k(K)\subset \VV \times \W$.
\item [(C.2)] $\W^* = \ppol_{k+1}(K)$.
\item [(C.3)] $\alpha = Id$.
\end{itemize}
In this case, we get that
\begin{subequations}
\label{error-practical} 
\begin{alignat}{2}
 \|\und\sigma-\und\sigma_h\|_\Oh\le &\;C\, h^{k+1}(\|\und\sigma\|_{k+1}+\|\bld u\|_{k+1}),\\
 \|\bld u-\bld u_h\|_\Oh\le &\;C\, h^{k+1}(\|\und\sigma\|_{k+1}+\|\bld u\|_{k+1}),\\
 \|\bld u-\bld u_h^*\|_\Oh\le &\;C\, h^{k+2}(\|\und\sigma\|_{k+1}+\|\bld u\|_{k+2}),
 \end{alignat}
\end{subequations}
where the last estimate require the regularity estimate \eqref{regularity} holds.

We remark that, as we will make clear in the next two sections, the natural choice 
$\VV\times \W\times \bld M:=\ppol_k(K,\mathbb{S})\times \ppol_k(K)\times \ppol_k(\dK)$ \red{proposed} in 
\cite{SoonCockburnStolarski09} and 
analyzed in \cite{FuCockburnStolarski14} does not satisfy condition (C.1) due to the lack of an 
\mm-decomposition for  $\VV\times \W$. Actually, for this choice of spaces, with $\alpha = Id$, 
it was  proven in \cite{FuCockburnStolarski14} that
\begin{alignat*}{2}
 \|\und\sigma-\und\sigma_h\|_\Oh\le &\;C\, h^{k+1/2}(\|\und\sigma\|_{k+1}+\|\bld u\|_{k+1}),\\
 \|\bld u-\bld u_h\|_\Oh\le &\;C\, h^{k+1}(\|\und\sigma\|_{k+1}+\|\bld u\|_{k+1}),\\
 \|\bld u-\bld u_h^*\|_\Oh\le &\;C\, h^{k+1}(\|\und\sigma\|_{k+1}+\|\bld u\|_{k+1}),
 \end{alignat*}
where numerical results suggested that the orders are actually sharp for $k=1$.
We will see in Section \ref{sec:main} that on triangular meshes, we only need to add {\em two} 
(rational) basis functions to $\VV$ for $k=1$, and {\em three} for $k\ge 2$ to obtain an
\mm-decomposition. Then, the desired (superconvergence) error estimates \eqref{error-practical}
follow.

\section{The \mm-decompositions}
\label{sec:m1}
\red{In this section, we obtain a characterization 
of \mm-decompositions}.
We then show how to use it to 
construct HDG and (hybridized) mixed methods 
that superconverge on unstructured meshes.

\subsection{A characterization of \mm-decompositions}
We first give a characterization of \mm-decompositions
expressed solely in terms of the spaces $\VV\times \W$. Roughly speaking, it
states that $\VV\times \W$ admits an \mm-decomposition if and only if the space $\bld M$ is the orthogonal sum 
of the traces of the kernels of $\divv$ in $\VV$ and of $\und{\epsilon}$ in $\W$. It is expressed in terms of a 
special integer we define next.

\begin{definition}[The \mm-index] The \mm-index of the space $\VV\times \W$ is the number
\begin{alignat*}{2}
I_{\bld M}(\VV\times \W):=&\;\dim \bld M &&-\dim\{\vv \nn|_{\partial K}:\, \vv\in \VV,\, \divv\vv=0\}
                                     \\&&&-\dim \{\w|_{\partial K}:\, \w\in \W,  \, \und{\epsilon} (\w)=0\}.
\end{alignat*}
\end{definition}

\begin{theorem}[A characterization of \mm-decompositions] 
\label{thm:1.5}
For a given space of traces $\bld M$, the space $\VV\times \W$ admits an \mm-decomposition if and only if 
\begin{itemize}
\item[{\rm (a)}] $\tr (\VV\times \W)\subset \bld M$,
\item[{\rm (b)}] $\und{\epsilon} (\W)\subset \VV$, $\divv \VV\subset \W$,
\item[{\rm (c)}] $I_{\bld M}(\VV\times \W)=0$.
\end{itemize}
In this case, we have 
\[
{\bld M}=\{\vv\,\nn|_{\partial K}:\, \vv\in \VV, \divv\vv=0\}{\oplus}\{\w|_{\partial K}:\, \w\in \W, \und{\epsilon}( \w)=0\},
\]
where the sum is orthogonal. 
\end{theorem}

\gfu{
The proof of the above result, which is very similar to the diffusion case considered in 
\cite[Section 2.4]{CockburnFuSayas16}, is given in Appendix C for \red{the sake of} completeness.
}

The importance of this 
result resides in that it allows us to know if any given space
$\VV\times \W$ admits an \mm-decomposition by just verifying some inclusion properties and by computing a 
single number, namely, $I_{\bld M}(\VV\times \W)$ -- a natural number, by property (a). Moreover, this result shows explicitly how 
$\bld M$ can be expressed in terms of traces of the kernels of the divergence in
$\VV$ and the trace of the kernel of the symmetric gradient in $\W$; we call the identity {\em the kernels' trace decomposition}. This identity is
going to be
the guiding principle for the systematic construction of \mm-decompositions we develop in the next subsection.

\subsection{The construction of \mm-decompositions}

%
%
%
%

Now,  we propose three ways of obtaining \mm-decompositions; we follow \cite[Section 5]{CockburnFuSayas16}.
We show how to modify a given space $\VVV\times\WWW$, which is assumed to 
\red{satisfy the first} two inclusion properties of an \mm-decomposition, 
to \red{obtain} a new space $\VV\times \W$ admitting an \mm-decomposition. 
By the assumption on the given space $\VVV\times\WWW$, 
the indexes $I_{\bld M}(\VVV\times\WWW)$ and 
\begin{alignat*}{2}
I_{S}(\VV\times \W):=&\dim\W-\dim\divv\VV,
\end{alignat*}
\noindent are non-negative. 
We propose three different ways of doing this according whether the indexes are zero or not.  

\

{\it The case $I_{\bld M}(\VVV\times\WWW)>0$.} In this case, the space $\VVV\times\WWW$ does not admit 
an \mm-decomposition.
By Theorem \ref{thm:1.5}, we have that 
\[
\{\vv\cdot\nn|_{\partial K}:\, \vv\in \VVV, \divv\vv=\bld 0\}{\oplus}
\{\w|_{\partial K}:\, \w\in \WWW, \und{\epsilon} (\w)=\und 0\}\underset{\neq}{\subset} \bld{M}.
\]
{To simplify the notation, we set $\VVV_{s}:=\{\vv\in \VVV: \divv\vv=\bld 0\}$
to be the divergence-free subspace of $\VVV$ ($s$ stands for solenoidal), and
$\WWW_{\!rm}:=\{\w\in \WWW: \und{\epsilon} (\w)=\und 0\}$
to be the $\und{\epsilon}$-free subspace of $\WWW$ ($rm$ stands for rigid motions).}
We see that, in order to achieve equality, we have to, roughly speaking,  {\em fill} the remaining part of 
$\bld M$ by adding a space 
of symmetric-tensorial, solenoidal functions $\delta\VV_{\mathrm{fillM}}$ of dimension $I_{\bld M}(\VVV\times \WWW)$. 
The precise description of this subspace is in the following result. 

\begin{proposition}[Filling the space of traces $\bld M$]
\label{prop:fillingM}
Let $\VVV\times \WWW$ satisfy the inclusion properties {\rm (a)} and {\rm (b)} of Theorem \ref{thm:1.5}. Assume that 
$\delta\VV_{\mathrm{fillM}}$ satisfies:
\begin{itemize}
\item[{\rm (a)}] $\gamma \delta\VV_{\mathrm{fillM}}\subset \bld M$,
\item[{\rm (b)}] $\divv \delta\VV_{\mathrm{fillM}}=\{\bld 0\}$,
\item[{\rm (c)}] $\gamma\VVV_{s} \cap \gamma \delta\VV_{\mathrm{fillM}}=\{\boldsymbol{0}\}$,
\item[{\rm (d)}] $\mathrm{dim}\, \delta\VV_{\mathrm{fillM}}=\mathrm{dim}\,\gamma \delta\VV_{\mathrm{fillM}}=I_{\bld M}(\VVV\times \WWW)$.
\end{itemize}
Then $(\VVV\oplus \delta\VV_{\mathrm{fillM}})\times \WWW$ admits an \mm-decomposition.
Moreover, at least one space $\delta\VV_{\mathrm{fillM}}$ can be constructed
when $\WWW_{\!rm} =  \bld {RM}(K)$ where
$
\bld{RM}(K):= \{\w\in \bld H^1(K):\quad \und\epsilon(\w) = \und 0\}
 $
is the space for {\it rigid motions}.

\end{proposition}

\begin{proof}
Let us just show how to construct one space $\delta\VV_{\mathrm{fillM}}$.
Let $\bld{\mathcal{B}}$ be a basis for $(\tr(\VVV_{s}\times \WWW_{\!rm}))^\perp$. 
Then we can take $\delta\VV_{\mathrm{fillM}}$
as the span of  $\{\und{\epsilon}(\bld\phi_{\bld \mu})\}_{\bld{\mu}\in\bld {\mathcal{B}}}$ where
\[
\divv(\und{\epsilon} (\bld\phi_{\bld{\mu}}))=\bld 0 \mbox{ in }K, \qquad
 \und{\epsilon}(\bld\phi_{\bld{\mu}})\,\boldsymbol{n}=\bld \mu\mbox{ on }\partial K,
\]
where $\bld \mu\in\bld {\mathcal{B}}$. 
Since $\WWW_{\!rm} = \bld {RM}(K)$,
the $\bld L^2(\partial K)$-projection of $\bld \mu$
onto $\gamma\bld{RM}(K)$ is zero and so, $\und{\epsilon}(\bld\phi_{\bld{\mu}})$
is well defined.   The boundary condition ensures the satisfaction of conditions (a) and (c), and condition (b) holds by construction.
Finally, condition (d) is also satisfied given that 
the set $\{\und{\epsilon}(\bld\phi_{\bld \mu})\}_{\bld \mu\in\bld{\mathcal{B}}}$ is linearly independent, and 
\[
\dim \delta\VV_{\mathrm{fillM}} =\dim \bld{\mathcal{B}} =  \dim \bld M - \dim \tr(\VVV_{s}\times \WWW_{\!rm}) = I_{\bld M}(\VVV\times \WWW).
\]
This completes the proof.
\end{proof}


\

{\it The case $I_{\bld M}(\VVV\times\WWW)=0$ but $I_S(\VVV\times\WWW)>0$.}
 In this case, the space $\VVV\times\WWW$ admits an \mm-decomposition
but  $\divv\VVV$ is a proper subspace of $\WWW$.
By the kernels' trace decomposition of Theorem \ref{thm:1.5}, \red{we have}
\[
\{\vv\nn|_{\partial K}:\, \vv\in \VVV, \divv\vv=\bld 0\}{\oplus}\{\w|_{\partial K}:\, \w\in \WWW, \und{\epsilon}( \w)=\und 0\}= \bld{M},
\]
and we then see that, if we seek a modification of 
$\VVV\times \WWW$ of the form $\VVV\times \W$, it must be such that
\[
\{\w|_{\partial K}:\, \w\in \W, \und{\epsilon}( \w)=\und 0\}=
\{\w|_{\partial K}:\, \w\in \WWW, \und{\epsilon}(\w)=\und 0\}.
\]
The following result gives a hypothesis under which we are allowed to reduce $\WWW$ to 
$\W:=\divv\VVV$. 
\begin{proposition}[Reducing the space $\W$]
\label{prop:reducingW}
Assume that $\VVV\times \WWW$ admits an \mm-decomposition.
Then $\VVV \times \divv\VVV$ admits an \mm-decomposition provided that $\divv\VVV$ contains 
the space of rigid motions $\bld {RM}(K)$.
\end{proposition}
\begin{proof}
Since $\bld {RM}(K)\subset \divv\VVV\subset \WWW$, we have 
\[
\{\w|_{\partial K}:\, \w\in \divv\VVV, \und{\epsilon}( \w)=\und 0\}=
\{\w|_{\partial K}:\, \w\in \WWW, \und{\epsilon}(\w)=\und 0\}
=\gamma \bld{RM}(K). 
\]
This completes the proof.
\end{proof}


Now, let us seek a modification of $\VVV\times \WWW$ of the form $\VV\times \WWW$, where $\VVV\subset \VV$.
Since
\[
\divv \VVV \underset{\neq}{\subset} \WWW,
\]
we see that in order to achieve the equality, 
we have to, roughly speaking,  {\em fill} the remaining part of $\WWW$ 
by adding a space of symmetric-tensorial, {\em non}-solenoidal functions
$\delta\VV_{\mathrm{fillV}}$ of dimension $I_S(\VVV\times\WWW)$. In this
case, we would immediately have that 
\[
\{\vv\nn|_{\partial K}:\, \vv\in \VV, \divv\vv=\bld 0\}=
\{\vv\nn|_{\partial K}:\, \vv\in \VVV, \divv\vv=\bld 0\},
\]
and, by Theorem \ref{thm:1.5}, 
the resulting space would admit an \mm-decomposition.
The precise way of choosing $\delta\VV_{\mathrm{fillV}}$
is described in the following result.

\begin{proposition}[Increasing the space $\VVV$]
\label{prop:increasingV}
Let the space $\VVV \times \WWW$ admits an 
\mm-decomposition and assume that $\divv\VVV$ is a proper subspace of $\WWW$. 
Let $\delta\VV_{\mathrm{fillV}}$ {satisfy the following hypotheses:} 
\begin{itemize}
\item[{\rm (a)}] $\gamma\delta\VV_{\mathrm{fillV}} \subset \bld M$,
{
\item[{\rm (b)}] $\divv \delta\VV_{\mathrm{fillV}}\subset \W$,
\item[{\rm (c)}] $\divv\VV \cap \divv \delta\VV_{\mathrm{fillV}}=\{{\bld 0}\}$,
\item[{\rm (d)}] $\mathrm{dim}\,\delta\VV_{\mathrm{fillV}}=
\mathrm{dim}\,\divv \delta\VV_{\mathrm{fillV}}=I_S(\VV\times\W)$.}
\end{itemize}
Then $(\VVV\oplus\delta\VV_{\mathrm{fillV}})\times \WWW$ admits an \mm-decomposition 
with $\WWW=\divv(\VVV\oplus\delta\VV_{\mathrm{fillV}})$. Moreover, at least one
space $\delta\VV_{\mathrm{fillV}}$ can be constructed {to satisfy all the 
hypotheses} when $\bld M$ contains
the space of traces of rigid motions $\gamma\bld{RM}(K)$. 
\end{proposition}




\begin{proof}
Let us just show how to construct one space $\delta\VV_{\mathrm{fillV}}$.
Let $\bld{\mathcal{B}}$ be a basis of $\Wtilde^\perp$. Then, we can take
$\delta\VV_{\mathrm{fillV}}$ as the span of
$\{\und{\epsilon}(\bld\phi_{{\w}})\}_{{\w}\in\bld{\mathcal{B}}}$ where
$\bld\phi_{{\w}}$ solves
\[
\divv(\und{\epsilon}(\bld\phi_{{\w}}))={\w} \mbox{ in }K \mbox{ and }
\und{\epsilon}(\bld\phi_{{\w}})\,\boldsymbol{n}=c({\w})\mbox{ on }\partial K,
\]
where $c({\w})$ is the element of $\gamma \bld{RM}(K)$ such that
\[
\langle c({\w}),\bld{\varphi}\rangle_{\partial K}=
({\w},\bld{\varphi})_K
 \quad\forall \bld\varphi\in \bld{RM}( K).
 \]  
  Note that since $\bld M$ contains the space $\gamma\bld{RM}(K)$, hypothesis (a) is
actually satisfied. Finally, it is not difficult to see that hypotheses (b), (c), and (d) are also satisfied
by the choice of $\bld{\mathcal{B}}$.
This completes the proof.
\end{proof}

\begin{table}[!ht]
 \caption{Three ways of constructing  spaces $\VV\times \W$ admitting an \mm-decomposition. The spaces are obtained by modifying 
 the space $\VVV\times\WWW$ according to whether it already admits an \mm-decomposition or not, and according to
 whether the space $\divv \VVV$ is a proper subspace of $\WWW$ or not.
 The space $\VVV\times\WWW$ is assumed to satisfy the first two inclusion properties of an \mm-decomposition, namely, 
 $\tr(\VVV\times\WWW)\subset \bld M$ and $\und\epsilon( \WWW)\times \divv \VVV\subset \VVV\times\WWW$. }
\centering
{
\begin{tabular}{c c c c c}
\hline
\noalign{\smallskip} 
way $\#$&$\begin{matrix}
\mbox{properties}\\
\mbox{of } \VVV\times\WWW
\end{matrix}$  &         $\VV$         &     $\W$                  &     
$\begin{matrix}
\mbox{properties of}\\
\VV\times \W
\end{matrix}$
          \\
\noalign{\smallskip}
\hline\hline
\noalign{\smallskip} 
$\underset {{\rm(Prop. \ref{prop:fillingM})}}{{\rm I}}$&
$\begin{matrix}
I_{\bld M}(\VVV\times\WWW)>0
\end{matrix}$
&
$\VVV\oplus\delta\VV_{\mathrm{fillM}}$& $\WWW$&
$\begin{matrix}
I_{\bld M}(\VV\times \W)=0\\
I_S(\VV\times \W)=I_S(\VVV\times\WWW)\\
\!\!\!\text{ if } \bld{RM}(K)\subset \WWW.
\end{matrix}$
\\
\noalign{\smallskip} 
\hline 
\noalign{\smallskip} 
$\underset {{\rm(Prop. \ref{prop:reducingW})}}{{\rm II}}$&
$\begin{matrix}
I_{\bld M}(\VVV\times\WWW)=0\\
I_S(\VVV\times\WWW)>0
\end{matrix}$
&$\begin{matrix}
 \VVV
\end{matrix}$&
$\begin{matrix} 
\divv\VVV\end{matrix}$&
$\begin{matrix}
I_{\bld M}(\VV\times \W)=0\\
I_S(\VV\times \W)=0\\
\!\!\! \text{ if } \bld{RM}(K)\subset \divv\VVV.
\end{matrix}
$
\\
\noalign{\smallskip} 
\hline 
\noalign{\smallskip} 
$\underset {{\rm(Prop. \ref{prop:increasingV})}}{{\rm III}}$&
$\begin{matrix}
I_{\bld M}(\VVV\times\WWW)=0\\
I_S(\VVV\times\WWW)>0
\end{matrix}$
&$\begin{matrix}
\VVV\oplus\delta\VV_{\mathrm{fillV}}
\end{matrix}$&
$\WWW$&
$\begin{matrix}
I_{\bld M}(\VV\times \W)=0\\
I_S(\VV\times \W)=0\\
\!\!\! \text{ if } \gamma\bld{RM}(K)\subset \bld M.
\end{matrix}$
\\
\noalign{\smallskip} 
\hline 
\end{tabular}
}
\label{table:ways}
\end{table}

\begin{table}[!ht]
 \caption{Spaces $\VV\times \W$
   admitting an \mm-decomposition. They are constructed from the single space
   $\VVV\times\WWW$ which is assumed to satisfy the first two inclusion properties
   of an \mm-decomposition, namely,  
   $\tr(\VVV\times\WWW)\subset \bld M$ and $\und\epsilon( \WWW)\times \divv \VVV\subset \VVV\times\WWW$. }
\centering
{
\begin{tabular}{c l l l}
\hline
\noalign{\smallskip} 
 way $\#$&         $\phantom{ooooooo}\VV$         &     $\phantom{ooo}\W$                  &     
          \\
\noalign{\smallskip}
\hline\hline
\noalign{\smallskip} 
III&\red{$\VV^{\mathrm{mix}}\!\!:=\VV^{\mathrm{hdg}}\oplus\delta\VV_{\mathrm {fillV}}$}& $\W^{\mathrm{mix}}
\!\!:=\WWW$& \\
I&\red{$\VV^{\mathrm{hdg}}\!\!:=\VVV\;\;\,\,\oplus\delta\VV_{\mathrm{fillM}}$}& $\W^{\mathrm{hdg}}\!\!:=\WWW$
\\
II&\red{$ \VV_{\mathrm{mix}}\!\!:=\VVV\,\;\,\;\oplus\delta\VV_{\mathrm{fillM}}$}& $\W_{\mathrm{mix}}\!:=\divv \VVV$&
\\
\noalign{\smallskip} 
\hline 
\end{tabular}
}
\label{table:systematic}
\end{table}

\begin{table}[!ht]
 \caption{The main properties of the spaces $\delta\VV$.}
\centering
{
\begin{tabular}{c c c c l}
\hline
\noalign{\smallskip} 
  $\delta\VV$ &   $\divv\delta\VV$    &
 $\gamma\delta\VV$&$\dim\dvm$               \\
\noalign{\smallskip}
\hline\hline
\noalign{\smallskip} 
$\delta\VV_{\mathrm{fillM}}$& $\{\boldsymbol{0}\}$&\hspace{-0.5cm}$\subset \bld M$, $\cap
\gamma\VVV_{s}=\{\bld 0\}$
&$I_{\bld M}(\VVV\times\WWW)$ & \hspace{-0.35cm}$=\dim \gamma\,\dvm$\\
$\delta\VV_{\mathrm{fillV}}$& $\subset \WWW$, $\cap
\divv\VVV=\{\bld 0\}$ & $\subset
\bld M$ & $I_S(\VVV\times\WWW)$& \hspace{-0.4cm}$=\dim \divv\dvm$\\
\noalign{\smallskip} 
\hline 
\\
\end{tabular}
}
\label{table:deltaspaces}
\end{table}

\subsection{A systematic procedure for obtaining \mm-decompositions} 
We can now use these three ways of obtaining \mm-decompositions, summarized in Table \ref{table:ways},  to propose a systematic way for constructing spaces admitting \mm-decompositions starting from a single, 
given space $\VVV\times\WWW$. Let us recall that the space
$\VVV\times\WWW$  is assumed to satisfy the first two inclusion properties of an \mm-decomposition, 
and so the indexes $I_{\bld M}(\VVV\times\WWW)$ and $I_S(\VVV\times\WWW)$ are non-negative. 

The systematic construction is described in Tables 
\ref{table:systematic} and \ref{table:deltaspaces}.
Note that the construction provides three different spaces admitting \mm-decompositions. 
The first is associated to an HDG method.
The other two are associated to (hybridized) mixed methods which can 
be though of as {\em sandwiching} the HDG method.

It is now clear that we are left to construct the filling spaces $\delta\VV_{\mathrm{fillM}}$ and 
$\delta\VV_{\mathrm{fillV}}$ that satisfy the properties
in Table \ref{table:deltaspaces} for a given space $\VVV\times\WWW$ satisfying the first two inclusions of an \mm-decomposition. 
In the next section, we present such spaces defined on general polygonal elements in two-space dimensions.

\section{An explicit construction of 
the spaces $\delta\VV_{\mathrm{fillM}}$ and $\delta\VV_{\mathrm{fillV}}$ in two-space dimensions}
\label{sec:main}
This section contains the explicit construction of  the spaces $\delta\VV_{\mathrm{fillM}}$ and $\delta\VV_{\mathrm{fillV}}$
satisfying the properties in Table \ref{table:deltaspaces} in two-space dimensions. 

Here we consider a polygonal element $K$, 
\red{fix the trace space}
$\boldsymbol{M}(\partial K):=\bpol{k}{\partial K}$, and
\red{study} two choices of the initial guess spaces 
$\VVV\times \WWW$, namely, 
\[
 \VVV\times \WWW:=
 \bppol_k\times \ppol_{k}\quad
 \text{ and }\quad  \VVV\times \WWW:=
 \bqqol_k\times \qqol_{k},
\]
where $
 \bppol_k:=\pcol_k(K;\mathbb{S})$ and $ 
  \bqqol_k:=Q_k(K;\mathbb{S})$ are the symmetric-tensorial polynomial fields.

\subsection{Notation, the Airy stress operator and the 
lifting functions}To state our results, we need to introduce some notation. 
Let $\{\vt_i\}_{i=1}^{ne}$ be the set of vertices of the polygonal element 
$K$ which we take to be counter-clockwise ordered. 
Let $\{\eg_i\}_{i=1}^{ne}$ be the set of edges 
of $K$ where the edge $\eg_i$ connects the vertices $\vt_i$ and $\vt_{i+1}$.
Here the subindexes are integers module  $ne$, for example, $\vt_{ne+1}=\vt_1$.
An illustration for a quadrilateral element $K$ is presented in Fig.~\ref{Fig:element}. 
\begin{figure}[ht!]
\centering
\includegraphics[scale=.4]{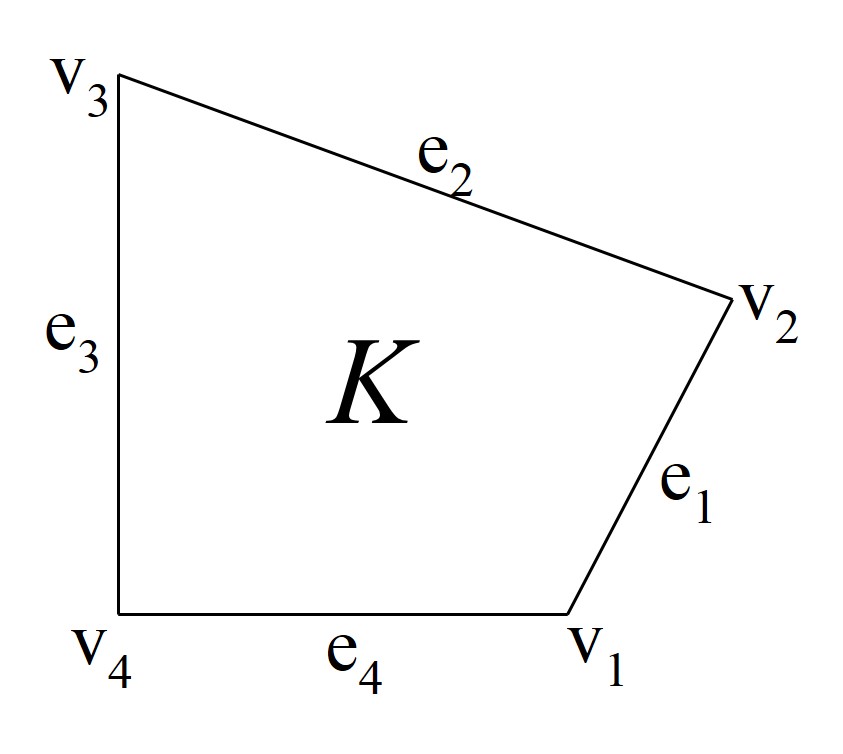}
\caption{A quadrilateral element $K$.}
\label{Fig:element}
\end{figure}
We also define, for $1\le i\le ne$,  $\lambda_i$ to be the  linear function that 
vanishes on edge $\eg_i$ and reaches maximum value $1$ in the closure of the  element $K$.

Since $\delta\VV_{\mathrm{fillM}}$
is a divergence-free symmetric tensor field, it can be 
characterized, see\red{, for example,} \cite{ArnoldWinther02}, as the Airy stress operator of some 
$H^2$-conforming scalar field,  
where the Airy stress operator
is defined as follows:
\begin{align*}
 J : = \left(\begin{tabular}{c c}
        $\frac{\partial^2 }{\partial y^2}$ & 
$- \frac{\partial^2 }{\partial x\partial y}$ \\
$         -\frac{\partial^2 }{\partial x\partial y} $&  
$\frac{\partial^2 }{\partial x^2}$
        \end{tabular}\right).
\end{align*}
%

Now, we introduce two functions which we are going to use as tools to define {\it lifting} 
of traces on $\dK$ into the inside of the element $K$. 
The first is associated to a vertex.
To each vertex $\vt_i$, we associate a   
 function $\xi_i$ satisfying the following conditions: 
 
 \
 
 \begin{tabular}{l }
(L.1) $\xi_i|_{\eg_j}\in \pcol_1(\eg_j),\;j=1,\dots,ne.$ \\
 \vspace{.1cm}
(L.2) $\xi_i(\vt_j)=\delta_{ij},\;j=1,\dots,ne.$ \\
  \vspace{.1cm}
 (L.3) $\frac{\partial\xi_i}{\partial \n_j}|_{\eg_j}\in \pcol_0(\eg_j),\;j=1,\dots,ne.$ 
 \end{tabular}

\

 Here, $\n_j$ denotes the outward normal to the edge $\eg_j$.
 The second is associate to an edge. To each edge $\eg_i$, we associate a function $B_i$ satisfying the following conditions:
 
 \
 
 \begin{tabular}{l }
(H.1) $B_i|_{\eg_j}= 0,\;j=1,\dots,ne,$\\
 \vspace{.1cm} 
(H.2) $\frac{\partial B_i}{\partial \n_j}|_{\eg_j}= 0,
\;j=1,\dots,ne, j\not= i.$\\
  \vspace{.1cm}
 (H.3) $\frac{\partial B_i}{\partial \n_i}|_{\eg_i}= \lambda_{i-1}\lambda_{i+1},
\;j = i.$
 \end{tabular}

\

Next, we give some examples of these functions.
If $K$ is a triangle, we simply take $\xi_i := \lambda_{i+1}$ and 
$B_i = (\prod_{k=1}^3\lambda_k) \; \prod_{j\not= i}\frac{\lambda_j}{\lambda_j+\lambda_i}$. 
Note that $B_i$ is nothing but the rational bubble related to the edge $\eg_i$ defined in \cite{GuzmanNeilan14}. 
If $K$ is a star-shaped polygon with respect to an interior node $\vt_o$, we subdivide the element $K$ into 
$ne$ triangles $\{T_i\}_{i=1}^{ne}$, with $T_i$ begin the triangle with vertices  $\vt_o,\vt_{i-1},\vt_{i}$.
We then take
$\xi_i$ to be the piecewise linear function on $\{T_i\}_{i=1}^{ne}$ satisfying condition (L.2) and $\xi_i(\vt_o) = 0$. We
take $B_i$ to be the (composite) function that vanishes on $T_j$ for $j\not = i+1$ and equals to the rational bubble associated to $\eg_i$ on $T_{i+1}$. This 
choice is similar to the composite lifting introduced in \cite{CockburnFu17a}.  

Let us remark that the conditions (H) on the function $B_i$, 
derived from our analysis in the next section
(see Lemma \ref{lemma:Bi}),
 \red{ensure} that 
$
 \lim_{x\rightarrow \vt_{i+1}}\left. (J\,B_i) \n_{i}\cdot \n_{i+1}\right |_{\eg_i}\not = 
 \lim_{x\rightarrow \vt_{i+1}}\left. (J\,B_i) \n_{i+1}\cdot \n_{i}\right |_{\eg_{i+1}} = 0.
$ 
The importance of this 
nodal discontinuity, which is not \red{made evident} in our  construction of \mm-decompositions,
is well established in the literature; see, for example, the discussion 
at the last two paragraphs of Section 3 in \cite{ArnoldWinther02}. 
Indeed, in \cite{ArnoldWinther02}, it is 
argued that there exist no hybridizable mixed \red{method (which does not 
use vertex degrees of freedom) that only uses} polynomial shape functions. 
Hence, it comes at no surprise that our space 
$\dvm_{\mathrm{fillM}}$ consists of non-polynomial functions.

\subsection{The case $\VVV\times \WWW:=
 \bqqol_k\times \qqol_{k}$}
 We start by considering $K$ to be a unit square with edges parallel to the axes and $\vt_1 = (0,1)$. 
 (This implies $\lambda_1 = x, \lambda_2 = y, \lambda_3 = 1-x, \lambda_4 = 1-y$ )
 We omit the proof due to its similarity with the proof of 
 the more difficult result, Theorem \ref{thm:polygon}, in Section \ref{sec:proof}.

\gfu{
The proof of the next two theorems in this subsection is sketched 
in Appendix D. We remark that a more detailed proof for a similar, but much more difficult result, mainly
Theorem \ref{thm:polygon} in the next subsection, is given in Section \ref{sec:proof}.}
\begin{theorem}
\label{thm:square}
Let $K$ be a unit square with edges parallel to the axes. Then, for 
$\bld M=\ppol_k(\dK)$ and $\VVV\times \WWW = \bqqol_k(K)\times \qqol_k(K)$, where $k\ge 1$, we have that 
\[
 I_{\bld M}(\VVV\times \WWW) =\left\{
 \begin{tabular}{l l}
  $6$ & if $k = 1$,\\
  $9$ & if $k = 2$,\\
  $10$ & if $k \ge 3$,
 \end{tabular}
\right.
 \text{ and }\quad
  I_S(\VVV\times \WWW) = 3.
\]
Moreover, the spaces
\begin{align*}
 \delta\VV_{\mathrm{fillM}}&\!\!:= \!\!
 \begin{cases}
  J\,\mathrm{span}\{B_2, B_3,B_4, \xi_4^2,\\
  \hspace{1.2cm}\xi_4^2(1-x),\xi_4^2(1-y)\} & \mbox{ if } k = 1,
 \\
 J\,\mathrm{span}\{B_2, B_3,B_4,B_4 x,\\
  \hspace{1.2cm}\xi_4^2(1-x),\xi_4^2(1-y),\\
  \hspace{1.2cm}\xi_4^2(1-x)(1-y),\\
  \hspace{1.2cm}\xi_4^2(1-x)^2,\xi_4^2(1-y)^2\} & \mbox{ if } k = 2,
\\
  J\,\mathrm{span}\{B_2, B_3,B_4,B_4 x,\\
  \hspace{1.2cm}\xi_4^2(1-x)^{k-1},\xi_4^2(1-y)^{k-1},\\
 \hspace{1.2cm}\xi_4^2(1-x)^{k-1}(1-y),\xi_4^2(1-x)(1-y)^{k-1},\\  
 \hspace{1.2cm}\xi_4^2(1-x)^{k},\xi_4^2(1-y)^{k}\} & \mbox{ if } k \ge 3,
   \end{cases}
\\
           \delta\VV_{\mathrm{fillV}}&:=  
           \mathrm{span} \left\{\left[\begin{tabular}{c c}
                                                    $x^{k+1}y^{k-1}$ &$0$\\
                                                    $0$ &$0$
                                                   \end{tabular}\right],
                                                   \left[\begin{tabular}{c c}
                                                    $x^{k+1}y^{k}$ &$0$\\
                                                    $0$ &$0$
                                                   \end{tabular}\right],
                                                   \left[\begin{tabular}{c c}
                                                    $0$ &$0$\\
                                                    $0$ &$x^ky^{k+1}$
                                                   \end{tabular}\right]\right\},
\end{align*}
satisfy the properties in Table \ref{table:deltaspaces}. 
Here
 $\xi_4$ satisfies conditions (L) and $B_2, B_3, B_4$ satisfy conditions (H).
\end{theorem}

Let us remark that in practical implementation, we can take
$\xi_4$ to be the composite lifting function presented in the previous subsection, and  
$\{B_i\}$ 
to be the following rational functions:
\begin{align}
\label{bubble-square}
B_i = \prod_{k=1}^4\lambda_k\; \prod_{j\not = i} \frac{\lambda_j}{\lambda_j+\lambda_i}. 
\end{align}


When the polynomial degree $k\ge 2$, 
we can bypass the use of the composite function $\xi_4$
in the definition 
of $\delta\VV_{\mathrm{fillM}}$ by using an exponential function, as we see in the next result.


\begin{theorem}
\label{thm:square2}
Let $K$ be a unit square with edges parallel to the axes. Then, for 
$\bld M=\ppol_k(\dK)$ and $\VVV\times \WWW = \bqqol_k(K)\times \qqol_k(K)$, where $k\ge 2$, we have
the spaces
\begin{align*}
 \delta\VV_{\mathrm{fillM}}&\!\!:= \!\!
 \begin{cases}
 J\,\mathrm{span}\{B_2, B_3,B_4,B_4 x,\\
  \hspace{1.2cm}(x y)^2(1-x),(x y)^2(1-y),\\
  \hspace{1.2cm}(x y)^2(1-x)(1-y),\\
  \hspace{1.2cm}(x e^{1-y} y)^2(1-x)^2,(e^{1-x}x  y)^2(1-y)^2\} & \mbox{ if } k = 2,
\\
  J\,\mathrm{span}\{B_2, B_3,B_4,B_4 x,\\
  \hspace{1.2cm}(x y)^2(1-x)^{k-1},(x y)^2(1-y)^{k-1},\\
 \hspace{1.2cm}(x y)^2(1-x)^{k-1}(1-y),(x y)^2(1-x)(1-y)^{k-1},\\  
 \hspace{1.2cm}(x e^{1-y}y)^2(1-x)^{k},(e^{1-x}x y)^2(1-y)^{k}\} & \mbox{ if } k \ge 3.
   \end{cases}
\end{align*}
satisfy the properties in Table \ref{table:deltaspaces}. 
Here
$B_i$ are defined in \eqref{bubble-square}.
\end{theorem}


\subsection{The case $\VVV\times \WWW:=
 \bppol_k\times \ppol_{k}$} 
Now, we consider $K$ to be a general polygon without hanging nodes.
 
 \begin{theorem}
\label{thm:polygon}
Let $K$ be a  polygon of $ne$ edges without hanging nodes.
Then, for $\bld M:=\ppol_k(\dK)$ and $\VVV\times \WWW := \bppol_k\times \ppol_k$ with $k\ge 1$, we have that 
\[
 I_{\bld M}(\VVV\times \WWW) = 2(\theta+1)ne-\frac{1}{2}(\theta+3)(\theta+4),\;\;
 \text{ and }\quad
  I_S(\VVV\times \WWW) = 2(k+1),
\]
where $\theta:=\min\{k,2 ne-4\}$.

Moreover, the spaces
\begin{align*}
 \delta\VV_{\mathrm{fillM}}&\; :=\oplus_{i=1}^{ne}J \,\Psi_i,\\
  \delta\VV_{\mathrm{fillV}}&\; :=\{\und\phi^1_{k+1,a}, \und\phi^2_{k+1,a}\}_{a=0}^k
  \quad\quad\text{(for $k\ge 2$)}
\end{align*}
satisfy the properties in Table 3. Here
 \[
\Psi_i =\left\{\begin{tabular}{l l}
            $\{0\}$ & if $i = 1$,           
            \vspace{.3cm}
\\
            $\mathrm{span}
\left  \{ (\xi_{i+1})^2\lambda_{i+1}^b\right\}_{b=\max\{k+5-2i,0\}}^{k}$\\
$ \hspace{.2cm}\oplus \mathrm{span}
\left  \{ (\xi_{i+1})^2\lambda_i\lambda_{i+1}^b\right\}_{b=
\max\{k+4-2i,0\}}^{k-1}$\\
$ \hspace{.4cm}\oplus \mathrm{span}
\left  \{B_i\right\}
$
            & if $2 \le i\le ne-1$,
            \vspace{.3cm}
            \\
                        $ \mathrm{span}
            \left  \{B_i\right\}
$
            & if $i = ne$ and $k= 1$,\vspace{.3cm}\\
            $\mathrm{span}
\left  \{ (\xi_{i+1})^2\lambda_{i+1}^{2+b}\right\}_{b=\max\{k+5-2i,0\}}^{k-2}$\\
$ \hspace{.2cm}\oplus \mathrm{span}
\left  \{(\xi_{i+1})^2\lambda_{i}\lambda_{i+1}^{2+b}\right\}_{b=
\max\{k+4-2i,0\}}^{k-3}$\\
$ \hspace{.4cm}\oplus \mathrm{span}
\left  \{B_i, B_i\lambda_{i+1}\right\}
$
            & if $i = ne$ and $k\ge 2$,
            \end{tabular}\right.
 \] 
 where $\xi_{i+1}$ satisfies conditions (L) and $B_i$ satisfies conditions (H),
and 
\begin{align*}
\und \phi^1_{k+1,a}= \left[ \begin{tabular}{l l}
                  $x^{k+1-a}y^a$ & $0$\\
                  $0$ & $0$
                 \end{tabular}\right] +\sum_{i=1}^{ne}J\,(C^1_{ai}\, \xi_{i+1}^2\lambda_{i+1}^{k+1}+D^1_{ai}\,
                 \xi_{i+1}^2\lambda_i\lambda_{i+1}^{k}),\\
 \und\phi^2_{k+1,a}= \left[ \begin{tabular}{l l}
                  $0$ & $0$\\
                  $0$ & $x^{k+1-a}y^a$
                 \end{tabular}\right] +\sum_{i=1}^{ne}J\,(C^2_{ai}\, \xi_{i+1}^2\lambda_{i+1}^{k+1}+D^2_{ai}\,
                 \xi_{i+1}^2\lambda_i\lambda_{i+1}^{k}),                 
\end{align*}
where the constants $\{C^1_{ai},D^1_{ai},C^2_{ai},D^2_{ai}\}$ are chosen such that 
$\gamma(\und\phi^1_{k+1,a}), \gamma(\und\phi^2_{k+1,a})\in \ppol_k(\dK)$.
\end{theorem}


 Let us give a more compact presentation of the space $\delta\VV_{\mathrm{fillM}}$ in Theorem \ref{thm:polygon}
 for two special cases, namely, when $K$ is a triangle and when $K$  a quadrilateral.
 \vspace{.1cm}

\

 {\underline{\textbf{ $K$ is a triangle.}}}
 We have
\begin{align*}
 \delta\VV_{\mathrm{fillM}} =&\; \left\{\begin{tabular}{l l}
                                       $J\,\mathrm{span}\{
 B_2, B_3\}
$& if $k = 1$,\\
                                       $J\,\mathrm{span}\{
 B_2, B_3, B_3\lambda_1\}
$& if $k \ge 2$.
                                      \end{tabular}\right.
\end{align*}
Here $B_i = \Pi_{i=1}^3\lambda_i \cdot\Pi_{j\not= i}\frac{\lambda_j}{\lambda_j+\lambda_i}$
is the rational bubble defined in \cite{GuzmanNeilan14}. 
Notice that the filling space  $\delta\VV_{\mathrm{fillM}}$ on a triangle in
\cite{GuzmanNeilan14} (defined for $k\ge 2$), in our notation, is 
\[
 \delta\VV_{\mathrm{fillM}} =J\,\mathrm{span}\{
 B_1\lambda_2, B_2\lambda_3, B_3\lambda_1\},
\]
which can be easily verified to satisfy the properties in Table \ref{table:deltaspaces}.

\

 {\underline{\textbf{ $K$ is a quadrilateral.}}}
We have
\begin{align*}
 \delta\VV_{\mathrm{fillM}} = \left\{
 \begin{tabular}{l l}
  $J\,\mathrm{span}\{B_2, B_3,B_4, \xi_4^2,\xi_4^2\lambda_3,\xi_4^2\lambda_4\}$ & if $k = 1$,  \vspace{.1cm}
\\
  $J\,\mathrm{span}\{B_2, B_3,B_4,B_4 \lambda_1,\xi_4^2\lambda_3,\xi_4^2\lambda_4,$\\
  $\hspace{1.2cm}\xi_4^2\lambda_3\lambda_4,\xi_4^2\lambda_4^2,\xi_1^2\lambda_1^2\}$ & if $k = 2$,\vspace{.1cm}\\
    $J\,\mathrm{span}\{B_2, B_3,B_4,B_4 \lambda_1,\xi_4^2\lambda_3\lambda_4,\xi_4^2\lambda_4^2,$\\
  $\hspace{1.2cm}\xi_4^2\lambda_3\lambda_4^2,\xi_4^2\lambda_4^3,
  \xi_1^2\lambda_1^2,\xi_1^2\lambda_1^2\lambda_4,\xi_1^2\lambda_1^3\}$ & if $k = 3$,\vspace{.1cm}\\
    $J\,\mathrm{span}\{B_2, B_3,B_4,B_4 \lambda_1,\xi_4^{2}\lambda_3\lambda_4^{k-2},\xi_4^{2}\lambda_4^{k-1},$\\
  $\hspace{1.2cm}\xi_4^2\lambda_3\lambda_4^{k-1},\xi_4^2\lambda_4^k,
  \xi_1^2\lambda_1^{k-1},\xi_1^2\lambda_1^{k-2}\lambda_4,\xi_1^2\lambda_1^k,\xi_1^2\lambda_1^{k-1}\lambda_4\}$ & if $k \ge 4$.
 \end{tabular}
\right.
\end{align*}

Now, when $K$ is a square, we can use similar 
spaces in Theorem \ref{thm:square2} to bypass the use of
composite functions $\xi_4$ and $\xi_1$. 

\

 {\underline{\textbf{ $K$ is a unit square.}}}
We can choose \red{$\delta\VV_{\mathrm{fillM}}$
as in Theorem \ref{thm:square2}}:
\begin{align*}
 \delta\VV_{\mathrm{fillM}}&\!\!:= \!\!
 \begin{cases}
 J\,\mathrm{span}\{B_2, B_3,B_4,B_4 x,\\
  \hspace{1.2cm}(x y)^2(1-x),(x y)^2(1-y),\\
  \hspace{1.2cm}(x y)^2(1-x)(1-y),\\
  \hspace{1.2cm}(x \exp^{1-y} y)^2(1-x)^2,(\exp^{1-x}x  y)^2(1-y)^2\} & \mbox{ if } k = 2,
\\
  J\,\mathrm{span}\{B_2, B_3,B_4,B_4 x,\\
  \hspace{1.2cm}(x y)^2(1-x)^{2},(x y)^2(1-x)(1-y), (x y)^2(1-y)^{2},\\
 \hspace{1.2cm}(x y)^2(1-x)^{2}(1-y),(x y)^2(1-x)(1-y)^{2},\\  
 \hspace{1.2cm}(x \exp^{1-y}y)^2(1-x)^{3},(\exp^{1-x}x y)^2(1-y)^{3}\} & \mbox{ if } k = 3,\\
   J\,\mathrm{span}\{B_2, B_3,B_4,B_4 x,\\
  \hspace{1.2cm}(x y)^2(1-x)^{k-1},(x y)^2(1-x)^{k-2}(1-y),\\
\hspace{1.2cm}  (x y)^2(1-y)^{k-1},(x y)^2(1-x)(1-y)^{k-2},\\
 \hspace{1.2cm}(x y)^2(1-x)^{k-1}(1-y),(x y)^2(1-x)(1-y)^{k-1},\\  
 \hspace{1.2cm}(x \exp^{1-y}y)^2(1-x)^{k},(\exp^{1-x}x y)^2(1-y)^{k}\} & \mbox{ if } k \ge 4.
   \end{cases}
\end{align*}
 
\section{Proof of Theorem \ref{thm:polygon}}
\label{sec:proof}
In this section, we prove Theorem \ref{thm:polygon}, which is the main result of  Section \ref{sec:main}.
we proceed by carrying out a systematic construction of the spaces $\delta\VV_{\mathrm{fillM}}$
for the trace space $\bld M=\ppol_k(\partial K)$ on a general polygon $K$.
We begin by developing an algorithm that, given 
a counter-clockwise ordering of the $ne$ edges of $K$, 
$\{\eg_i\}_{i=1}^{ne}$,
and an initial space $\VVV\times \WWW$ satisfying the inclusion properties (a) and (b), and 
$\ppol_1\subset \WWW$, 
provides a space $\delta\VV_{\mathrm{fillM}}$ 
satisfying the properties in Table \ref{table:deltaspaces}. 
We then apply it to show that the space $\dvm_{\mathrm{fillM}}$ in Theorem \ref{thm:polygon} satisfies 
the properties in Table \ref{table:deltaspaces}.
We end the proof by showing that the space $\dvm_{\mathrm{fillV}}$ in Theorem \ref{thm:polygon}
also satisfies the related properties in Table \ref{table:deltaspaces}.


\subsection{An algorithm to construct the space $\delta\VV_{\mathrm{fillM}}$}
We use the notation introduced in the previous section. 
For $i=1, \dots, ne+1$, we define 
$\VVV_{s,i}$ to be the divergence-free subspace of $\VVV$ with vanishing normal
traces on the first $i-1$ edges. In other words, we set
\begin{alignat*}{2}
 \VVV_{s,i} :=&\;\{\und \tau\in \VVV : \;\divv \und \tau =  \bld 0, \;\und \tau\,\n|_{\eg_j}=\bld 0
 , \; 1\le j\le i-1\}, &&\text{ for }1 \le i\le ne+1.
\end{alignat*}
The subspace of $\WWW$ given by 
$\WWW_{\!rm} =\{\w\in \WWW : \;\und\epsilon(\w) = \und { 0}\}$
also plays an important role in the theory of \mm-decompositions; 
see the kernels' trace decomposition  in Theorem \ref{thm:1.5}. 
Since $\ppol_1\subset \WWW$, we have that $\WWW_{\!rm}=\bld{RM}(K)$ 
is just the space of rigid motions on $K$, which has dimension $3$.

For  $i=1, \dots, ne$, we define $\gamma_i(\VV):=\{\und \tau \n|_{\eg_i}:\;\und \tau\in \VV\}$
to be the normal trace of  $\VV$ on $\eg_i$, and 
$\gamma_i(\W):=\{\w|_{\eg_i}:\;\w\in \W\}$ to be the trace of  $\W$ on $\eg_i$. 
We have 
\[
 \dim \gamma_i(\WWW_{\! rm}) =  \dim  \gamma_{i}(\bld{RM}(K)) = 3.
\]

Now, we define the \mm-index  for each edge.
\begin{definition} [The \mm-index for each edge] 
The \mm-index of the space $\VVV\times \WWW$ 
for the $i$-th edge $\eg_i$ is the number
\[
 I_{\bld M,i}(\VVV\times \WWW):= \dim \bld M(\eg_i) -  \dim \gamma_i(\VVV_{s,i}) - 
\delta_{i,ne} \dim \gamma_{ne}(\WWW_{\!rm}),
\]
where $\delta_{i,ne}$ is the Kronecker delta.
\end{definition}

Since $\VVV\times \WWW$ satisfies the inclusion properties of an \mm-decomposition, 
we have 
\begin{align*}
  \gamma_i(\VVV_{s,i})\subset&\; \bld M(\eg_i)\quad\quad \text{ for all } 1\le i\le ne-1,\\
  \gamma_{ne}(\VVV_{s,ne}) +  \gamma_{ne}(\WWW_{\!rm})
  \subset&\; \bld M(\eg_{ne}).
\end{align*}
Actually, the sum in the last inclusion is an ($L^2(\eg_{ne})$-orthogonal) direct sum because, given any 
$(\und \tau,\w)\in \VVV_{s,ne}\times \WWW_{\!rm}$,
we have
\begin{align*}
\langle\gamma_{ne}\und \tau, \gamma_{ne}\w\rangle_{\eg_{ne}} =
\langle\und \tau\n, \w\rangle_{\eg_{ne}}
=
\langle\und \tau\n, \w\rangle_{\dK}
=
(\und \tau, \und\epsilon( \w))_K+(\divv\und \tau, \w)_K 
= 
0.
\end{align*}
Using these facts, we immediately get that $I_{\bld M,i}(\VVV\times \WWW)$  is a natural number for any $1\le i\le ne$.

We are now ready to state our first result.
\begin{theorem}
\label{thm:algorithm}
Set $\dvm_{\mathrm{fillM}} := \oplus_{i=1}^{ne}\delta\VV_{\mathrm{fillM}}^i$ where
\begin{itemize}
\item[{\rm($\alpha$)}]  $\gamma (\delta\VV_{\mathrm{fillM}}^i)\subset \bld M$,
\item[{\rm($\beta$) }] $\divv \delta\VV_{\mathrm{fillM}}^i=\{{0}\}$,
\item[{\rm($\gamma.1$)}]    $\gamma_j (\delta\VV_{\mathrm{fillM}}^i) = \{ 0\}$, for 
     $1\le j\le i-1$,
\item[{\rm($\gamma.2$)}]  $\gamma_i(\VVV_{s,i}) \cap 
     \gamma_i(\delta\VV_{\mathrm{fillM}}^i)
     =\{\bld 0\}$,
\item[{\rm($\delta$)}] 
$\dim\delta\VV_{\mathrm{fillM}}^i=\dim \gamma_i(\delta\VV_{\mathrm{fillM}}^i)
     =  I_{\bld M,i}(\VVV\times \WWW)$.
\end{itemize}
Then
$\delta\VV_{\mathrm{fillM}}$ satisfies the properties in Table \ref{table:deltaspaces},
that is,
\begin{itemize}
\item[{\rm (a)}] $\gamma \delta\VV_{\mathrm{fillM}}\subset\bld M$,
\item[{\rm (b)}] $\divv \delta \VV_{\mathrm{fillM}}=\{ 0\}$,
\item[{\rm (c)}] $\gamma\VVV_{s,1} \cap \gamma \delta \VV_{\mathrm{fillM}}=\{ {0}\}$,
\item[{\rm (d)}] $\mathrm{dim}\, \delta \VV_{\mathrm{fillM}}=\mathrm{dim}
\,\gamma \delta \VV_{\mathrm{fillM}}=I_{\bld M}(\VVV\times\WWW)$.
\end{itemize}

\end{theorem}


\begin{proof}
Properties (a), (b) and (c) follow directly form properties ($\alpha$), ($\beta$) and ($\gamma$), respectively. It remains to prove property (d). But, we have 
\begin{align*}
 \dim \delta\VV_{\mathrm{fillM}}
=&\;
\sum_{i=1}^{ne} 
\dim \delta\VV_{\mathrm{fillM}}^i
=\sum_{i=1}^{ne} I_{\bld M,i}(\VVV\times \WWW)\\
=&\;
\sum_{i=1}^{ne} 
\dim \gamma_i\delta\VV_{\mathrm{fillM}}^i
=
\dim \gamma\delta\VV_{\mathrm{fillM}}.
\end{align*}
Now, by the definition of $I_{\bld M,i}(\VVV\times \WWW)$, we get
\begin{alignat*}{1}
 \dim \delta\VV_{\mathrm{fillM}}
=&
\sum_{i=1}^{ne}\big(\dim \bld M(\eg_i) -  \dim \gamma_i(\VVV_{s,i}) - 
\delta_{i,ne} \dim \gamma_{ne}(\WWW_{\!rm})\big)
\\
 =&\;\dim \bld M-
\sum_{i=1}^{ne} \dim \gamma_i(\VVV_{s,i}) -\dim \gamma_{ne}(\WWW_{\!rm})
\\
=&\;\dim \bld M-
\sum_{i=1}^{ne} 
(\dim \VVV_{s,i}- \dim \VVV_{s,i+1})
-\dim \gamma_{ne}(\WWW_{\!rm})
\\
=&\;\dim \bld M-
(\dim \VVV_{s,1}- \dim \VVV_{s,ne+1})
-\dim \gamma_{ne}(\WWW_{\!rm}).
\end{alignat*}
Finally, since 
\begin{alignat*}{1}
\VVV_{s,1}&:=\{\und \tau\in \VVV : \;\divv \und \tau =  \bld 0\},
\\
\VVV_{s,ne+1}&:=\{\und \tau\in \VVV : \;\divv \und \tau =  \bld 0, \und \tau\bld n|_{\partial K}=\bld 0\},
\\
\WWW_{\!rm}&=\{\w\in \WWW:\; \und\epsilon(\w) = \und 0\},
\end{alignat*}
we get
\begin{alignat*}{1}
 \dim \delta\VV_{\mathrm{fillM}}  
 =&\;\dim \bld M
-
\dim \{\und \tau\bld n|_{\partial K}: \und \tau\in \VVV, \divv \und \tau =  \und 0\}\\
&\; -
\dim\{\w|_{\dK}:  \w\in \WWW, \und\epsilon(\w) = \und 0\}
\\
=&\;I_{\bld M}(\VVV\times \WWW).
\end{alignat*}
This completes the proof.
\end{proof}

Based on this result, we can see that the following algorithm provides a practical construction of 
the filling space  $\delta\VV_{\mathrm{fillM}}$.


\newenvironment{varalgorithm}[1]
    {\algorithm\renewcommand{\thealgorithm}{#1}}
    {\endalgorithm}
\renewcommand{\algorithmicrequire}{\textbf{Input:}}
\renewcommand{\algorithmicensure}{\textbf{Output:}}

\begin{varalgorithm}{PC}
 \caption{Construction of $\delta\VV_{\mathrm{fillM}}$ 
 satisfying properties $(\alpha)$--$(\delta)$ of Theorem \ref{thm:algorithm}.}
\label{algorithm1}
 \begin{algorithmic}
 \Require 
 A counter-clockwise ordering of the $ne$ edges of the polygon $K$, $\{\eg_i\}_{i=1}^{ne}$.
\Require The space of traces $\bld M$.
 \Require A space  $\VVV\times \WWW$ satisfying the inclusion properties
 of an \mm-decomposition.
 \Ensure The space $\delta\VV_{\mathrm{fillM}}$.
 \vskip.2truecm
 \vskip.2truecm
\State For each $i = 1,\cdots, ne$,
\State (1) Find the auxiliary spaces $\VVV_{s,i}$. 
\State (2) Find an $I_{\bld M,i}(\VVV\times \WWW)$-dimensional complement space 
$C_{\bld M,i}$ on edge $\eg_i$:
\begin{alignat*}{2}
 \gamma_i(\VVV_{s,i}) \oplus C_{\bld M,i}=&\;\widetilde{\bld M}(\eg_i),
\end{alignat*}
\hspace{.6cm}here $\widetilde{\bld M}(\eg_i)={\bld M}(\eg_i)$ if $i< ne$, and 
$\widetilde{\bld M}(\eg_{ne})=\gamma_{ne}(\WWW_{\!rm})^\perp$ is the subspace of \\
\hspace{.6cm}$\bld M(\eg_{ne})$ that is $L^2(\eg_{ne})$-orthogonal to
$\gamma_{ne}(\WWW_{\!rm})$.
\State (3) Find an $I_{\bld M,i}(\VVV\times \WWW)$-dimensional, divergence-free filling space $\delta\VV_{\mathrm{fillM}}^i$ on $K$:
\begin{align*}
&(3.1)\;\;\gamma_j(\delta\VV_{\mathrm{fillM}}^i)= \{\bld 0\},\;\;\;\;\;\text{ for } 1\le j\le i-1,\\
&(3.2)\;\;\gamma_i(\delta\VV_{\mathrm{fillM}}^i)= C_{\bld M,i},\hspace{7cm}\\
&(3.3)\;\;\gamma_j(\delta\VV_{\mathrm{fillM}}^i)\subset \bld M(\eg_j),\,\text{ for } i+1\le j\le ne.
\end{align*}
\\
\Return $\delta\VV_{\mathrm{fillM}} := \oplus_{i=1}^{ne} 
\delta\VV_{\mathrm{fillM}}^i$.
\end{algorithmic}
\end{varalgorithm}

Now, we apply Algorithm \ref{algorithm1} to prove the first part of 
Theorem \ref{thm:polygon}, that is,
the space $\delta\VV_{\mathrm{fillM}}$ satisfies the 
properties in Table \ref{table:deltaspaces}.
Note that in this case, we have 
$\bld M=\ppol_k(\partial K)$ and $\VVV\times \WWW = \bppol_k\times\ppol_k$
with $k\ge 1$. We proceed in three steps as follows.

\

{\bf (1). Finding the spaces  $\VVV_{s,i}$.}
We begin by characterizing the spaces $\VVV_{s,i}$.
\begin{proposition}
\label{lemma:characterization-s}
We have that 
\[
 \VVV_{s,i} = J\,\Phi_i,\qquad {\color{red}1\le i \le ne+1},
\] 
where 
$
\Phi_i:= \{b_{i-1}^2\phi_i:\quad \phi_i\in \pol{k+4-2i}{K}\}.
$
Here $b_{0} = 1$, and $b_\ell : = \Pi_{j=1}^{\ell}\lambda_j$ for $\ell\ge 1$.
\end{proposition}
To prove this result,
we need to characterize the kernel of the operator $\gamma_i J$.

\begin{lemma}
\label{lemma:gJ}
We have that
$
 \gamma_i(J\,\phi) =\bld  0$ if and only if $\grads\phi|_{\eg_i}\in \ppol_0(\eg_i),$
 for any $\phi\in H^2(K)$ and any edge $\eg_i$ of the element $K$.
\end{lemma}
\begin{proof}
The result follows form the fact that 
$\gamma_i(J\,\phi) = \frac{\partial {\bld{\mathrm{curl}}}\,\phi}{\partial \bld t_i}$,
where $\frac{\partial }{\partial \bld t_i}$ is the tangential derivative on 
the edge $\eg_i$.
\end{proof}

We are now ready to prove Proposition \ref{lemma:characterization-s}.

\begin{proof}[Proof of Proposition \ref{lemma:characterization-s}]
Since $\VVV = \bppol_k$, it is easy to show that
\[
J\,\Phi_i\subset\VVV_{s,i} \subset J\,\pcol_{k+2}.
\]
Since $\Phi_1 = \pcol_{k+2}$, the reverse inclusion,  
$ \VVV_{s,i}\subset J\,\Phi_i$, is true for $i = 1$.
Let us prove that the reverse inclusion also holds for $i\ge 2$.
Let $\und \tau= J\,\phi \in \VVV_{s,i}$  with $\phi\in \pcol_{k+2}$. We have 
$\gamma_j(J\,\phi) = \bld 0$ for $1\le j \le i-1$. By Lemma \ref{lemma:gJ}, 
$\grads\phi|_{\eg_i}\in \ppol_0(\eg_i)$ for $1\le j \le i-1$.
Since $\phi$ is defined up to a linear function, 
we can assume $\grads\phi|_{\eg_1} = \bld 0$, hence $\lambda_1^2$ {\color{red} divides} $\phi$. 
This immediately implies  
$\grads\phi|_{\eg_j} = \bld 0$ for $1\le j \le i-1$, 
and so $b_{i-1}^2$ { divides} $\phi$.
This completes the proof.
\end{proof}

\

{\bf (2). Finding the complement spaces $C_{\bld M,i}$}.

By definition, see Algorithm \ref{algorithm1},
the space $C_{\bld M,i}$ is any subspace of $\widetilde{\bld M}(\eg_i)$ such that $
\gamma_i( \VVV_{s,i} )\oplus C_{\bld M,i}=\widetilde {\bld M}(\eg_i).$  Thus, to find
a choice of $C_{\bld M,i}$, which is not necessarily unique, we first need to
to characterize $\gamma_i( \VVV_{s,i})$.
We do that in the following corollary of the previous proposition.
\begin{corollary}
 \label{corollary-dim-count}  We have, for $1\le i\le ne$,
 \begin{align*} 
  \gamma_i( \VVV_{s,i} )=&\;\mathrm{span}
\left  \{ \gamma_i\left(J(b_{i-1}^2\lambda_{i+1}^a)\right)\right\}_{a=2\delta_{1,i}}^{k+4-2i}
  \oplus \mathrm{span}
\left  \{ \gamma_i\left(J(b_{i-1}^2\lambda_i\lambda_{i+1}^a)\right)\right\}_{a=\delta_{1,i}}^{k+3-2i},
\\
  \dim\gamma_i( \VVV_{s,i} ) =&\; \dim \pcol_{k+4-2i}(\eg_i)
  +\dim \pcol_{k+3-2i}(\eg_i)-3\,\delta_{1,i},
 \\
  I_{\bld M,i}(\VVV\times \WWW) =
  &\;
  \min\left(k+1, 2i-4\right) +
    \min\left(k+1, 2i-3\right)+3\,\delta_{1,i} -3\,\delta_{ne,i}.
 \end{align*}
\end{corollary}
Here  we use the convention that,  for any negative integer $m$, 
$\dim \pcol_{m} = 0$.
\begin{proof}
{\color{red}The first identity follows from the definition of the auxiliary space $\VVV_{s,i}$ and from the fact that
$\gamma_i\left(J(b_{i-1}^2\lambda_i^b\lambda_{i+1}^a)\right)=0$ when $b\ge2$.

Let us now prove the second identity.} By construction,
\[
\dim\gamma_i( \VVV_{s,i})
=\dim  \VVV_{s,i}-\dim \VVV_{s,i+1},
\]
and since, by Proposition \ref{lemma:characterization-s},
$\dim \VVV_{s,i} =\;   \dim  \pcol_{k+4-2i}{(K)}-3\,\delta_{1,i},$ we get that
\begin{alignat*}{1}
\dim\gamma_i( \VVV_{s,i})&
= (\dim \pcol_{k+4-2i}{(K)}-3\,\delta_{1,i})- (\dim \pcol_{k+2-2i}{(K)}-3\,\delta_{1,i+1})
\\
&
= \dim \pcol_{k+4-2i}(\eg_i)
  +\dim \pcol_{k+3-2i}(\eg_i)-3\,\delta_{1,i}.
\end{alignat*}

{\color{red}It remains to prove the last identity. By the definition of $I_{\bld M,i}(\VVV\times \WWW)$, we have
\begin{alignat*}{1}
I_{\bld M,i}(\VVV\times \WWW)
= &\;\dim \bld M(\eg_i) -  \dim \gamma_i(\VVV_{s,i}) - 
\delta_{i,ne} \dim \gamma_{ne}(\WWW_{\!rm})
\\
=&\; 2 \dim \pcol_{k}{(\eg_i)} -  \Big(\dim \pcol_{k+4-2i}(\eg_i)
\\
&\;
  +\dim \pcol_{k+3-2i}(\eg_i)-3\,\delta_{1,i}\Big) - 3\,
\delta_{i,ne}
\\
=&\; (\dim \pcol_{k}{(\eg_i)} -  \dim \pcol_{k+4-2i}(\eg_i))
 \\&+(\dim \pcol_{k}{(\eg_i)} -\dim \pcol_{k+3-2i}(\eg_i)+3\,\delta_{1,i} - 3\,
\delta_{i,ne}
\end{alignat*}
and the result follows.} This completes the proof.
\end{proof}

We now give a particular choice of the trace space $C_{\bld M,i}$.
\begin{proposition}
 Set, for $i=1,\dots,ne$,
  \[
  C_{\bld M,i} =\left\{\begin{tabular}{l l}
            $\{0\}$ & if $i = 1$,           
            \vspace{.3cm}
\\
            $\mathrm{span}
\left  \{ \gamma_i\left(J(\eta_i^2\lambda_{i+1}^b)\right)\right\}_{b=\max\{k+5-2i,0\}}^{k}$\\
$ \hspace{.2cm}\oplus \mathrm{span}
\left  \{ \gamma_i\left(J(\eta_i^2\lambda_i\lambda_{i+1}^b)\right)\right\}_{b=
\max\{k+4-2i,0\}}^{k-1}$\\
$ \hspace{.4cm}\oplus \mathrm{span}
\left  \{ \gamma_i\left(J(\eta_i\lambda_i\lambda_{i+1})\right)\right\}
$
            & if $2 \le i\le ne-1$,
            \vspace{.3cm}
            \\
                        $\mathrm{span}
\left  \{ \gamma_i\left(J(\eta_i\lambda_{i}\lambda_{i+1})\right)\right\}
$
&if $i = ne$ and $k = 1$,\vspace{.3cm}\\
            $\mathrm{span}
\left  \{ \gamma_i\left(J(\eta_i^2\lambda_{i+1}^{2+b})\right)\right\}_{b=\max\{k+5-2i,0\}}^{k-2}$\\
$ \hspace{.2cm}\oplus \mathrm{span}
\left  \{ \gamma_i\left(J(\eta_i^2\lambda_{i}\lambda_{i+1}^{2+b})\right)\right\}_{b=
\max\{k+4-2i,0\}}^{k-3}$\\
$ \hspace{.4cm}\oplus \mathrm{span}
\left  \{ \gamma_i\left(J(\eta_i\lambda_{i}\lambda_{i+1})\right),
\gamma_i\left(J(\eta_i\lambda_{i}\lambda_{i+1}^2)\right)\right\}
$
            & if $i = ne$ and $k \ge 2$,
            \end{tabular}\right.
 \]
where $\eta_i$ is any linear function on $\mathbb{R}^2$ 
 such that  $\eta_i(\vt_{i}) = 0$ and $\eta_i(\vt_{i+1}) \not = 0$. 
 
 Then,  for $i=1,\dots,ne$,
 the space $C_{\bld M,i}$ of 
 functions defined on the edge $\eg_i$ has dimension $I_{\bld M,i}(\VVV\times \WWW)$
 and satisfies the identity 
 \[
  \gamma_i( \VVV_{s,i} )\oplus C_{\bld M,i}=\widetilde {\bld M}(\eg_i).
  \]
\end{proposition}
\begin{proof}
Since $\eta_i$ is a linear function, 
it is easy to check that $\dim C_{\bld M,i} = I_{\bld M,i}(\VVV\times \WWW)$ and 
$C_{\bld M,i}\subset \widetilde{M}(\eg_i)$. We are left to show that
$\gamma_{i}(\VVV_{s,i})\cap C_{\bld M,i} = \{\bld 0\}$. We prove this result for the case
$2\le i\le ne-1$
and 
$k\ge 2i-3$. The other cases are similar and simpler.

To show $\gamma_{i}(\VVV_{s,i})\cap C_{\bld M,i} = \{\bld 0\}$, we only need to 
prove the linear independence of the following five sets
\begin{alignat*}{2}
&\mathrm{span}
\left  \{ \gamma_i\left(J(b_{i-1}^2\lambda_{i+1}^a)\right)\right\}_{a=0}^{k+4-2i},
&&\quad\mathrm{span}
\left  \{ \gamma_i\left(J(b_{i-1}^2\lambda_i\lambda_{i+1}^a)\right)\right\}_{a=0}^{k+3-2i},\\
&\mathrm{span}
\left  \{ \gamma_i\left(J(\eta_i^2\lambda_{i+1}^b)\right)\right\}_{b=k+5-2i}^{k},
&&\quad \mathrm{span}
\left  \{ \gamma_i\left(J(\eta_i^2\lambda_i\lambda_{i+1}^b)\right)\right\}_{b=
k+4-2i}^{k-1},\\
&\mathrm{span}
\left  \{ \gamma_i\left(J(\eta_i\lambda_i\lambda_{i+1})\right)\right\}.
\end{alignat*}
Note that the first two sets span a set of bases for $\gamma_{i}(\VVV_{s,i})$
and the last three sets span a set of bases for $C_{\bld M,i}$.
Let us assume that there exists constants $\{C_a\}_{a=0}^{k+4-2i}$,
$\{D_a\}_{a=0}^{k+3-2i}$, $ \{E_b\}_{a=k+5-2i}^{k}$,
$\{F_b\}_{b=k+4-2i}^{k-1}$, and $G$ such that 
\[
\gamma_{i}( J \phi)=0,
\]
where
\begin{align*}
\phi
:=
&\sum_{a=0}^{k+4-2i}C_a\,b_{i-1}^2\lambda_{i+1}^a
 +\sum_{a=0}^{k+3-2i}D_a\,b_{i-1}^2\lambda_{i}\lambda_{i+1}^a
 \\
 &+\sum_{b=k+5-2i}^{k}E_b\,\eta_{i}^2\lambda_{i+1}^b
 +\sum_{b=k+4-2i}^{k-1}F_b\,\eta_{i}^2\lambda_i\lambda_{i+1}^b
 +G\,\eta_{i}\lambda_i\lambda_{i+1}.
\end{align*}
By Lemma \ref{lemma:gJ}, this implies that 
$\grads\phi |_{\eg_i}\in \ppol_0(\eg_i)$ and so that $\phi |_{\eg_i}\in \pcol_1(\eg_i)$. As a consequence,
\[
\left.\left(\sum_{a=0}^{k+4-2i}C_a\,b_{i-1}^2\lambda_{i+1}^a 
 +\sum_{b=k+5-2i}^{k}{E_b}\,\eta_i^2\lambda_{i+1}^{b}\right)\right |_{\eg_i} \in \pcol_1(\eg_i),
\]
because $\lambda_i=0$ on $\eg_i$.
Since $b_{i-1}(\vt_i)=0$ (because $i\ge2$) and since $\eta_i(\vt_i)=0$, we have $\eta_i|_{\eg_i}$ is proportional to 
$\lambda_{i-1}|_{\eg_i}$ and we get that
\[
\left.\left(\sum_{a=0}^{k+4-2i}C_a\,b_{i-1}^2\lambda_{i+1}^a 
 +\sum_{b=k+5-2i}^{k}{E_b}\,\eta_i^2\lambda_{i+1}^{b}\right)\right |_{\eg_i} =0.
\]
Now, evaluating the expression at the node $\vt_{i+1} = \eg_i\cap \eg_{i+1}$, we get
$C_0 = 0$ since $b_{i-1}(\vt_{i+1})\not = 0, \eta_i(\vt_{i+1})\not=0$ and $\lambda_{i+1} (\vt_{i+1}) = 0$.
Then, dividing it by $\lambda_{i+1}$ and evaluating the resulting expression
again at $\vt_{i+1} = \eg_i\cap \eg_{i+1}$, we get $C_1=0$. Similarly, we get 
$C_a = 0$ for $a=2,\cdots, k+4-2i$, and $E_b = 0$ for $b = k+5-2i,\cdots, k$. 
This implies that
\[
\phi=\sum_{a=0}^{k+3-2i}D_a\,b_{i-1}^2\lambda_{i}\lambda_{i+1}^a
 +\sum_{b=k+4-2i}^{k-1}F_b\,\eta_{i}^2\lambda_i\lambda_{i+1}^b
 +G\,\eta_{i}\lambda_i\lambda_{i+1},
\]
and so, that $\grads\phi |_{\eg_i}=\varphi\,\grads\lambda_i$,
where
\[
\varphi:=\sum_{a=0}^{k+3-2i}D_a\,b_{i-1}^2\lambda_{i+1}^a
 +\sum_{b=k+4-2i}^{k-1}F_b\,\eta_{i}^2\lambda_{i+1}^b
 +G\,\eta_{i}\lambda_{i+1}.
\]
Since $\varphi |_{\eg_i}\in \pol{0}{\eg_i}$ and $\varphi(\vt_i)=0$, because
$b_{i-1}(\vt_{i})= 0$ and $\eta_i(\vt_{i+1})=0$, we conclude that
$\varphi  |_{\eg_i}=0$, that is, that
\begin{align*}
\left.\left(
   \sum_{a=0}^{k+3-2i}D_a\,b_{i-1}^2\lambda_{i+1}^a
 +\sum_{b=k+4-2i}^{k-1}F_b\,\eta_{i}^2\lambda_{i+1}^b
 +G\,\eta_{i}\lambda_{i+1}
 \right)\right|_{\eg_i}=0.
\end{align*}
Since $\eta_i(\vt_{i+1})=0$, $\eta_i=\alpha\lambda_{i-1}|_{\eg_i}$ for some number $\alpha$. Then, dividing the above expression $\lambda_{i-1}$ and evaluating the resulting 
 expression at $\vt_i$, we obtain that $G=0$.  Finally, we can get  that $D_a=0$ and $F_b=0$ by 
 consecutively evaluating the expression at $\vt_{i+1}$ and dividing it by $\lambda_{i+1}$.
This completes the proof.
\end{proof}

\

{\color{red}
{\bf (3). Finding the filling spaces $\delta\VV_{\mathrm{fillM}}^i$}.
Note that the definition of $\Psi_i$ in Theorem \ref{thm:polygon}
is obtained from our choice of the space $C_{\bld M,i}$ by formally replacing, 
in the definition of the basis of $C_{\bld M,i}$, 
$\eta^2_i$ by $\xi^2_{i+1}$ and then $\eta_i\lambda_i\lambda_{i+1}$ by $B_i$. 
The fact that the choice $\delta\VV_{\mathrm{fillM}}^i:=J\,\Psi_i$ does satisfy the conditions (3.1), (3.2) and (3.3) of Algorithm PC follows immediately from the following results.

\begin{lemma}
\label{lemma:xi}
Let $\psi$ be any function in 
$\pol{k}{K}$. Then we have that
\begin{itemize}
\item[{\rm (i)}] 
$\gamma_j\big(J (\xi^2_{i+1}\psi)\big)=\bld 0$ for $j=1,\dots,i-1$ provided 
$2\le i<ne$, or $i=ne$ and $\psi$ 
is divisible by $\lambda^2_1$.
\item[{\rm (ii)}] $\gamma_i\big(J (\xi^2_{i+1}\psi)\big)=
\gamma_j\big(J( \eta^2_i\psi)\big)$ for some linear 
function $\eta_i$ such that $\eta_i(\vt_i)=0$.
\item[{\rm (iii)}] $\gamma_j\big(J( \xi^2_{i+1}\psi)\big)
\in\ppol_{k}(\eg_j)$ for $j=i+1,\dots,ne$.
\end{itemize}
\end{lemma}

\begin{lemma}
\label{lemma:Bi}
Let $\psi$ be any function in $H^2(K)$. Then we have that
\begin{itemize}
\item[{\rm (i)}] $\gamma_i\big(J(B_{i}\psi)\big)
=\gamma_j\big(J(\alpha\eta_i\lambda_i\lambda_{i+1}\psi)\big)
$ for some 
constant $\alpha$ and some linear function $\eta_i$ such that $\eta_i(\vt_i)=0$.
\item[{\rm (ii)}] $\gamma_j\big(J(B_{i}\psi)\big) = \bld 0$ for $1\le j\le ne$ and $j\neq i$.
\end{itemize}
\end{lemma} 

\

\begin{proof}[Proof of Lemma \ref{lemma:xi}] Let us begin by proving (i). By properties (L.1) and (L.2) in Section 4.1, $\xi_{i+1}=0$ on $\eg_j$ for $j=1,\dots,i-1$ if $i<ne$. Since this implies that $\nabla(\xi_{i+1}^2\,S)|_{\eg_j}=0$, property (i) follows from Lemma \ref{lemma:gJ} for $i<ne$. If $i=ne$, we have, by properties (L.1) and (L.2), that
$\xi_{i+1}=0$ on $\eg_j$ for $j=2,\dots,i-1$ and property (i) follows  
in the same manner. It remains to consider the case $i=ne$ and $j=1$. In this case, on $\eg_1$, 
$\xi_{ne+1}$ is different from zero. As a consequence, property (i) holds if $S$ is divisible by $\lambda_1^2$. This proves property (i).

Let us now prove property (ii).
We can take $\eta_i$ such that 
$\eta_i|_{\eg_i}=\xi_{i+1}|_{\eg_i}$ and 
$\frac{\partial}{\partial\n_i}\eta_i =\frac{\partial}{\partial \n_i} \xi_{i+1}$. 
This is possible by properties (L).
Then,
we have
\begin{alignat*}{1}
\left.\left(\nabla((\xi^2_{i+1}-\eta^2_i)\psi)\right)
\right|_{\eg_i}
& = 
\left.\big(\psi(\xi_{i+1}+\eta_i)\nabla(\xi_{i+1}-\eta_i)
\big)
\right|_{\eg_i} = 0.
\end{alignat*}

Finally, property (iii) follows by simple manipulations and using properties (L.1) and (L.3).  
This completes the proof of Lemma \ref{lemma:xi}
\end{proof}

\begin{proof}[Proof of Lemma \ref{lemma:Bi}]
We first prove prove property (i). Since $\lambda_i=0$ on $\eg_i$ and $B_i=0$ on $\eg_i$, 
by property (H.1) in Section 4.1, we have that, 
\begin{alignat*}{2}
\nabla (B_i \psi-\alpha\eta_i\lambda_i\lambda_{i+1})|_{\eg_i}
&= \n_i\,\left.\left(\psi\,(\frac{\partial}{\partial \n_i} B_i 
- \alpha\,\eta_i\lambda_{i+1}\frac{\partial}{\partial \n_i}\lambda_i)\right)\right|_{\eg_i}
\\
&= \n_i\,\left.\left(\psi\,(\lambda_{i-1}\lambda_{i+1}
- \alpha\,\eta_i\lambda_{i+1}\frac{\partial}{\partial \n_i}\lambda_i)\right)\right|_{\eg_i}
&&\;\mbox{ by (H.3)},
\\
&= \n_i\,\left.\left(\psi\,\lambda_{i-1}\lambda_{i+1} (1
- \alpha\,\eta_i(\vt_{i+1})\frac{\partial}{\partial \n_i}\lambda_i)\right)\right|_{\eg_i}
\\
&=0,
\end{alignat*}
if we take $\alpha$ as the unique solution of the equation  
$\alpha \eta_i(\vt_{i+1})\frac{\partial}{\partial \n_i}\lambda_i=1$ on the edge $\eg_i$. Property (i) now follows from Lemma \ref{lemma:gJ}.

It remains to prove property (ii). We have, by property (H.1) of $B_i$, that
\begin{alignat*}{2}
\nabla (B_i \psi)|_{\eg_j}= \n_j\,\left.\left(\psi\,\frac{\partial}{\partial \n_i} B_i\right)\right|_{\eg_j}=0,
\end{alignat*}
by property (H.2). Property (i) now follows from Lemma \ref{lemma:gJ}.This
completes the proof of Lemma \ref{lemma:Bi}
\end{proof}

With these results, we conclude that the choice $\dvm_{\mathrm{fillM}}$ indeed satisfies 
the related
properties in Table \ref{table:deltaspaces}.

{\bf The computation of the dimension of $\delta\VV_{\mathrm{fillM}}$}.
Now, we compute the dimension of $\delta\VV_{\mathrm{fillM}}$.
We have
\begin{alignat*}{1}
\dim \delta\VV_{\mathrm{fillM}}
&
=
\sum_{i=1}^{ne} \dim \delta\VV_{\mathrm{fillM}}^i
=\sum_{i=1}^{ne} I_{\bld M,i}(\VVV\times\WWW)
\\
&=\sum_{i=1}^{ne} (\min\{k+1,2i-4\}+\min\{k+1,2i-3\}+
3\,\delta_{1,i}-3\,\delta_{ne,i})
\\
&=\sum_{i=1}^{ne} (\min\{k+1,2i-4\}+\min\{k+1,2i-3\}+
3\,\delta_{1,i}-3\,\delta_{ne,i}),
\end{alignat*}
by Corollary \ref{corollary-dim-count}. Finally, simple algebraic manipulations give that
\begin{alignat*}{1}
\dim \delta\VV_{\mathrm{fillM}}
&
=\begin{cases}
  2(k+1)\,ne-\frac{(k+3)(k+4)}{2}& \mbox{ if }k\le 2ne -5,
  \vspace{.2cm}\\
  (2ne-5)\,ne& \mbox{ if }  k\ge 2ne -4,
 \end{cases}\\
 &
= 2(\theta+1)ne-\frac{1}{2}(\theta+3)(\theta+4),
\end{alignat*}
where $\theta:=\min\{k,2 ne-4\}$.

\

%

Finally, we conclude the proof of Theorem \ref{thm:polygon} 
by proving that the choices of $\delta\VV_{\mathrm{fillV}}$  
also satisfy the related properties in Table \ref{table:deltaspaces}.
Since  $\VVV \times \WWW = \bppol_k\times \ppol_k$, 
\red{we have $\divv \VVV = \ppol_{k-1}$.}
Hence we have \[I_S(\VVV\times \WWW) = 2(k+1).\]
It is then elementary to prove that the choice of $\delta\VV_{\mathrm{fillV}}$ satisfies
the properties in Table \ref{table:deltaspaces}.
 This completes the proof of Theorem \ref{thm:polygon}.

\section{Numerical results}
\label{sec:n}
In this section, we present numerical results validating 
{the theory in the case of triangular elements}. 
For simplicity, the material is chosen to be isotropic \eqref{isotropic}.
Recall that  the Lam\'e modules $\lambda$ and $\mu$ have the following form in terms of Young's modules 
$E$ and Possion's ratio $\nu$:
\begin{align*}
 \lambda = \frac{E\nu}{(1+\nu)(1-2\nu)}, \quad\mu = \frac{E}{2(1+\nu)}.
\end{align*}



For comparison, we also present the numerical results with 
the HDG method in \cite{SoonCockburnStolarski09,FuCockburnStolarski14}. The method in 
\cite{SoonCockburnStolarski09}, see also \cite{FuCockburnStolarski14},
uses the following local spaces:
\[
  \VV(K) \times \W(K)\times \bld M(\dK)= 
 \pcol_k(K;\mathbb{S})\times \ppol_k(K)\times \ppol_k(\dK).
\]
This space $\VV(K) \times \W(K)$ does {\em not} admit an $\bld M(\dK)$-decomposition.
We denote this method by HDG$_k$.

Our method on triangles {enriches} the local stress space on each element  with 
a rational function space $\delta\und\Sigma_{\mathrm{fillM}}$  
that has dimension $2$ if $k=1$,
and dimension $3$ if $k\ge 2${\color{red}; see 
the discussion following Theorem \ref{thm:polygon}.}
We denote this method by HDG$_k$--M.

For the postprocessing $\bld u_h^*$,
we take 
$\bld V^*(K) :=\bpol{k+1}{K}$ and $\widetilde{\bld V}^*(K) :=\bpol{0}{K}$:
\begin{alignat*}{2}
\left(\gradv \bld u_h^{*},\gradv \bld  w\right)_K=& 
-\left(\bld u_h,\triangle \,\bld  w \right)_K
+\bintK{\widehat{\bld u}_h}{\gradv\bld w\,\n}
&&\quad\forall\;\bld w 
\in{{\bld V}}^{*}(K),
\\
(\bld u_h^{*},\bld r)_K=&\;(\bld u_h,\bld r)_K && \quad \forall \; 
\bld r\in \bpol{0}{K}.
\end{alignat*}

We present the same two test problems considered in \cite{FuCockburnStolarski14}. 
The first test problem is obtained by taking  $E = 1$, $\nu = 0.3$, and choosing  data so 
that the exact solution for the displacement is  $u_1(x,y) = 10\,(y-y^2)sin(\pi \, x)(1-x)(1-\tfrac{y}{2})$ and $u_2(x,y) = 0$ on the domain $\Omega$. 
The second test problem is  obtained by taking $E = 3$, and choosing data so 
that the exact solution is $u_1(x,y) = -x^2 (x-1)^2 y (y-1)(2y-1)$ and $u_2(x,y) = -u_1(y,x)$. 
For the second problem, we also vary the Poisson ratio $\nu$ from $0.3$ to $0.499999$ to show 
that the methods are free of 
volumetric locking.

We  carry out our experiments on uniform triangular meshes obtained
by discretizing the domain $\Omega = (0,1) \times (0,1)$ with triangles of side
$2^{-l}$ as depicted in Fig. \ref{fig:meshes}. And we fix the polynomial degree to be either $k=1$ or $k=2$.

\begin{figure}[ht!]
\centering
\includegraphics[scale=.30]{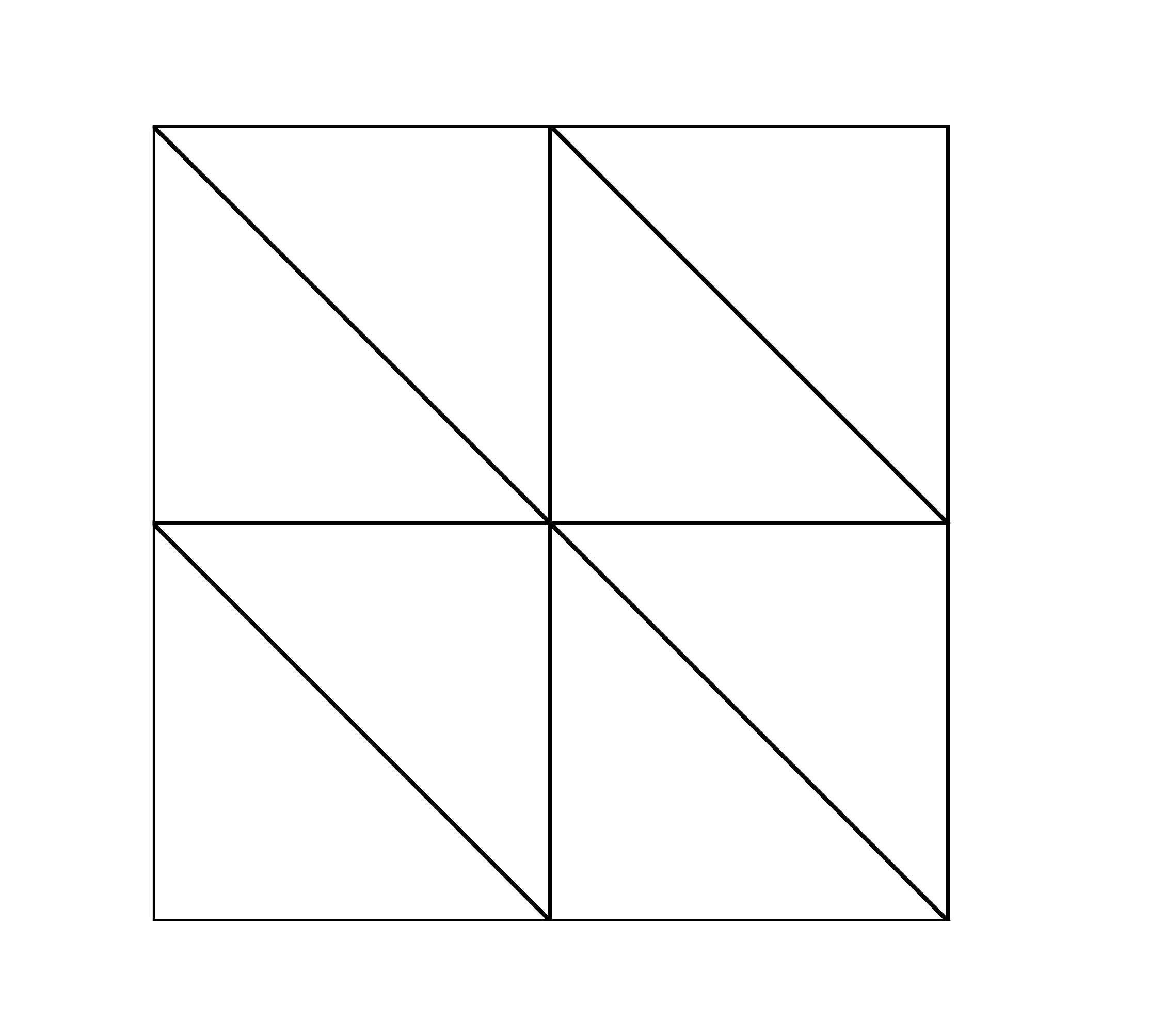}
\caption{Example of meshes with $h$ = $2^{-1}$.}\label{Fi:mesh}
\label{fig:meshes}
\end{figure}

For both methods, we choose the stabilization function $\alpha = Id$.

The history of convergence for the first test is displayed in Table \ref{table:pstress1}, 
and the one for the second test in Table \ref{table:pstress2}. 
The orders of convergence of the HDG$_k$--M method match the
theory developed in Section \ref{sec:error} very well. 
In particular, we get the optimal orders of convergence in the $L^2$-error for 
$\bld u_h, \und\sigma_h$ and $\bld u^*_h$, that is,  $k+1, k+1$ and  $k+2$, respectively.
We also see clearly the superior performance of
HDG$_k$--M over HDG$_k$ for the stress error $\|\und\sigma-\und\sigma_h\|_\Oh$ as well as for 
the postprocessed displacement error $\|\bld u-\bld u_h^*\|_\Oh$. 
Finally, note that, since the global equation for 
both methods have {\em exactly} the same dimension and sparsity structure, 
the HDG methods whose spaces admit \mm-decompositions perform significantly better.

 \def\baselinestretch{1}

\begin{table}[ht]
\caption{History of convergence for the first test.} 
\centering 
\resizebox{\textwidth}{!}
{%
\begin{tabular}{|c|c|c c|c c|c c||c c|c c|c c|}
\hline 
          &  mesh 
          & \multicolumn{2}{c|}{$\Arrowvert \bld{u} - \bld{u}_h \Arrowvert_{\Oh}$} 
	  & \multicolumn{2}{c|}{$\Arrowvert \und{{\sigma}} - \und{{\sigma}}_h \Arrowvert_{\Oh}$}
          & \multicolumn{2}{c||}{$\Arrowvert \bld{u} - \bld{u}_{h}^* \Arrowvert_{\Oh}$}
           & \multicolumn{2}{c|}{$\Arrowvert \bld{u} - \bld{u}_h \Arrowvert_{\Oh}$} 
	  & \multicolumn{2}{c|}{$\Arrowvert \und{{\sigma}} - \und{{\sigma}}_h \Arrowvert_{\Oh}$}
          & \multicolumn{2}{c|}{$\Arrowvert \bld{u} - \bld{u}_{h}^* \Arrowvert_{\Oh}$}\tabularnewline
$k$   & $l$  & error & order & error & order & error & order& error & order & error & order & error & order\tabularnewline
\hline
\multicolumn{2}{|c|}{}&
\multicolumn{6}{c||}{HDG$_k$}&\multicolumn{6}{c|}{HDG$_k$--M} \\
\hline 
   & 3 & 2.10E-2 & -        & 6.00E-2 & -       & 4.25E-3 & -      & 2.06E-2 & -        & 5.35E-2 & -       & 1.89E-3 & -     \tabularnewline 
   & 4 & 5.30E-3 & 1.99  & 1.59E-2 & 1.91 & 1.20E-3 & 1.83& 5.21E-3 & 1.98  & 1.40E-2 & 1.94 & 3.48E-4 & 2.44\tabularnewline 
1 & 5 & 1.33E-3 & 2.00  & 4.22E-3 & 1.91 & 3.27E-4 & 1.88& 1.31E-3 & 1.99  & 3.61E-3 & 1.95 & 5.49E-5 & 2.67\tabularnewline 
   & 6 & 3.32E-4 & 2.00  & 1.13E-3 & 1.90 & 8.68E-5 & 1.91& 3.29E-4 & 1.99  & 9.20E-4 & 1.97 & 7.73E-6 & 2.83\tabularnewline 
   & 7 & 8.31E-5 & 2.00  & 3.07E-4 & 1.88 & 2.26E-5 & 1.94& 8.26E-5 & 2.00  & 2.32E-4 & 1.98 & 1.03E-6 & 2.91\tabularnewline 
 \hline 
   & 3 & 1.25E-3 & -     & 3.65E-3 & -    & 1.59E-4 & -   & 1.25E-3 & -     & 3.30E-3 & -    & 6.22E-5 & -      \tabularnewline 
   & 4 & 1.57E-4 & 2.99  & 4.71E-4 & 2.95 & 2.30E-5 & 2.79& 1.58E-4 & 2.98  & 4.17E-4 & 2.99 & 4.52E-6 & 3.78\tabularnewline 
2 & 5 & 1.97E-5 & 3.00   & 6.06E-5 & 2.96 & 3.19E-6 & 2.85& 1.99E-5 & 2.99  & 5.24E-5 & 2.99 & 3.07E-7 & 3.88\tabularnewline 
   & 6 & 2.46E-6 & 3.00  & 7.82E-6 & 2.95 & 4.25E-7 & 2.91& 2.49E-6 & 3.00  & 6.57E-6 & 3.00 & 2.01E-8 & 3.94\tabularnewline 
   & 7 & 3.08E-7 & 3.00  & 1.02E-6 & 2.94 & 5.53E-8 & 2.94& 3.12E-7 & 3.00  & 8.22E-7 & 3.00 & 1.28E-9 & 3.97\tabularnewline 
\hline
\end{tabular}
}
\label{table:pstress1} 
\end{table}

 \def\baselinestretch{1}

\begin{table}[ht]
\caption{History of convergence for the second test.} 
\centering 
\resizebox{\textwidth}{!}
{
\begin{tabular}{|c|c|c c|c c|c c||c c|c c|c c|}
\hline 
          &  mesh 
          & \multicolumn{2}{c|}{$\Arrowvert \bld{u} - \bld{u}_h \Arrowvert_{\Oh}$} 
	  & \multicolumn{2}{c|}{$\Arrowvert \und{{\sigma}} - \und{{\sigma}}_h \Arrowvert_{\Oh}$}
          & \multicolumn{2}{c||}{$\Arrowvert \bld{u} - \bld{u}_{h}^* \Arrowvert_{\Oh}$}
           & \multicolumn{2}{c|}{$\Arrowvert \bld{u} - \bld{u}_h \Arrowvert_{\Oh}$} 
	  & \multicolumn{2}{c|}{$\Arrowvert \und{{\sigma}} - \und{{\sigma}}_h \Arrowvert_{\Oh}$}
          & \multicolumn{2}{c|}{$\Arrowvert \bld{u} - \bld{u}_{h}^* \Arrowvert_{\Oh}$}\tabularnewline
$k$   & $l$  & error & order & error & order & error & order& error & order & error & order & error & order\tabularnewline
\hline
\multicolumn{2}{|c|}{}&
\multicolumn{6}{c||}{HDG$_k$, $\nu = 0.3$}&\multicolumn{6}{c|}{HDG$_k$--M, $\nu = 0.3$} \\
\hline 
   & 3 & 5.23E-4 & -        & 3.62E-3 & -       & 1.32E-4 & -      & 4.81E-4 & -        & 1.98E-3 & -       & 2.53E-5 & -     \tabularnewline 
   & 4 & 1.39E-4 & 1.91  & 1.28E-3 & 1.51 & 4.46E-5 & 1.57& 1.22E-4 & 1.98  & 5.31E-4 & 1.90 & 4.66E-6 & 2.44\tabularnewline 
1 & 5 & 3.66E-5 & 1.93  & 4.49E-4 & 1.51 & 1.47E-5 & 1.60& 3.06E-5 & 2.00  & 1.39E-4 & 1.94 & 7.82E-7 & 2.57\tabularnewline 
   & 6 & 9.58E-6 & 1.93  & 1.56E-4 & 1.53 & 4.62E-6 & 1.67& 7.66E-6 & 2.00  & 3.57E-5 & 1.96 & 1.16E-7 & 2.75\tabularnewline 
   & 7 & 2.49E-6 & 1.94  & 5.27E-5 & 1.57 & 1.39E-6 & 1.73& 1.91E-6 & 2.00  & 9.06E-6 & 1.98 & 1.59E-8 & 2.87\tabularnewline 
 \hline 
   & 3 & 3.46E-5 & -     & 3.33E-4 & -    & 8.84E-6 & -   & 3.38E-5 & -     & 2.02E-4 & -    & 1.77E-6  & -      \tabularnewline 
   & 4 & 4.58E-6 & 2.92  & 5.01E-5 & 2.72 & 1.47E-6 & 2.59& 4.37E-6 & 2.95  & 2.68E-5 & 2.91 & 1.43E-7  & 3.63\tabularnewline 
2 & 5 & 5.94E-7 & 2.95   & 7.63E-6 & 2.72 & 2.36E-7 & 2.64& 5.51E-7 & 2.99  & 3.42E-6 & 2.97 & 1.02E-8  & 3.81\tabularnewline 
   & 6 & 7.64E-8 & 2.96  & 1.17E-6 & 2.70 & 3.54E-8 & 2.74& 6.92E-8 & 2.99  & 4.31E-7 & 2.99 & 6.82E-10 & 3.90\tabularnewline 
   & 7 & 9.76E-9 & 2.97  & 1.80E-7 & 2.71 & 5.02E-9 & 2.82& 8.66E-9 & 3.00  & 5.40E-8 & 3.00 & 4.41E-11 & 3.95\tabularnewline 
\hline
\multicolumn{2}{|c|}{}&
\multicolumn{6}{c||}{HDG$_k$, $\nu = 0.499$}&\multicolumn{6}{c|}{HDG$_k$--M, $\nu = 0.499$} \\
\hline 
   & 3 & 4.49E-4 & -        & 3.30E-3 & -       & 1.01E-4 & -      & 4.14E-4 & -        & 2.75E-3 & -        & 2.73E-5 & -     \tabularnewline 
   & 4 & 1.20E-4 & 1.91  & 1.18E-3 & 1.49 & 3.51E-5 & 1.52& 1.06E-4 & 1.97  & 5.53E-4 & 2.31 & 2.59E-6 & 3.40\tabularnewline 
1 & 5 & 3.15E-5 & 1.92  & 4.21E-4 & 1.48 & 1.18E-5 & 1.57& 2.67E-5 & 1.99  & 1.15E-4 & 2.27 & 2.64E-7 & 3.29\tabularnewline 
   & 6 & 8.25E-6 & 1.93  & 1.48E-4 & 1.50 & 3.77E-6 & 1.64& 6.69E-6 & 2.00  & 2.79E-5 & 2.04 & 3.31E-7 & 3.00\tabularnewline 
   & 7 & 2.15E-6 & 1.94  & 5.08E-5 & 1.55 & 1.15E-6 & 1.71& 1.68E-6 & 2.00  & 7.30E-6 & 1.93 & 4.66E-9 & 2.83\tabularnewline 
 \hline 
   & 3 & 3.02E-5 & -        & 2.98E-4 & - & 7.03E-6 & -   & 2.96E-5 & -     & 3.39E-4 & -    & 1.63E-6  & -      \tabularnewline 
   & 4 & 3.99E-6 & 2.92  & 4.54E-5 & 2.71 & 1.19E-6 & 2.57& 3.83E-6 & 2.95  & 3.70E-5 & 3.20 & 9.52E-8  & 4.10\tabularnewline 
2 & 5 & 5.18E-7 & 2.95  & 7.08E-6 & 2.68 & 1.94E-7 & 2.62 & 4.85E-7 & 2.98  & 4.02E-6 & 3.20 & 5.77E-9  & 4.04\tabularnewline 
   & 6 & 6.66E-8 & 2.96  & 1.11E-6 & 2.67 & 2.95E-8 & 2.72& 6.08E-8 & 2.99  & 4.53E-7 & 3.15 & 3.56E-10 & 4.02\tabularnewline 
   & 7 & 8.52E-9 & 2.97  & 1.74E-7 & 2.68 & 4.23E-9 & 2.80& 7.62E-9 & 3.00  & 5.31E-8 & 3.09 & 2.21E-11 & 4.01\tabularnewline 
\hline
\multicolumn{2}{|c|}{}&
\multicolumn{6}{c||}{HDG$_k$, $\nu = 0.49999$}&\multicolumn{6}{c|}{HDG$_k$--M, $\nu = 0.49999$} \\
\hline 
   & 3 & 4.49E-4 & -     & 3.30E-3 & -    & 1.01E-4 & -      & 4.13E-4 & -        & 3.62E-3 & -       & 3.44E-5 & -     \tabularnewline 
   & 4 & 1.19E-4 & 1.91  & 1.18E-3 & 1.49 & 3.50E-5 & 1.52& 1.06E-4 & 1.97  & 9.32E-4 & 1.96 & 4.07E-6 & 3.08\tabularnewline 
1 & 5 & 3.15E-5 & 1.92  & 4.21E-4 & 1.48 & 1.18E-5 & 1.57& 2.66E-5 & 1.99  & 2.28E-4 & 2.03 & 4.81E-7 & 3.08\tabularnewline 
   & 6 & 8.25E-6 & 1.93  & 1.48E-4 & 1.50 & 3.77E-6 & 1.64& 6.68E-6 & 2.00  & 5.14E-5 & 2.15 & 5.37E-8 & 3.16\tabularnewline 
   & 7 & 2.15E-6 & 1.94  & 5.08E-5 & 1.55 & 1.15E-6 & 1.71& 1.67E-6 & 2.00  & 1.01E-5 & 2.34 & 5.35E-9 & 3.33\tabularnewline 
 \hline 
   & 3 & 3.02E-5 & -     & 2.98E-4 & -    & 7.02E-6 & -   & 2.95E-5 & -     & 3.65E-4 & -    & 1.72E-6  & -      \tabularnewline 
   & 4 & 3.99E-6 & 2.92  & 4.54E-5 & 2.71 & 1.19E-6 & 2.57& 3.83E-6 & 2.95  & 3.90E-5 & 3.22 & 9.80E-8  & 4.13\tabularnewline 
2 & 5 & 5.17E-7 & 2.95  & 7.08E-6 & 2.68  & 1.93E-7 & 2.62& 4.84E-7 & 2.98  & 4.18E-6 & 3.22 & 5.85E-9  & 4.07\tabularnewline 
   & 6 & 6.66E-8 & 2.96  & 1.11E-6 & 2.67 & 2.95E-8 & 2.71& 6.08E-8 & 2.99  & 4.65E-7 & 3.17 & 3.57E-10 & 4.03\tabularnewline 
   & 7 & 8.51E-9 & 2.97  & 1.74E-7 & 2.68 & 4.23E-9 & 2.80& 7.61E-9 & 3.00  & 5.39E-8 & 3.11 & 2.32E-11 & 3.95\tabularnewline 
\hline
\end{tabular}
}
\label{table:pstress2} 
\end{table}

%

\section{Concluding remarks}
\label{sec:c}

We extended the use of $M$-decomposition for the devising of new superconvergent methods 
for the pure diffusion problems \cite{CockburnFuSayas16} to 
the linear elasticity \red{with symmetric approximate stresses}. 
It provides a {\em simple a priori error analysis} of HDG methods 
for linear elasticity with strong symmetry
and gives us guidelines for the devising of new superconvergent methods. 

We applied the concept of an \mm-decomposition to construct new HDG and (hybridized) 
mixed methods 
with symmetric approximate stresses for linear elasticity in two-space dimensions. 
Numerical results on triangular meshes confirm the theoretical 
convergence properties.

Let us emphasize the fact that it is not necessary to use the compliance matrix
for formulate the methods. The same results obtained here do hold for methods
formulated in terms of the standard constitutive tensor $\mathcal{A}^{-1}$, like the one in 
\cite{SoonThesis08,SoonCockburnStolarski09,FuCockburnStolarski14}, for example.

The practical construction of \mm-decompositions in the three-dimensional 
case constitutes the subject of ongoing work.

\section*{Appendix \gfu{A}: Proofs of the error estimates in Section \ref{sec:error}}
In this Appendix, we provide proofs of our a priori error estimates in Section \ref{sec:error}, \gfu{namely, Theorem \ref{L2error-q}--\ref{L2error-u*}}. 
The main idea is to work with the following projection of the errors:
\begin{alignat*}{2}
\es :=&\; \Pis \und{\sigma} - \und{\sigma}_h,
&&\hspace{.45cm}
\eu :=\; \Piv \bld u -\bld u_h,
\\
\euhat :=&\; \Pim \bld u - \uhat,
&&\hspace{.2cm}
\eshat \n :=\; \es\n - \alpha(\eu-\euhat).
\end{alignat*}
We also  use $\eos:= \und\sigma-\und\sigma_h$ to simplify notation.

We begin by obtaining the equations satisfied by these projections.
{We then }use an energy argument to obtain an estimate of $\es$;
this would prove first part of Theorem \ref{L2error-q}. 
We prove the second part of Theorem \ref{L2error-q} following the idea in 
\cite[Appendix A.1]{FuCockburnStolarski14} for treating the incompressible limit.
Then, we prove the local error estimates in Theorem \ref{thm:stability} 
for the piecewise divergence $\divv\es$,
the piecewise symmetric gradient $\und\epsilon (\eu)$, and the jump term $\eu-\euhat$, 
using an adjoint HDG-projection similar to \cite{CockburnFuSayas16}.
{Next, we obtain an estimate of $\eu$ 
with an elliptic duality.}  After that, we 
obtain the estimate  for the displacement postprocessing in
Theorem \ref{L2error-u*}. 

\subsection*{Step 1: The equations for  the projection of the errors} 
We begin our error analysis with the following auxiliary result.
\begin{lemma}\label{error-equations}
Suppose that for every $K\in \Oh$, the space $\VV(K)\times \W(K)$
admits an $\bld M(\dK)$-decomposition and that the stabilization function $\alpha$ satisfies 
$a_{\Wtilde^\perp}>0$.
Then, we have
\begin{subequations}
 \begin{alignat}{3}
   \label{error-equation-1}
(\mathcal{A}\es,\und\tau)_\Oh + (\eu, \divv\und\tau)_\Oh-\bint{\euhat}{\und\tau\n}
 &\;=- \left(\mathcal{A}(\und\sigma -\Pis\und\sigma),\und\tau\right)_\Oh\\
   \label{error-equation-2}
(\es,\gradv \w)_\Oh - \bint{\eshat\bld n}{\w} &\;=0, \\
   \label{error-equation-3}
\langle{\eshat\bld n},{\bld \mu}\rangle_{\partial\Oh\setminus\partial\Omega} &\;= 0,\\
   \label{error-equation-4}
 \langle \euhat,{\bld \mu}\rangle_{\partial\Omega} &\;=  0,
 \end{alignat}
\end{subequations}
for all $(\und \tau, \w,\bld \mu) \in \und \Sigma_h \times \bld V_h \times \bld M_h$.
\end{lemma}

\gfu{
\begin{proof}
The proof follows directly from the consistency of the HDG method \eqref{weak formulation} and the definition of the 
HDG-projection \eqref{proj}. For example, to prove the second equation
\eqref{error-equation-2}, we proceed as follows. We have that
\begin{alignat*}{2}
 (\es,\gradv \w)_\Oh - \bint{\eshat\bld n}{\w} = &\; (\Pis\und\sigma, \gradv \w)_\Oh - \bint{\Pis\und\sigma\,\bld n
 - \alpha(\Piv \bld u-\Pim \bld u)}{\w}\\
 &\;-  (\und\sigma_h,\gradv \w)_\Oh + \bint{\widehat{\und\sigma}_h \n}{\w}\\
 =&  (\und\sigma, \gradv \w)_\Oh - \bint{\und\sigma\,\bld n}{\w} \;-  (\und\sigma_h,\gradv \w)_\Oh + \bint{\widehat{\und\sigma}_h \n}{\w},
 \end{alignat*}
by equations \eqref{proj-1} and \eqref{proj-3}. Finally, by equation \eqref{weak formulation-2},
\begin{alignat*}{2}
 (\es,\gradv \w)_\Oh - \bint{\eshat\bld n}{\w} =&  (\und\sigma, \gradv \w)_\Oh - \bint{\und\sigma\,\bld n}{\w} 
 - (\bld f, \bld v)_\Oh \\
 = &\; (-\divv \und\sigma -\bld f,\bld v)_\Oh = 0.
\end{alignat*}
The other equations can be proven in a similar way.
\end{proof}
}

\subsection*{Step 2: The proof of Theorem \ref{L2error-q}}
We begin with an energy argument to prove the first result \eqref{sigma-e1} in Theorem \ref{L2error-q}.
We proceed as follows. Taking $\und{\tau}:= \es$  in the error equation 
\eqref{error-equation-1}, 
$\w:=\eu$ in the error equation
\eqref{error-equation-2},
$\bld \mu:=\euhat$ 
in the error equation
\eqref{error-equation-3},
and 
$\bld \mu:=\eshat \bld n$
in the error equation \eqref{error-equation-4},
and adding the resulting equations up, we obtain
\[
\left(\mathcal{A}\es,{\es}\right)_\Oh + \Theta_h
= - \left(\mathcal{A}(\und\sigma- \Pis\und\sigma),{\es}\right)_\Oh 
\]
where
\begin{alignat*}{1}
\Theta_h:=&\; (\eu, \divv\es)_\Oh-\bint{\euhat}{\es\n}
+(\es,\gradv \eu) - \bint{\eshat\bld n}{\eu}
+\langle{\eshat\bld n},{\euhat}\rangle_{\partial\Oh}\\
=&\; \langle{\es\n-\eshat\bld n},{\eu-\euhat}\rangle_{\partial\Oh}\\
=&\; \langle{\alpha(\eu-\euhat)},{\eu-\euhat}\rangle_{\partial\Oh}.
\end{alignat*}
So we have that
\[
\|{\es}\|^2_{\mathcal{A},\Oh}+\|\eu-\euhat\|_{\alpha,\dOh}^2 \le
-\left(\mathcal{A}(\und\sigma-\Pis\und\sigma),{\es}\right)_\Oh
\le\|\und\sigma-\Pis\und\sigma\|_{\mathcal{A},\Oh} \|\es\|_{\mathcal{A},\Oh},
\]
and the result follows. This completes the proof of the first result 
\eqref{sigma-e1} of Theorem \ref{L2error-q}. 

Now, let us prove the second result \eqref{sigma-e2}. 
We define the deviatoric part of a tensor by 
$\und\tau^D := \und\tau-\frac{1}{n}\mathrm{tr}(\und\tau)\und I$. Hence, we have 
\[
 (\mathcal{A}\und\sigma,\und\tau)_\Oh = \frac{1}{2\mu}\left(\und\sigma^D,\und\tau^D\right)_\Oh+\frac{1}{n(2\mu+n\lambda)}\left(
\mathrm{tr}(\und\sigma),\mathrm{tr}(\und\tau)\right)_\Oh.
\]
Then, the first result \eqref{sigma-e1}
implies that
\begin{align}
\label{e-div}
\|\esd\|_\Oh \le 2\mu\|\und\sigma-\Pis\und\sigma\|_{\mathcal{A},\Oh}\le \|\und\sigma-\Pis\und\sigma\|_\Oh.                      
\end{align}
Let $\ep := \frac{1}{n}\mathrm{tr}(\es)$, then $\es = \esd+\ep\und I$.
In order to prove \eqref{sigma-e2}, we are left to bound the $L^2$-norm of $\ep$.
Now, taking $\und\tau$ to be the identity tensor in \eqref{error-equation-1} 
and by the fact that $\euhat = \bld 0$ on $\dO$, we obtain
\[n(\ep, 1)_\Oh = ( \mathrm{tr}(\es), 1)_\Oh = -(\mathrm{tr}(\und\sigma-\Pis\und\sigma),1)_\Oh = -(\und\sigma-\Pis\und\sigma,\und I)_\Oh = 0.\] 
Hence $(\ep, 1)_\Oh = 0$.
It is well known \cite{Temam79} that for any function $q\in L^2(\Omega)$ such that $(q,1)_\Omega = 0$ we have 
\[
 \|q\|_\Omega \le \theta\sup_{\bld 0\not=\bld w\in \bld H_0^1(\Omega)}\frac{(q,\divs \bld w)}{\|\bld w\|_{1,\Omega}},
\]
for some constant $\theta$ independent of $q$. Now, we take $q := \ep$ and work with the numerator in the above expression.
We have
\begin{alignat*}{2}
 (\ep,\divs\bld  w)_\Omega =&\; -(\grads\ep,\bld  w)_\Oh + \bint{\ep\n}{\bld  w} = T_1+T_2,
\end{alignat*}
where 
\begin{align*}
 T_1=&\; -(\grads \ep,\bld w-\bld P_V \bld  w)_\Oh,\\
 T_2=&\; -(\grads \ep,\bld P_V \bld  w)_\Oh + \bint{\ep\n}{\bld  w}.
\end{align*}
Let us bound the above terms individually. We have
\begin{alignat*}{2}
T_1 =&\; -(\grads \ep,\bld w-\bld P_V \bld  w)_\Oh\\
=&\; -(\divv \es-\divv\esd,\bld w-\bld P_V\bld  w)_\Oh\\
= &\;(\divv\esd,\bld w-\bld P_V\bld  w)_\Oh && \; \text{ since $\divv\und{\Sigma}(K)\subset \bld V(K)$,}\\
\le &\; C h\|\divv\esd\|_\Oh |\bld  w|_{1,\Oh},
\end{alignat*}
and 
\begin{alignat*}{2}
T_2 =&\;  (\ep,\divs\bld P_V\bld  w)_\Oh - \bint{\ep\n}{\bld P_V\bld  w-\bld  w}\\
= &\; -(\esd,\gradv\bld P_V\bld  w)_\Oh + \bint{\eshat\n}{\bld P_V\bld  w}
- 
\bint{\ep\n}{\bld P_V\bld  w-\bld  w},
\end{alignat*}
by the error equation \eqref{error-equation-2} with $\w:=\bld P_V\bld  w$. Now, since
$\eshat\n$ is single-valued and $\bld  w\in \bld H_0^1(\Omega)$, we have that 
$\bint{\eshat\n}{\bld  w}=0$, and so
\begin{alignat*}{2}
T_2
=&\; -(\esd,\gradv\bld P_V\bld  w)_\Oh + \bint{\eshat\n-\ep\n}{\bld P_V\bld  w-\bld  w}\\
=&\; -(\esd,\gradv\bld P_V\bld  w)_\Oh \\
&\;+ \bint{\esd\n - \alpha(\eu-\euhat)}{\bld P_V\bld  w-\bld  w}
&&\; \text{ by the definition of $\eshat\n$,}\\
=&\; -(\esd,\gradv\bld  w)_\Oh - \bint{ \alpha(\eu-\euhat)}{\bld P_V\bld  w-\bld  w} && \;{\text{ by integration by parts,}}\\
\le &\; \left(\|\esd\|_\Oh+Ch^{1/2}\|\alpha\|\, \|\eu-\euhat\|_{\alpha,\dOh}\right)|\bld  w|_{1,\Omega},
\end{alignat*}
by the Cauchy-Schwarz inequality. Hence.
\[
 \|\ep\|_\Oh\le \; \theta  \left(\|\esd\|_\Oh+C\,h\|\divv\esd\|_\Oh+C\,h^{1/2}\|\alpha\|\,
 \|\eu-\euhat\|_{\alpha,\dOh}\right).
\]
Combining this result with \eqref{e-div} and the estimate of $\eu-\euhat$ in Theorem \ref{thm:stability}, 
we obtain 
\[
 \|\es\|_\Oh \le \|\esd\|_\Oh + n\|\ep\|_\Oh\le C(1+h^{1/2}\|\alpha\|)\|\und\sigma-\Pis\und\sigma\|_\Oh,
\]
with $C$ independent of $h$, $\alpha$, and $\mathcal{A}$. 
This completes the proof of Theorem \ref{L2error-q}.
$\square$ 

\subsection*{Step 3: The proof of Theorem \ref{thm:stability}}

\newcommand{\Pivs}{\und{\Pi}_\Sigma^*}
\newcommand{\Piws }{\bld \Pi_V^* }

Following \cite{CockburnFuSayas16},  we first introduce an auxiliary adjoint HDG-projection onto $\VV\times \W$ 
for functions in the finite element space 
$ \VV\times \W \times \bld M$, 
and 
then choose test functions in the local equations
to be the adjoint HDG-projection of certain finite element data to prove Theorem \ref{thm:stability}.
The adjoint HDG-projection is defined as follows.
\begin{definition}[The auxiliary adjoint HDG-projection]
\label{adjoint-HDG-proj}
Let $\VV\times \W$ admit an \mm-decomposition. 
Let $d := (\underline{\bld{\mathrm{d}}}_{\tau},\bld{\mathrm{d}}_v,\bld{\mathrm{d}}_\mu)\in \VV \times \W \times \bld M$. 
Then,  $\Pi_h^* d:=(\Pivs d, \,\Piws d)\in \VV\times \W$ defined by the equations
\begin{alignat*}{6}
(\Piws d, \und \tau)_K &= (\bld{\mathrm{d}}_v, \und \tau)_K& & \forall \und \tau\in \Wtilde,\\
(\Pivs d,\w)_K &=(\underline{\boldsymbol{\mathrm{d}}}_\tau,\w)_K&\qquad & \forall \w\in\Vtilde,\\
\langle \Pivs d\,\nn + \alpha( \Piws d),\bld\mu \rangle_{\partial K}
&=\langle \bld{\mathrm{d}}_\mu ,\bld\mu \rangle_{\partial K} & &\forall\bld \mu \in\bld M,
\end{alignat*}
is the auxiliary adjoint HDG-projection associated to the \mm-decomposition. 
\end{definition}

{It is easy to see that this adjoint is well-defined whenever the HDG projection is. 
In fact, a glance to the definition of the auxiliary HDG projection in Theorem \ref{actual-proj}, allows us to see that
$(\Pis\und\sigma,\Piv \bld u)=(-\Pivs d, \,\Piws d)$ for 
$d=(-\und{\sigma},\bld  u, -\und\sigma\,\nn-\alpha(\bld P_M \bld u))$.}


\gfu{We have the following bounds for the adjoint projection whose proof is 
very similar to the one for the case of $M$-decompositions for diffusion; see
 \cite[Appendix]{CockburnFuSayas16}. For completeness, we sketch its proof in Appendix B.}

\newcommand{\dtau}{\underline{\boldsymbol{\mathrm{d}}}_\tau}
\newcommand{\dv}{\bld{\mathrm{d}}_v}
\newcommand{\dmu}{\bld{\mathrm{d}}_\mu}
\begin{lemma} [Stability of the adjoint HDG-projection]
\label{lemma:adj-proj} 
Let $\VV\times \W$ admit an \mm-decomposition and let 
the stabilization function $\alpha$ satisfy Property (S).
Then, we have, for data $d := 
(\underline{\boldsymbol{\mathrm{d}}}_\tau,\bld{\mathrm{d}}_v,\bld{\mathrm{d}}_\mu)\in 
\und\epsilon( \W) \times \divv \VV\times \bld M $,
\begin{alignat*}{1}
\| \Pi^*_\VV d\|_K\le&\;
    C_1\| \bld{\mathrm{d}}_v\|_K 
    +
C_3
\|\underline{\boldsymbol{\mathrm{d}}}_\tau \|_K
+C_5\,h^{1/2}_K\,\|\bld{\mathrm{d}}_\mu\|_{ \partial K},\\
    \| \Pi^*_{\W} d\|_K\le&\;
    C_2\| \bld{\mathrm{d}}_v\|_K 
    +
C_4
\|\underline{\boldsymbol{\mathrm{d}}}_\tau \|_K
+C_6\,h^{1/2}_K\,\|\bld{\mathrm{d}}_\mu\|_{ \partial K},
  \end{alignat*}
  where $\{C_i\}_{i=1}^6$ are the constants defined in Theorem \ref{thm:stability}.
\end{lemma}



Now, we are ready to prove the stability estimates in
Theorem \ref{thm:stability}.
To do that, we begin by noting 
that we can rewrite the first two equations defining the HDG methods \eqref{weak formulation} on each element $K$
as 
\begin{align}
\label{local-1}
(\mathcal{A}\,\und{\sigma}_h,\und{\tau})_K - (\und\epsilon(\bld  u_h),\und{\tau})_K  
- (\divv \und{\sigma}_h,\w)_K&\\
+\langle \bld u_h-\widehat{\bld u}_h,\und{\tau}\, 
 \nn+\alpha(\w)\rangle_{\partial K}
& 
 = \;(\bld f,\w)_K,\nonumber
 \end{align}
 for all $(\und\tau,\w)\in \VV(K)\times \W(K)$.
 {Therefore, testing with $(\und\tau,\w):=(\Pivs{d},\Piws{d})$ and using the equations that define the adjoint HDG projection, 
it follows that
\begin{equation}\label{eq:newTheorem4.3}
(\und\epsilon (\bld u_h), \underline{\bld{\mathrm{d}}}_\tau)_K+(\divv\und \sigma_h,\bld{\mathrm{d}}_v)_K
-\langle \bld u_h-\widehat {\bld u}_h,\bld{\mathrm d}_\mu\rangle_{\partial K}
=-(\bld f,\Piws{d})_K+(\mathcal{A}\und \sigma_h,\Pivs{d})_K
\end{equation}
for an arbitrary $d=(\underline{\bld{\mathrm{d}}}_\tau,\bld {\mathrm d}_v,\bld {\mathrm d}_\mu)$.}

{To prove the first stability estimate, we take $d := (\und{0},\divv\und \sigma_h,\bld 0)$ in \eqref{eq:newTheorem4.3} 
and use that $\Piws d\in \W$, so that
\begin{alignat*}{1}
 \|\divv \und\sigma_h\|_K^2
  = &\;-(\bld P_{V}\bld f,\Piws d)_K+(\mathcal{A}\und \sigma_h,\Pivs{d})_K
 \\
 \le &\; \|\mathcal{A}\|_{L^\infty(K)}^{1/2}\|\und \sigma_h\|_{\mathcal{A},K}\|\Pivs d\|_K+\|\bld P_{V}\bld f\|_K\|\Piws d\|_K\\
 \le &\; \left(C_1\, \|\mathcal{A}\|_{L^\infty(K)}^{1/2}\|\und \sigma_h\|_{\mathcal{A},K}+ C_2\,\|\bld P_{V}\bld f\|_K\right)\,
 \|\divv \und \sigma_h \|_K,
\end{alignat*}
by the stability properties of the adjoint projection in Proposition \ref{lemma:adj-proj}.
To prove the second estimate, we take  $d := (\und\epsilon(\bld u_h),\bld 0,\bld 0)$ in \eqref{eq:newTheorem4.3} and
note that $\Piws d\in \Wtilde^\perp$ because $\bld{\mathrm{d}}_v=\bld 0$. Then
\begin{align*}
 \|\und\epsilon(\bld u_h)\|_K^2  
 = &\;-(\bld P_{\widetilde{V}^\perp}\bld f,\Piws d)_K+(\mathcal{A}\und \sigma_h,\Pivs{d})_K\\
 \le &\; \|\mathcal{A}\|_{L^\infty(K)}^{1/2}\|\und \sigma_h\|_{\mathcal{A},K}\|\Pivs d\|_K+\|\bld P_{\widetilde{V}^\perp}\bld f\|_K\|\Piws d\|_K\\
 \le &\; \left(C_3\, \|\mathcal{A}\|_{L^\infty(K)}^{1/2}\|\und \sigma_h\|_{\mathcal{A},K}+ C_4\,\|\bld P_{\widetilde{V}^\perp}\bld f\|_K\right)\,
 \|\und\epsilon( \bld u_h) \|_K,
\end{align*}
by  Proposition \ref{lemma:adj-proj}. To prove the third estimate, we take $d := -(\und 0,\bld 0,\bld u_h-\widehat{\bld u}_h)$  
in \eqref{eq:newTheorem4.3}
\begin{align*}
 \|\bld u_h-\widehat{\bld u}_h\|_{\dK}^2  = &\;-(\bld P_{\widetilde{V}^\perp}\bld f,\Piws d)_K+(\mathcal{A}\und \sigma_h,\Pivs{d})_K\\
 \le &\; \|\mathcal{A}\|_{L^\infty(K)}^{1/2}\|\und \sigma_h\|_{\mathcal{A},K}\|\Pivs d\|_K+\|\bld P_{\widetilde{V}^\perp}\bld f\|_K\|\Piws d\|_K\\
 \le &\; \left(C_5\, \|\mathcal{A}\|_{L^\infty(K)}^{1/2}\|\und \sigma_h\|_{\mathcal{A},K}+ C_6\,\|\bld P_{\widetilde{V}^\perp}\bld f\|_K\right)
 h_K^{1/2}\|\bld u_h-\widehat{\bld u}_h\|_{\dK},
\end{align*}
by Proposition \ref{lemma:adj-proj}.
}

{The remaining estimates can be proven in exactly the same way given that we have
\begin{equation*}
(\und\epsilon (\eu), \underline{\bld{\mathrm{d}}}_\tau)_K+(\divv\es,\bld{\mathrm{d}}_v)_K
-\langle \eu-\euhat,\bld{\mathrm d}_\mu\rangle_{\partial K}
= (\mathcal{A}\bld e_{\und\sigma},\Pivs{d})_K
\end{equation*}
for an arbitrary $d=(\underline{\bld{\mathrm{d}}}_\tau,\bld {\mathrm d}_v,\bld {\mathrm d}_\mu)$.
This completes the proof of Theorem \ref{thm:stability}. \qed
}

\subsection*{Step 4: The proof of Theorem \ref{L2error-u}}
The estimate of $\eu$ in Theorem \ref{L2error-u}
will follow from the following identity, whose
proof  is a standard duality argument hence is omitted.
We refer to \cite[Lemma 3]{FuCockburnStolarski14} for details of a similar proof.
\begin{lemma}\label{duality}
Suppose that the assumptions of Theorem \ref{L2error-u} are satisfied.
Then, we have
\[({\eu},{\bld\theta})_\Oh = (\mathcal{A}\,\eos,\und\psi-\Pis\und\psi)_\Oh + (\und\sigma-\Pis\und\sigma,
\und\epsilon(\bld \phi-\bld \phi_h))_\Oh\quad \quad \forall \bld \phi_h\in \bld V_h,
\]
where $(\und\psi, \bld \phi)$ is the solution of the dual problem \eqref{adjoint}.
\end{lemma}

From Lemma \ref{duality}, we get that
\begin{alignat*}{1}
 \|\eu\|_\Oh \le H(\bld \theta)\left(\|\mathcal{A}\|^{1/2}_{L^\infty(K)}\|\eos\|_{\mathcal{A},\Oh}+\|\und\sigma-\Pis\und\sigma\|_{\Oh}\right),
\end{alignat*}
where 
\begin{align*}
 H(\bld \theta) :=\sup_{\bld 0\not=\bld \theta\in \bld L^2(\Omega)}
\frac{\|\und \psi-\Pis \und\psi\|_\Oh}{\|\bld \theta\|_\Oh}
+ \sup_{\bld 0\not=\bld \theta\in \bld L^2(\Omega)}
\inf_{\bld \phi_h \in \bld V_h}
\frac{\|\und\epsilon(\bld \phi-\bld \phi_h)\|_\dOh}{\|\bld \theta\|_\Oh}.
\end{align*}
If the elliptic regularity estimate \eqref{regularity} holds, we have
\begin{align*}
 H(\bld \theta) \lesssim &\; h\sup_{\bld 0\not=\bld \theta\in \bld L^2(\Omega)}
\frac{\|\und\psi\|_{H^1(\Oh)}+\|\bld \phi\|_{H^{2}(\Oh)}}{\|\bld \theta\|_\Oh}
\lesssim\; h.
\end{align*}
This completes the proof of Theorem \ref{L2error-u}. $\square$


\subsection*{Step 5: The proof of Theorem \ref{L2error-u*}}
We conclude this section by sketching the proof of Theorem \ref{L2error-u*} on the postprocessing of 
the displacement. We denote $\bld P_{V^*}$, $\bld P_{\widetilde{V}^{*}}$, and 
$\bld P_{{V}^{*,\perp}}$ as the corresponding $L^2$-projections onto the spaces $\bld V^*(K)$,
$\widetilde{\bld V}^*(K)$, $\bld V^{*}(K)^\perp$, respectively.

First, the inclusion $\bpol{0}{K}\subset \widetilde\W^{*}(K)$
ensures the well-posedness of the postprocessing in \eqref{u2*}. Next, using 
$\widetilde{\bld V}^{*}(K)\subset \divv \und\Sigma(K)$,  equation \eqref{u2*-2} and the definition of $\Piv \bld u$, 
we have
\[
 \bld P_{\widetilde V^{*}} (\bld u-\bld u_h^{*}) =  \bld P_{\widetilde V^{*}} (\Piv\bld u-\bld u_h)
= \bld P_{\widetilde V^{*}} \eu.
\]
Hence,
\begin{align}
 \label{es-1}
 \|  \bld P_{\widetilde V^{*}} (\bld u-\bld u_h^{*}) \|_K \le \|\eu\|_K.
\end{align}
\gfu{
Then, by the assumptions  $\triangle \W^*\subset \divv\VV$
and $(\gradv\W^*\n)|_{\dK}\subset\bld M(\dK)$, we have 
the following identity
\begin{align*}
 (\gradv  \bld u,\gradv \bld w)_K = &\;
 -(\bld u,\triangle \bld w)_K + \bintK{\bld u}{\gradv\bld w\,\n}\\
 =&\;
 -(\Piv\bld u,\triangle \bld w)_K + \bintK{\Pim\bld u}{\gradv\bld w\,\n}.
\end{align*}
Combining the above equality with equation \eqref{u2*-1}, we get 
\begin{align*}
 (\gradv  (\bld u-\bld u_h^*),\gradv \bld w)_K = &\;
 -(\eu,\triangle \bld w)_K + \bintK{\euhat}{\gradv\bld w\,\n}.
\end{align*}
Now, taking $\bld w = \bld P_{\widetilde{V}^{*,\perp}}(\bld u-\bld u_h^{*})$ in the above equality, 
we get 
\begin{align}
\label{es-2}
 \|\bld P_{\widetilde{V}^{*,\perp}}(\bld u-\bld u_h^{*})\|_K
 \le &\;C\,h_K\|\gradv \bld P_{\widetilde{V}^{*,\perp}}(\bld u-\bld u_h^{*})\|_K \nonumber\\
 \le&\; 
 C\, h_K \Big(\|\gradv (\bld u - \bld P_{{V}^*}\bld u)\|_K+ \|\gradv \bld P_{\widetilde{V}^{*}}(\bld u-\bld u_h^{*})\|_K
 \nonumber\\
 &\;
 + h_K^{-1}\|\eu\|_K + h_K^{-1/2}\|\euhat\|_\dK
 \Big)\nonumber\\
 \le&\; 
 C\, (h_K\|\gradv (\bld u - \bld P_{{V}^*}\bld u)\|_K+
 \|\eu\|_K + h_K^{1/2}\|\euhat\|_\dK).
\end{align}
Combining the estimates in \eqref{es-1} and \eqref{es-2}, and the fact that 
$\|\bld u-\bld P_{V^*}\bld u\|_K\le C\,h_K\|\gradv(\bld u-\bld P_{V^*}\bld u)\|_K$,
we conclude the proof of Theorem \ref{L2error-u*}. \qed
}
\section*{Appendix B: Proofs of Theorem \ref{actual-proj} and Lemma \ref{lemma:adj-proj}}
In this Appendix, we prove the approximation properties of the HDG-projection in Theorem \ref{actual-proj}, and 
the stability properties of the adjoint HDG-projection in Lemma \ref{lemma:adj-proj}.
\subsection*{Proof of Theorem \ref{actual-proj}}
We first  estimate
the quantities
$\ds:=\Pis \und\sigma- \und {P}_{\Sigma}\und{\sigma}$ 
and $\du:=\Piv \bld u -\bld P_V \bld u$, and then
use the triangle inequality to obtain the desired estimates.
We proceed in three steps.

\subsection*{Step 1: The equations for $\ds$ 
and $\du$} 
By the equations \eqref{proj} defining the HDG-projection, we have that
\begin{subequations}
 \begin{alignat}{2}
 (\ds, \vv)_K 
 & = 0  
&&\forall\;{\vv} \in \Vtilde,\\
 (\du, \w)_K & = 0 
&& \forall\;\w \in \Wtilde,\\
\label{bdequation}
\langle \ds \n + {\alpha(\du)} , \bld\mu \rangle_{\dK} & = 
\langle \boldsymbol{I}_{\und\sigma} \n + {\alpha(\bld I_{\bld u})} , \bld \mu \rangle_{\partial K} \quad &&\forall\;\bld \mu \in\bld M,
\end{alignat}
\end{subequations}
Here 
$\boldsymbol{I}_{\und\sigma}:=\und\sigma- \und{P}_{\Sigma}\und\sigma$ 
and $\bld I_{\bld u}:=\bld P_M \bld u -\bld P_V \bld u$. The first equation implies $\ds\in \Vtilde^\perp$, and
the second equation implies $\du\in \Wtilde^\perp$. {Therefore  
\begin{equation}\label{afterbdequation}
(\divv\und P_\Sigma \und\sigma,\du)_K=0,
\qquad\langle \ds\n,\du\rangle_{\partial K}=0.
\end{equation}}

\newcommand{\Is}{\bld I_{\und\sigma}}
\newcommand{\Iu}{\bld I_{\bld u}}
\subsection*{Step 2: The estimate of $\du$}
Next, we obtain an estimate of $\du$. 
{Taking $\bld \mu = \du$ in \eqref{bdequation}, and using \eqref{afterbdequation}, integration by parts and the fact that $\du\in \Wtilde^\perp$, it follows that 
\begin{align*}
 \langle {\alpha(\du)} , \du \rangle_{\dK} 
&=\; \langle \Is\n + {\alpha(\Iu)} , \du \rangle_{\partial K}\\
&=\; (\Is , \gradv\du)_K+(\divv \Is ,
\du)_K+\langle {\alpha(\Iu)} , \du \rangle_{\partial K}\\
&=\; \left((Id-\bld P_{\Wtilde})\divv{\und\sigma} ,
\du\right)_K+\langle {\alpha(\Iu)} , \du \rangle_{\partial K}\\
&\le \;\|(Id-\bld P_{\Wtilde})\divv{\und\sigma}\|_K\,\|\du\|_K
+ \|\alpha\|\,\|\Iu\|_{\dK} \, \|\du\|_{\partial K}.
\end{align*}}
By the definition of the constants $C_{\Wtilde^\perp}$, $a_{\Wtilde^\perp}$, and $\|\alpha\|$, we get
\[
\|\du\|_{\partial K} 
\le \frac{C_{\Wtilde^\perp}}{a_{\Wtilde^\perp}}\,h^{1/2}_K\,
\|(Id-\bld {P}_{\widetilde{\W}})\divv {\und\sigma}\|_{K}
+ \frac{\|\alpha\|}{a_{\Wtilde^\perp}}\,\|\Iu\|_{\partial K}.
\]
{Since $\du \in \Wtilde^\perp$, we have that $\|\du\|_K
\le\, C_{\Wtilde^\perp}\,h^{1/2}_K\,\|\du\|_{\partial K}$, by the definition of the constant $ C_{\Wtilde^\perp}$.
 As a consequence, we get that}
\[
\|\du\|_K
\le \frac{C_{\Wtilde^\perp}^2}{a_{\Wtilde^\perp}}\,h_K\,
\|(Id-\bld {P}_{\widetilde{\W}})\divv {\und\sigma}\|_{K}
+
\frac{C_{\Wtilde^\perp}}{a_{\Wtilde^\perp}}\|\alpha\|\, h^{1/2}_K\|\Iu\|_{\partial K}.
\]

\subsection*{Step 3: The estimate of $\ds$}
Finally, let us estimate of $\ds$.
{Taking $\bld\mu = \ds \nn$ in the boundary equation \eqref{bdequation} and applying the Cauchy-Schwarz inequality, 
we obtain}
\begin{alignat*}{2}
\|\ds\nn\|_\dK
\le&\; 
\|\Is \cdot \n\|_{ \partial K}
+ \|\alpha\|\,\|\Iu\|_{\partial K}
+ \|\alpha\|\,\|\du\|_{\partial K}.
\end{alignat*}
{\color{blue}Since $\ds \in \Vtilde^\perp$, we have that $\|\ds\|_K
\le\, C_{\Vtilde^\perp}\,h^{1/2}_K\,\|\ds\nn\|_{\partial K}$, by the definition of the constant $ C_{\Vtilde^\perp}$.
 As a consequence, we get that}
\begin{alignat*}{2}
\|\ds\|_K
\le&\; C_{\Vtilde^\perp}\,h^{1/2}_K\,
\|\Is \n\|_{ \partial K}
+ \left(\frac{C_{\Wtilde^\perp}}{a_{\Wtilde^\perp}}\,C_{\Vtilde^\perp}\,\|\alpha\|\right)\,h_K\,
\|(Id-\bld{P}_{\widetilde{\W}})\divv\und\sigma\|_{K}
\\&
+ \left(1+\frac{\|\alpha\|}{a_{\Wtilde^\perp}}\right)\,
C_{\Vtilde^\perp}\,\|\alpha\|\,h^{1/2}_K\,\|\Iu\|_{\partial K}.
\end{alignat*}
This completes the proof. \qed

\subsection*{Proof of Lemma \ref{lemma:adj-proj}}
Here we only sketch the proof of Lemma \ref{lemma:adj-proj}
since it is quite similar to the proof of 
Theorem \ref{actual-proj}.
We first estimate
the quantities
$\bld\delta_{\bld v}^*d:=\Piv^* d -\dv$ and $\boldsymbol{\delta}^*_{\vv}d:=
\Pis^*d- \dtau$ , and then
use the triangle inequality to obtain the estimates for $\Piv^*d$ and $\Pis^*d$.

\newcommand{\deltau}{\boldsymbol{\delta}_{\und\tau}^*d}
\newcommand{\delv}{\boldsymbol{\delta}_{\bld v}^*d}
\newcommand{\delmu}{\boldsymbol{\delta}_{\bld \mu}^*d}

First, by the equations defining the adjoint projection in Definition \ref{adjoint-HDG-proj}, we have that
 \begin{alignat*}{2}
 (\deltau, {\vv})_K 
 & = 0  
&&\forall\;\vv \in \Vtilde,\\
 (\delv, {\bld v})_K & = 0 
&& \forall\;{\bld v} \in \Wtilde,\\
\langle\deltau \n - 
{\alpha(\delv)} , \bld\mu \rangle_{\dK} & = 
\langle \dmu-\dtau \n + {\alpha(\dv)} , \bld\mu \rangle_{\partial K} \quad &&\forall\;\bld\mu \in \bld M,
\end{alignat*}
The first equation implies $\deltau\in \Vtilde^\perp$, and
the second equation implies $\delv\in \Wtilde^\perp$.

Next, we obtain an estimate of $\delv$. 
Take $\bld\mu = \delv$ in the previous boundary equation, use 
the $L^2(\dK)$-orthogonality of $\delv$ and $\deltau\nn$, and
apply the Cauchy-Schwarz inequality, we obtain
\begin{align*}
 \langle {\alpha(\delv)} , \delv \rangle_{\dK} & \le 
\big( \|\dmu\|_\dK+ \|\dtau \nn\|_\dK+\|\alpha\|\,\|\dv\|_\dK\big)\|\delv\|.
\end{align*}
Now, the desired estimate for $\Piv^*d$ comes from {using inverse inequalities} and the triangle inequality.

Finally, taking $\bld\mu = \deltau\nn$ in the boundary equation and applying the Cauchy-Schwarz inequality,
inverse inequalities, and the triangle inequality, we get the estimate for $\Pis^*d$.
This completes the proof. \qed
}

\section*{Appendix C: Proof of Theorem \ref{thm:1.5}}
In this Appendix, we prove Theorem \ref{thm:1.5} on the characterization of \mm-decompositions. 

We first collect several auxiliary results that will be used in the proof of Theorem \ref{thm:1.5}.

\begin{lemma}[Uniqueness of $\Wtilde$]
\label{th:1.1}
If $\VV\times \W$ admits an \mm-decomposition then $\Wtilde=\divv\VV.$
\end{lemma}
\begin{proof}
Since $\divv\VV \subset\Wtilde$, we only need to prove that 
$\Wtilde \cap (\divv\VV)^\perp=\{0\}.$ So, if we take $\widetilde \w\in\Wtilde$ satisfying
\[
(\widetilde \w,\divv\vv)_K=0\qquad\forall \vv\in \VV,
\]
we have to show that $\widetilde \w=0$. To do that, we integrate by parts to
get that
\[
-(\gradv\widetilde \w,\vv)_K+\langle \widetilde \w,\vv\nn\rangle_{\partial K}=-(\und{\epsilon}(\widetilde \w),\vv)_K+\langle \widetilde \w,\vv\nn\rangle_{\partial K}=0\qquad\forall \vv\in \VV,
\]
and, in particular, that
\[
\langle\widetilde \w,\vv^\perp\nn\rangle_{\partial K}=0\;\forall \vv^\perp\in \Vperp
\] since $\und{\epsilon}{\W}\subset\Vtilde$.
This implies that  there exists $\widetilde{\w}^\perp\in \Wperp$ such that $\widetilde \w+\widetilde \w^\perp=0$ on $\partial K$. As a consequence, we have that, for any $\vv\in\VV$,
\begin{alignat*}{6}
(\und{\epsilon}(\widetilde \w),\vv)_K 
	&=-\langle \widetilde \w^\perp,\vv\nn\rangle_{\partial K} 
      \\&=-(\gradv \widetilde \w^\perp,\vv)_K
         -(\widetilde \w^\perp,\divv\vv)_K 
      \\&=-(\gradv \widetilde \w^\perp,\vv)_K
      \\&=-(\und{\epsilon}(\widetilde \w^\perp),\vv)_K, 
\end{alignat*}
since $\divv\VV\subset\Wtilde$. This is equivalent to
$
(\und{\epsilon} (\widetilde \w+ \widetilde \w^\perp),\vv)_K=0\; \forall \vv\in \VV,
$
and therefore to $\und{\epsilon} (\widetilde \w+ \widetilde \w^\perp)=0$, since 
$\und{\epsilon}( \W)\subset \VV$ by hypothesis. Given the fact that 
$\widetilde \w+ \widetilde \w^\perp=0$ on $\partial K$, this implies that $\widetilde \w=- \widetilde \w^\perp\in \Wtilde\cap\Wperp=\{0\}$ and the proof is complete.
\end{proof}


\begin{lemma}[Relation between the spaces $\widetilde{\VV}^\perp$]
\label{prop:charVtildeperp}
If $\VV\times \W$ admits two \mm-decompositions with associated subspaces $\Vtilde_1$ and $\Vtilde_2$, then
\[
\dim \Vtilde^\perp_1
=\dim \Vtilde_2^\perp
=\dim \bld M-\dim \W+\dim \divv\VV,
\] 
and 
$
\Vperp_1+\VV_{\mathrm{sbb}}=\Vperp_2+\VV_{\mathrm{sbb}}.
$
Here the space 
\[
\VV_{\mathrm{sbb}}:=\{\vv \in \VV \,:\, \divv\vv=0, \quad \vv\nn=0\}.
\]
\end{lemma}

\begin{proof} 
Let us calculate the dimension of $\Vtilde_i$ for $i=1,2$. By property (c) of an
\mm-decomposition, we have that
\begin{alignat*}{1}
\dim \Vtilde_i^\perp
&=\dim \bld M-\dim \Wtilde_i^\perp
\\
&=\dim \bld M-(\dim \W-\dim \Wtilde_i)
\\&=\dim \bld M-(\dim \W-\dim \divv\VV).
\end{alignat*}

To prove the second equality, 
it is clear that we only need to show that $\Vperp_1\subset \Vperp_2+\VV_{\mathrm{sbb}}$. Let $\vvtilde_1^\perp\in \Vperp_1$. By hypothesis (c) in the definition of an \mm-decomposition, there exists $(\vvtilde_2^\perp,\wtilde_2^\perp)\in \Vperp_2\times \Wperp_2$ such that 
$
\vvtilde_1^\perp\nn=\vvtilde_2^\perp\nn+\wtilde_2^\perp\quad\mbox{on $\partial K$}.
$
However
\begin{alignat*}{6}
\langle \wtilde_2^\perp,\wtilde_2^\perp\rangle_{\partial K}
	&=\langle (\vvtilde_1^\perp-\vvtilde_2^\perp)\nn,\wtilde_2^\perp\rangle_{\partial K}\\
	&=(\vvtilde_1^\perp-\vvtilde_2^\perp,\und\epsilon( \wtilde_2^\bot))_K
		+(\divv(\vvtilde_1^\perp-\vvtilde_2^\perp),\wtilde_2^\bot)_K=0,
\end{alignat*}
since $\vvtilde_1^\perp-\vvtilde_2^\perp$ is orthogonal to $\und\epsilon( \W)$
and $\wtilde_2^\perp$ is orthogonal to $\divv\und\Sigma$
This proves that $\vvtilde_1^\perp\nn=\vvtilde_2^\perp\nn$ on $\partial K$. Finally
\[
(\divv(\vvtilde_1^\perp-\vvtilde_2^\perp),\w)_K=(\vvtilde_1^\perp-\vvtilde_2^\perp,\gradv \w)=0\qquad\forall \w\in \W,
\]
and therefore $\divv(\vvtilde_1^\bot-\vvtilde_2^\bot)=0$, since $\W$ contains all divergences of elements of $\VV$. This proves that $\vvtilde_1^\perp-\vvtilde_2^\perp\in \VV_{\mathrm{sbb}}$.
\end{proof}

%
\begin{lemma}[The canonical \mm-decomposition]
\label{th:1.3}
If the space $\VV\times \W$ admits an \mm-decomposition, then it admits an \mm-decomposition based on the subspaces
\[
\Vtilde = \und\epsilon( \W){\oplus} \VV_{\mathrm{sbb}} \qquad \mbox{(orthogonal sum)}, \qquad \Wtilde=\divv\VV.
\]
\end{lemma}

To prove this result, we are going to use the following auxiliary result.
 \begin{lemma}
 \label{Mn}
Therefore, $\VV\times \W$ admits a decomposition if and only if
\begin{itemize}
\item[{\rm (a)}] $\tr (\VV\times \W)\subset \bld M$,
\item[{\rm (b)}]$\und{\epsilon} (\W)\subset \Vtilde$, $\divv\VV \subset \W$,
\end{itemize}
and there exists a subspace $\Vtilde\subset \VV$ such that:
\begin{itemize}
\item[{\rm (c)}] the restricted trace map 
\[
\Vperp\ni \vv \longmapsto \vv\nn \in \bld M_{\nn} 
	:=\{ \mu \in \bld M\,:\, \langle\mu,\w\rangle_{\partial K}=0\quad\forall \w\in (\divv\VV)^\perp\},
\]
is an isomorphism.
\end{itemize}
\end{lemma}

\begin{proof}
This result follows from the fact that $\Wtilde=\divv \VV$ by Lemma \ref{th:1.1}, and by the definition of an \mm-decomposition.
\end{proof}

We are now read to prove Lemma  \ref{th:1.3}.
\begin{proof}
Let us begin by noting that
\[
(\gradv \w,\vv_{\mathrm{sbb}})_K=(\w,\divv\vv_{\mathrm{sbb}})_K+\langle
\w,\vv_{\mathrm{sbb}}\nn\rangle_{\partial K}=0\quad\forall \w\in \W,\;\forall \vv_{\mathrm{sbb}}\in \VV_{\mathrm{sbb}}.
\]
This shows that $\gradv \W\perp\VV_{\mathrm{sbb}}$ and the sum $\gradv \W\oplus\VV_{\mathrm{sbb}}$ is orthogonal.

Now, let us construct the space $\Vtilde$.
Let $\Vtilde_1$ be the space associated to the existing \mm-decomposition, let $\Pi_{\mathrm{sbb}}:\VV \to \VV_{\mathrm{sbb}}$ be the $\boldsymbol{L}^2(K)$-orthogonal projector onto $\VV_{\mathrm{sbb}}$, and let us set
\[
\underline{\boldsymbol{\mathsf{S}}}:=\{\vvtilde^\perp-\Pi_{\mathrm{sbb}}\vvtilde^\perp:\vvtilde^\perp\in \Vperp_1\}.
 \]
We claim that $\Vtilde=\underline{\boldsymbol{\mathsf{S}}}^\perp$.

Before proving the claim, let us show that $\VV\times \W$ admits an \mm-decomposition with the associated subspace 
$\underline{\boldsymbol{\mathsf{S}}}^\perp\times\Wtilde$.
By definition of $\underline{\boldsymbol{\mathsf{S}}}$, it is clear that $\mathrm{dim}\,\underline{\boldsymbol{\mathsf{S}}}\le \mathrm{dim}\,\Vperp_1$.
Noting that $(\Pi_{\mathrm{sbb}}\vv)\nn =0$ for all $\vv\in\VV$, it follows that the range of the normal trace operators $\vv\mapsto \vv\nn$ from $\underline{\boldsymbol{\mathsf{S}}}$ and from $\Vperp_1$ is the same.
From Proposition \ref{th:1.1},  it follows that $\Vperp_1\ni\vv \mapsto \vv\nn \in \bld M_{\nn}:=\{\mu\in \bld M: \langle\mu,\w\rangle_{\partial K}=0 \quad \forall\; \w\in (\divv \VV)^\perp\}$ is an isomorphism of finite dimensional spaces, and therefore the normal trace $\underline{\boldsymbol{\mathsf{S}}}\to \bld M_{\nn}$ is also an isomorphism. This implies that $\VV\times \W$ admits an \mm-decomposition with the associated subspace 
$\underline{\boldsymbol{\mathsf{S}}}^\perp\times\Wtilde$.

Now, let us prove the claim. First, we show that $\Vtilde\subset\underline{\boldsymbol{\mathsf{S}}}^\perp$, that is, that
$\und{\epsilon}(\W)\subset\underline{\boldsymbol{\mathsf{S}}}^\perp$ and $\VV_{\mathrm{sbb}} \subset\underline{\boldsymbol{\mathsf{S}}}^\perp$. The second inclusion follows by the definition of $\underline{\boldsymbol{\mathsf{S}}}$. Let us prove the first.
Take any $\w\in \W$ and any $\underline{\boldsymbol{\mathsf{s}}}\in\underline{\boldsymbol{\mathsf{S}}}$. By construction, there is $\vvtilde^\perp\in \Vperp_1$
such that $\underline{\boldsymbol{\mathsf{s}}}=\vvtilde^\perp-\Pi_{\mathrm{sbb}}\vvtilde^\perp$ and so,
\[
(\und{\epsilon}( \w), \underline{\boldsymbol{\mathsf{s}}})_K\
=(\und{\epsilon}( \w), \vvtilde^\perp-\Pi_{\mathrm{sbb}}\vvtilde^\perp)_K
=(\und{\epsilon}( \w), \vvtilde^\perp)_K
=0,
\]
because $\Vtilde_1\supset \und{\epsilon}(\W).$ This implies that $\und{\epsilon}(\W)\subset \underline{\boldsymbol{\mathsf{S}}}^\perp$ and hence that $\Vtilde\subset\underline{\boldsymbol{\mathsf{S}}}^\perp$.

It remains to show that $\Vtilde^\perp\cap \underline{\boldsymbol{\mathsf{S}}}^\perp=\{0\}$, which proves the reverse inclusion. Let then 
$\underline{\boldsymbol{\mathsf{s}}}^\perp\in \underline{\boldsymbol{\mathsf{S}}}^\perp$ satisfy
\begin{equation}\label{eq:1.4}
(\underline{\boldsymbol{\mathsf{s}}}^\perp,\und{\epsilon}( \w)+\vv_{\mathrm{sbb}})_K=0\qquad \forall \w\in \W, \;\forall\vv_{\mathrm{sbb}}\in \VV_{\mathrm{sbb}}.
\end{equation}
Then, for all $\w^\perp\in \Wperp=(\divv\VV)^\bot$, we have that
\[
0 =(\underline{\boldsymbol{\mathsf{s}}}^\perp, \und{\epsilon}(\w^\perp))_K=-(\divv\underline{\boldsymbol{\mathsf{s}}}^\perp,\w^\perp)_K+\langle \underline{\boldsymbol{\mathsf{s}}}^\perp\nn,\w^\perp\rangle_{\partial K}
=
\langle \underline{\boldsymbol{\mathsf{s}}}^\perp\nn,\w^\perp\rangle_{\partial K}.
\]
Since $\VV\times \W$ admits an \mm-decomposition with the associated subspace 
$\underline{\boldsymbol{\mathsf{S}}}^\perp\times\Wtilde$, we have that
$\bld M=\gamma\underline{\boldsymbol{\mathsf{S}}}{\oplus} \gamma\Wperp$ with orthogonal sum.
 Hence,  there exists $\underline{\boldsymbol{\mathsf{s}}}\in \underline{\boldsymbol{\mathsf{S}}}$ such that $(\underline{\boldsymbol{\mathsf{s}}}+\underline{\boldsymbol{\mathsf{s}}}^\perp)\nn=0$ on $\partial K$.  Using \eqref{eq:1.4} it follows that, for all $\w\in \W$,
\begin{alignat*}{6}
0	&=(\underline{\boldsymbol{\mathsf{s}}}^\perp,\und{\epsilon}( \w))_K\\
        &=(\underline{\boldsymbol{\mathsf{s}}}^\perp,\gradv \w)_K\\
        &=-(\divv\underline{\boldsymbol{\mathsf{s}}}^\perp,\w)-\langle
        \underline{\boldsymbol{\mathsf{s}}}\nn,\w\rangle_{\partial K}\\
	&=-(\divv(\underline{\boldsymbol{\mathsf{s}}}^\perp+\underline{\boldsymbol{\mathsf{s}}}),\w)_K-(\underline{\boldsymbol{\mathsf{s}}},\gradv \w)_K \\
	&=-(\divv(\underline{\boldsymbol{\mathsf{s}}}^\perp+\underline{\boldsymbol{\mathsf{s}}}),\w)_K,
\end{alignat*}
because $\und{\epsilon}( \W)\subset \underline{\boldsymbol{\mathsf{S}}}^\perp$.
Therefore $\divv(\underline{\boldsymbol{\mathsf{s}}}^\perp+\underline{\boldsymbol{\mathsf{s}}})=0$, and thus $\underline{\boldsymbol{\mathsf{s}}}^\perp+\underline{\boldsymbol{\mathsf{s}}}\in \VV_{\mathrm{sbb}}$. Finally, 
by the identity \eqref{eq:1.4} with $\vv_{\mathrm{sbb}}:=\underline{\boldsymbol{\mathsf{s}}}^\perp+\underline{\boldsymbol{\mathsf{s}}}$ and $\w:=\boldsymbol{0}$, we get that
\[
0=(\underline{\boldsymbol{\mathsf{s}}}^\perp,\underline{\boldsymbol{\mathsf{s}}}^\perp+\underline{\boldsymbol{\mathsf{s}}})_K=(\underline{\boldsymbol{\mathsf{s}}}^\perp,\underline{\boldsymbol{\mathsf{s}}}^\perp)_K,
\]
and therefore $\underline{\boldsymbol{\mathsf{s}}}^\perp=\underline{\boldsymbol 0}$. This proves the claim and completes the proof.
\end{proof}

Now, we are ready to prove Theorem \ref{thm:1.5}.
\begin{proof} Since properties (a) and (b) are part of the definition of an mm-decomposition, we can always assume they hold.
Since, by Lemma \ref{th:1.3}, $\VV\times \W$ admits an \mm-decomposition if and only if it admits the canonical decomposition,
we can always take the choice $\Vtilde:=\gradv \W\oplus \VV_{\mathrm{sbb}}$ and $\Wtilde:=\divv \VV$. Finally, 
we have that the trace operator $\tr: \Vtilde^\perp\times \Wtilde^\perp\rightarrow {\bld M}$ is an isomorphism if an only if 
\[
\dim \bld M = \dim \gamma\Vtilde^\perp + \dim \gamma\Wtilde^\perp,
\quad
\dim \gamma \Vtilde^\perp 
=\dim \Vtilde^\perp,
\quad
\dim \gamma\Wtilde^\perp
=\dim \Wtilde^\perp.
\]
Thus, we have to show that the above equalities hold if and only if 
$I_{\bld M}(\VV\times \W)=0$ assuming properties (a) and (b).

Let us show that the last equality always hold. Indeed, if $\widetilde{\w}^\perp\in \Wtilde^\perp$ and is zero on $\partial K$, then, for any 
$\boldsymbol{\vv}\in \VV$, we have
\[
0=\langle \widetilde{\w}^\perp, \boldsymbol{\vv}\cdot\boldsymbol{n}\rangle_{\partial K}
=(\gradv \widetilde{\w}^\perp, \boldsymbol{\vv})_K+(\widetilde{\w}^\perp,\divv\boldsymbol{\vv})_K
=(\und{\epsilon} \widetilde{\w}^\perp, \boldsymbol{\vv})_K
\]
since $\Wtilde=\divv\VV$. Since $\VV\supset\und{\epsilon}(\W)$, we can now take $\boldsymbol{\vv}:= \und{\epsilon}(\widetilde{\w}^\perp)$ and conclude that
$\widetilde{\w}^\perp$ is a constant on $K$. As a consequence $\widetilde{\w}^\perp=\bld 0$, and the last equality follows. 

Let us now show that the second equality also holds. Indeed, if $\widetilde{\boldsymbol{\vv}}^\perp\in \Vtilde^\perp$ and its normal trace is
zero on $\partial K$, then, for any  $\w\in \W$, we have
\[
0=\langle {\w}, \widetilde{\boldsymbol{\vv}}^\perp\cdot\boldsymbol{n}\rangle_{\partial K}
=(\gradv {\w}, \widetilde{\boldsymbol{\vv}}^\perp)_K+({\w},\divv\widetilde{\boldsymbol{\vv}}^\perp)_K
=(\und{\epsilon} ({\w}), \widetilde{\boldsymbol{\vv}}^\perp)_K+({\w},\divv\widetilde{\boldsymbol{\vv}}^\perp)_K
=({\w},\divv\widetilde{\boldsymbol{\vv}}^\perp)_K
\]
since $\Vtilde\supset\und{\epsilon}(\W)$. Since $\W\supset\divv\VV$, we can now take $\w:=\divv \widetilde{\boldsymbol{v}}^\perp$ and conclude that
$\divv \widetilde{\boldsymbol{\vv}}^\perp$ is zero on $K$. As a consequence, $\widetilde{\boldsymbol{\vv}}^\perp\in\VV_{\mathrm{sbb}}\subset\Vtilde$, we have that $\widetilde{\boldsymbol{\vv}}^\perp=\boldsymbol{0}$, and the second equation follows.

Thus, we only have to show that
$
\dim \bld M = \dim \Vtilde^\perp + \dim \Wtilde^\perp,
$
if and only if $I_{\bld M}(\VV\times \W)=0$.  But, we have that
\begin{alignat*}{6}
I:=&\dim \bld M-\dim \Vtilde^\perp - \dim \Wtilde^\perp
\\=&
\dim \bld M-(\dim\VV- \dim \Vtilde) -(\dim \W-\dim \Wtilde)
\\
=&
\dim \bld M-(\dim\VV- \dim \gradv \W- \dim \VV_{\mathrm{sbb}}) -(\dim \W-\dim \divv \VV)
\end{alignat*}
by the definition of $\Vtilde$ and $\Wtilde$. After rearranging terms, we  get that
\begin{alignat*}{1}
I =&
\dim \bld M-(\dim\VV- \dim \divv \VV- \dim \VV_{\mathrm{sbb}}) -(\dim \W-\dim \gradv \W)
\\
=&
\dim \bld M-(\dim\{\vv\in\VV: \divv \vv=0\}- \dim\{\vv\in\VV: \divv \vv=0, \vv\nn|_{\partial K}=0\} )
\\&- \dim\{ \w\in \W: \gradv w=0\}
\\
=&\dim \bld M-\dim\{\vv\nn|_{\partial K}:\, \vv\in \VV, \gradv \cdot\vv=0\}
-\dim \{\w|_{\partial K}:\, \w\in \W, \und{\epsilon} (\w)=0\}
\\=&I_{\bld M}(\VV\times \W).
\end{alignat*}
and the result follows.

Finally, by the inclusion property {\rm (a)}, we have that
\[
\{\vv\nn|_{\partial K}:\, \vv\in \VV, \gradv \cdot\vv=0\}
\oplus \{\w|_{\partial K}:\, \w\in \W, \und{\epsilon}(\w)=0\}\subset {\bld M},
\]
where the sum is $L^2(\partial K)$-orthogonal since
\[
\langle\vv\nn, \w\rangle_{\partial K}=(\divv \vv, \w)_K+(\vv, \gradv \w)_K
=(\divv \vv, \w)_K+(\vv, \und{\epsilon}(\w))_K=0
\]
if $\divv \vv=0$ and $\und{\epsilon} (\w) =0$. Finally,
since $\VV\times \W$ admits and \mm-decomposition, the equality holds. 
This completes the proof.
\end{proof}

\section*{Appendix D: Proof of Theorem \ref{thm:square} and Theorem \ref{thm:square2}}
In this Appendix, we prove Theorem \ref{thm:square} and Theorem \ref{thm:square2} on the explicit \mm-decomposition construction on a unit square
with initial spaces 
\[
 \VVV\times \WWW\times \bld M:=
 \bqqol_k\times \qqol_{k}\times \ppol_k. (k\ge 1)
\]

For the above initial space, it is quite easy to show that
\[
 I_S(\VVV\times\WWW) = 3,
\]
and $\delta\VV_{\mathrm{fillV}}$ in Theorem \ref{thm:square} satisfy the properties in Table \ref{table:deltaspaces}.

Next, we apply Algorithm \ref{algorithm1} and follow the steps in the proof of Theorem \ref{thm:polygon} in Section 5 to show that 
$\delta\VV_{\mathrm{fillM}}$ in Theorem \ref{thm:square} and Theorem \ref{thm:square2} also satisfy properties in Table \ref{table:deltaspaces}.
Since the proof is similar and quite simpler than that for Theorem \ref{thm:polygon} in Section 5, we only sketch the main steps of the proof.

\

{\bf (1) Finding the spaces $\VVV_{s,i}$}. The following result is an analog to Proposition \ref{lemma:characterization-s}.
\begin{lemma}
 \label{lemma-characterization-sQ}
Let $K$ be the unit square. Then, for  $\VVV = \bqqol_k(K)$, we have 
\[
 \VVV_{s,i} = J\,\Phi_i,
\]
where 
\begin{align*}
\Phi_1:=&\;  \pcol_{k,k}{(K)}\oplus \mathrm{span}
\{x^{k+1},x^{k+1}y,x^{k+2}, y^{k+1}, xy^{k+1}, y^{k+2}\},\\
\Phi_2:= &\;x^2\,\pcol_{k-2,k}{(K)}\oplus \mathrm{span}
\{x^{k+1}, x^{k+1}y, x^{k+2} \},\\
\Phi_3:= &\;x^2y^2\,\pcol_{k-2,k-2}{(K)},\\
\Phi_4:= &\;x^2y^2(1-x)^2\,\pcol_{k-4,k-2}{(K)},\\
\Phi_5:= &\;x^2y^2(1-x)^2(1-y)^2\, \pcol_{k-4,k-4}{(K)}.
\end{align*}
Here $\pcol_{k_1,k_2}(K):=\pcol_{k_1}(x)\otimes\pcol_{k_2}(y)$ is the tensor production space
with variable degree.
\end{lemma}

\

{\bf (2) Finding the complement spaces $C_{M,i}$.}
From the above lemma, we can immediately get a characterization of $\gamma_i(\VVV_{s,i})$ and 
compute the \mm-indexes.
\begin{corollary}
 \label{corollary-dim-count-Q}
 We have
 \begin{alignat*}{2}
  \gamma_1( \VVV_{s,1} )=\;&
\mathrm{span}\{ \gamma_1(J\,y^a)\}_{a=2}^{k+2}
\oplus
  \mathrm{span}\{ \gamma_1(J\,xy^a)\}_{a=1}^{k+1},\\
  \gamma_2( \VVV_{s,2} )=\;&
\mathrm{span}\{ \gamma_2(J\,x^2(1-x)^a)\}_{a=0}^{k}
\oplus
  \mathrm{span}\{ \gamma_2(J\,x^2y(1-x)^a)\}_{a=0}^{k-1},\\
  \gamma_3( \VVV_{s,3} )=\;&
  \left\{
  \begin{tabular}{l l}
   $\mathrm{span}\{ \gamma_3(J\,x^2y^2(1-y)^a)\}_{a=0}^{k-2}
$ & if $k \le 2$, \vspace{.2cm}\\
   $\mathrm{span}\{ \gamma_3(J\,x^2y^2(1-y)^a)\}_{a=0}^{k-2}$\\
$\hspace{.2cm}\oplus
  \mathrm{span}\{ \gamma_3(J\,x^2y^2(1-x)(1-y)^a)\}_{a=0}^{k-2}
$ & if $k \ge 3$,
  \end{tabular}\right.\\
  \gamma_4( \VVV_{s,4} )=\;&
\mathrm{span}\{ \gamma_4(J\,x^2y^2(1-x)^a)\}_{a=2}^{k-2}
\oplus
  \mathrm{span}\{ \gamma_4(J\,x^2y^2(1-x)^a(1-y))\}_{a=2}^{k-2}.
  \end{alignat*}
  Hence, 
  \begin{align*}
I_{M,1}(\VVV\times \WWW) = &\; 0,\\
I_{M,2}(\VVV\times \WWW) = &\; 1\\
I_{M,3}(\VVV\times \WWW) = &\; \left\{
\begin{tabular}{l l}
 $4$ & if $k = 1$,\\
 $5$ & if $k = 2$,\\
 $4$ & if $k \ge 3$,
\end{tabular}\right.
\\
I_{M,4}(\VVV\times \WWW) = &\; \left\{
\begin{tabular}{l l}
 $1$ & if $k = 1$,\\
 $3$ & if $k = 2$,\\
 $5$ & if $k \ge 3$.
\end{tabular}\right.
\end{align*}
\end{corollary}

Now, it is easy to show the following $C_{M,i}$ satisfy 
$\gamma_i(\VVV_{s,i})\oplus C_{M,i}=\widetilde{\bld M}(\eg_i)$ :
 \[
  C_{M,i} =\left\{\begin{tabular}{l l}
            $\{0\}$ & if $i = 1$,           
            \vspace{.3cm}
\\
$\mathrm{span}
\left  \{ \gamma_2\left(J\,\eta_2y(1-x)\right)\right\}
$& if $i =  2$,
            \vspace{.3cm}
            \\
            $\mathrm{span}
\left  \{ \gamma_3\left(J\,\eta_3^2(1-y)^{b}\right)\right\}_{b=k-1}^{k}$\\
$ \hspace{.2cm}\oplus 
\mathrm{span}
\left  \{ \gamma_3\left(J\,\eta_3^2(1-x)(1-y)^{b}\right)\right\}_{b=k-1-\delta_{2,k}}^{k-1}$\\
$ \hspace{.4cm}\oplus \mathrm{span}
\left \{ 
\gamma_3\left(J\,\eta_3(1-x)(1-y)\right)\right\}$
            & if $i = 3$ ,
            \vspace{.3cm}\\
            $\mathrm{span}
\left  \{ \gamma_4\left(J\,\eta_1(1-x)(1-y)\right)\right\}
$
&if $i = 4$ and $k = 1$,\vspace{.3cm}\\
$\mathrm{span}\left\{\gamma_4\left(J\,\eta_1^2(1-x)^2\right)\right\}$\\
            $\hspace{.2cm}\oplus\mathrm{span}
\left  \{ \gamma_4\left(J\,\eta_1(1-x)(1-y)\right), 
\gamma_4\left(J\,\eta_1(1-x)^2(1-y)\right)\right\}
$
&if $i = 4$ and $k = 2$,\vspace{.3cm}.\\
$\mathrm{span}\left\{\gamma_4\left(J\,\eta_1^2(1-x)^{b}\right)
\right\}_{b=k-1}^{k}$\\
$\hspace{.2cm}\oplus\mathrm{span}\left\{\gamma_4\left(J\,\eta_1^2
(1-x)^{k-1}(1-y)\right)\right\}$\\
$\hspace{.4cm}\oplus\mathrm{span}
\left  \{ \gamma_4\left(J\,\eta_1(1-x)(1-y)\right), 
\gamma_4\left(J\,\eta_1(1-x)^2(1-y)\right)\right\}
$
&if $i = 4$ and $k \ge 3$,\vspace{.3cm}.
            \end{tabular}\right.
 \]
 Here $\eta_i$ is any linear function on $\mathbb{R}^2$ 
 such that  $\eta_i(\vt_{i}) = 0$ and $\eta_i(\vt_{i+1}) \not = 0$.

\

{\bf (3) Finding the spaces $\delta\VVV_{\mathrm{fillM}}^i$}.
Finally, it is easy to verify that the divergence-free spaces $\delta\VVV_{\mathrm{fillM}}$
defined in Theorem \ref{thm:square} and Theorem \ref{thm:square2} satisfy 
the trace properties (3.1-3.3) in Algorithm \ref{algorithm1} and has dimension $I_{M,i}(\VVV\times \WWW)$.
This completes the proof.\qed

\section*{Funding}

The first author is supported in part by the National Science Foundation
               (Grant DMS-1522657) and by the University of Minnesota
               Supercomputing Institute.

\end{document}